\documentclass[11pt]{article}
\usepackage[a4paper,margin=25mm]{geometry}
\usepackage{amsmath,amssymb,amsthm,mathrsfs,mathtools,booktabs,array,hyperref,microtype}
\hypersetup{colorlinks=true,linkcolor=blue,citecolor=blue,urlcolor=blue}
\newtheorem{theorem}{Theorem}[section]
\newtheorem{lemma}[theorem]{Lemma}
\newtheorem{proposition}[theorem]{Proposition}
\newtheorem{corollary}[theorem]{Corollary}
\theoremstyle{definition}\newtheorem{definition}[theorem]{Definition}
\newtheorem{assumption}[theorem]{Assumption}
\theoremstyle{remark}\newtheorem{remark}[theorem]{Remark}
\newcommand{\ii}{\mathrm i}\newcommand{\ee}{\mathrm e}
\newcommand{\Id}{\mathrm{Id}}\newcommand{\ad}{\operatorname{ad}}
\newcommand{\ord}{\operatorname{ord}}\newcommand{\supp}{\operatorname{supp}}
\newcommand{\NN}{\mathbb N}\newcommand{\ZZ}{\mathbb Z}\newcommand{\RR}{\mathbb R}

\newcommand{\nrm}[1]{\left\|#1\right\|}
\newcommand{\coef}[1]{\left\|#1\right\|^{\#}}
\newcommand{\vf}[1]{\left|#1\right|^{\mathrm{vf}}}
\newcommand{\br}[1]{\langle#1\rangle}

\allowdisplaybreaks
\title{Effective stability and sharp logarithmic time scales\protect\\for Hamiltonian PDEs on the circle}
\author{Bingqi Yu$^{1}$ \and Yong Li$^{1,2}$\thanks{Corresponding author.
E-mail addresses: \href{mailto:yubq23@mails.jlu.edu.cn}{yubq23@mails.jlu.edu.cn}
(B. Yu); \href{mailto:liyong@jlu.edu.cn}{liyong@jlu.edu.cn} (Y. Li).}}
\date{}
\begin{document}\maketitle
\begin{center}\small
$^{1}$School of Mathematics, Jilin University,\\
Changchun 130012, People's Republic of China.\\[3pt]
$^{2}$Center for Mathematics and Interdisciplinary Sciences,\\
Northeast Normal University, Changchun 130024, People's Republic of China.
\end{center}
\begin{abstract}
We prove effective stability and matching logarithmic upper bounds for small solutions of Hamiltonian PDEs on the one-dimensional torus. An analytic coefficient-majorant calculus and an explicit normal-form iteration yield stability times with logarithms of order $L^{1+1/p}/(\log L)^{1/p}$ for Gevrey weights and $L^{1+(q-1)/(pq+1)}$ for logarithmic ultra-differentiable weights, where $L=\log(1/\varepsilon)$ and $p\geq1$ measures the order dependence of the small divisors. Even frequencies are organized into the exact blocks $\{n,-n\}$; elementary frequency gaps and constant weights on these blocks close the energy estimate. We verify the hypotheses for convolution NLS, fractional NLS with every fixed noninteger exponent $\eta>1/2$, and a beam equation with a metric parameter and a variational gradient nonlinearity. Two fixed analytic nonlocal Hamiltonian PDEs admit finite-support solutions whose first prescribed norm amplification occurs on the same logarithmic scales along sequences $\varepsilon_n\to0$. The common lower theorem applies to every sufficiently small datum of each admissible equation. Its combination with these examples establishes sharpness of both regularity-dependent scales under the same analytic and arithmetic hypotheses.
\end{abstract}
\tableofcontents

\section{Introduction}\label{sec:joint}
\subsection{Effective stability and the question of optimal time scales}
How long does a Hamiltonian system remain close to its initial state?
For a nearly integrable system, the conservation of energy alone usually
does not answer this question: energy may be redistributed among its
degrees of freedom while the total energy stays fixed. Nekhoroshev
theory gives a quantitative answer. If the integrable Hamiltonian
$h(I)$ is steep, an analytic perturbation $h(I)+\epsilon f(I,\varphi)$
changes the actions by at most $O(\epsilon^b)$ over a time
$\exp(c\epsilon^{-a})$, with positive exponents determined by the
geometry and dimension of the system \cite{N77,P93}. Such a conclusion
expresses effective stability even when stability for all time is
unavailable.

The regularity of the Hamiltonian affects the time that can be obtained.
Marco and Sauzin \cite{MS03} established stability and instability
results for Gevrey quasi-convex Hamiltonians. Bounemoura proved
polynomial stability estimates in finite differentiability
\cite{B10} and treated Gevrey and finitely differentiable prevalent
Hamiltonians in \cite{Bou11}. Bounemoura and F\'ejoz \cite{BF21}
developed perturbation theory in ultra-differentiable classes, relating
the normal-form estimates and the resulting stability times to the
sequence controlling the derivatives. These results show why the
regularity dependence of effective stability should be studied together
with the possibility of instability on a comparable time scale.

For Hamiltonian PDEs, infinitely many Fourier modes introduce a further
difficulty. As the normalization order increases, the homological
equations involve longer combinations of frequencies and arbitrarily
large indices. A finite-dimensional estimate cannot simply be passed to
the limit, since its constants and exponents may deteriorate with the
number of modes. Moreover, conservation of a low-order energy does not
exclude transfer to high frequencies, where a stronger norm assigns a
larger weight. The relation between long-time stability and this transfer
has been a persistent theme in the subject. Bourgain's
\emph{Problems in Hamiltonian PDE's} \cite{BProb00} placed the long-time
dynamics of Hamiltonian equations within a broad programme of open
problems; his work on stability and diffusion
\cite{B00,B04} identified quantitative questions that remain central
to effective stability.

One particularly relevant prediction appears in
\cite[Section 2.1, equation (2.14)]{B04}. For exponentially decaying
Fourier data and independently variable frequencies, Bourgain's
normal-form analysis leads to the following time scale, with
$\varepsilon$ denoting the perturbation parameter in that argument:
\begin{equation}\label{eq:Bourgain-benchmark}
 \exp\!\left(c\frac{L^2}{\log L}\right),
 \qquad L=\log(1/\varepsilon).
\end{equation}
The proposed scale reflects the competition between the order of
normalization and the number of active Fourier modes. It also gives a
concrete form to the optimality question: can all small solutions remain
stable substantially longer under the same arithmetic and regularity
conditions? The role of \eqref{eq:Bourgain-benchmark} as an expected
optimal scale is discussed explicitly in \cite[Section 1.1, p.~2193]{BG25}.

Several approaches have made effective stability accessible in infinite
dimension. Early results include \cite{B88,B99}; the Birkhoff
normal-form method gives polynomial long-time control for broad classes
of dispersive equations \cite{BG06,BDGS07,BFG20}, including equations
on general tori and on manifolds with integrable geodesic flow
\cite{BFM24,BFLM25}. To go beyond a fixed polynomial time, one must
estimate the dependence on the normalization order. Quantitative
versions of this argument yield exponential or sub-exponential
stability bounds for NLS \cite{FG13,BMP19,BMP20}. Effective stability
near finite-gap tori has also been proved for perturbations of the
Benjamin--Ono equation \cite{BG24}.

An external parameter can supply the frequency separation needed in
these arguments. In his construction of full-dimensional invariant tori
for one-dimensional NLS, Bourgain \cite{B05} used a convolution
potential, or diagonal Fourier multiplier, to vary the linear
frequencies. This permits one to impose nonresonance conditions by
excluding a small set of potentials. The structure of the parameter
family matters as well as the decay of the initial data. For example,
Bernier and Gr\'ebert \cite{BG25} obtained the scale
\eqref{eq:Bourgain-benchmark} for small Sobolev data by using bounded
Fourier multipliers that are constant on dyadic blocks and chosen by a
blockwise random law.

A complementary line of work treats equations without external
parameters. The nonlinear dependence of the frequencies on the actions
then provides an internal modulation. Bernier, Faou and Gr\'ebert
\cite{BFG20RN} introduced rational normal forms for this purpose in
NLS-type equations, and Bernier and Gr\'ebert \cite{BG21} developed the
method for generalized Korteweg--de Vries and Benjamin--Ono equations.
For suitable sets of small initial data, these constructions give
stability over arbitrary prescribed polynomial times. Another approach,
originating in Bourgain's work \cite{B00}, recentres the actions at
their initial values and counts the paired factors separately from
the remaining Fourier factors. Bernier and Camps \cite{BC25} used
this approach for cubic NLS on non-rectangular flat tori, with
quantitative control of the dependence on the initial actions.
For the Schr\"odinger--Poisson equation, Bernier, Camps, Gr\'ebert
and Wang \cite{BCGW24} proved a Gevrey stability lower bound on the
scale \eqref{eq:Bourgain-benchmark} for asymptotically full sets of
small initial data. These results substantially extend the range of
equations for which effective stability is known.

The stability theorems just described give a time during which the
relevant solutions are guaranteed to remain small. Determining the
optimal time requires a complementary construction: a solution must
undergo a prescribed increase of the same norm on a comparable scale,
and its equation must obey the hypotheses used for the lower bound.
Energy-transfer mechanisms in Hamiltonian PDEs and infinite lattices
have been constructed in \cite{CKSTT10,GGM21,BGGH25,GG24}. To address
the present optimality question, we combine an order-dependent stability
theorem with two analytic nonlocal Hamiltonian equations on the circle.
One construction treats Gevrey weights and the other logarithmic
ultra-differentiable weights. Each equation is fixed before the sequence
of small initial data is chosen. Its first prescribed norm amplification
occurs on the logarithmic scale supplied by the lower theorem.

\subsection{Main results}

We work on the circle $\mathbb T=\mathbb R/(2\pi\mathbb Z)$.
For a fixed regularity weight, we compare the time for which all small
solutions remain within a prescribed multiple of their initial size
with the first time at which particular solutions reach that multiple.
The parameter $p\geq1$ below records the order dependence of the
small-divisor bound; the logarithmic weight has exponent $q>1$.
In each application, we choose an admissible exponent $p$ for
which the divisor estimate \eqref{eq:nr} holds.

For $L>\ee$ put
\begin{equation}\label{eq:scales}
 \Phi_{G,p}(L)=\frac{L^{1+1/p}}{(\log L)^{1/p}},\qquad
 \Phi_{U,p,q}(L)=L^{1+(q-1)/(pq+1)}.
\end{equation}
For a fixed $A>1$ define the forward effective escape time
\begin{equation}\label{eq:escape}
 T_A^+(u_0;H)=\inf\{t>0:\nrm{u(t)}_s=A\nrm{u_0}_s\},\qquad\inf\varnothing=+\infty.
\end{equation}
The infimum is taken along the solution on its maximal existence interval. The lower bounds include existence throughout the indicated interval.
This fixed-amplification criterion is also used by
Farr\'e--Fayad \cite[Section 4.2, Remark 2]{FF22}, with amplification
factor two.

\begin{theorem}[Lower and upper logarithmic scales]\label{thm:joint}
Let $f$ be the Gevrey or logarithmic weight of Section~\ref{sec:setting}. Assume the coefficient-majorant condition of Definition~\ref{def:coef} and the one-dimensional frequency conditions of Assumption~\ref{ass:freq}. Fix $s>s_0$ and $A>1$. There are $c_A,\varepsilon_A>0$ such that every initial datum with $0<\varepsilon=\nrm{u_0}_s<\varepsilon_A$ has a unique solution throughout
\[
 |t|\leq\exp(c_A\Phi(L)),\qquad L=\log(1/\varepsilon),
\]
and this solution satisfies $\nrm{u(t)}_s<A\varepsilon$. Here $\Phi$ is the corresponding function in \eqref{eq:scales}, and the norm is the physical Fourier norm \eqref{eq:physicalnorm}.

Conversely, for every fixed $p\geq1$, $A>1$ and either of these weights, there exist a real momentum-preserving Hamiltonian PDE on $\mathbb T$ satisfying the same conditions and a sequence of finite-support global solutions with initial sizes $\varepsilon_n\to0$, for which
\begin{equation}\label{eq:upperrough}
 T_A^+(u_0^{(n)};H)\leq \exp(C_A\Phi(L_n)),\qquad L_n=\log(1/\varepsilon_n).
\end{equation}
More precisely, the constants and asymptotics are given in Theorems~\ref{thm:gevreyupper} and~\ref{thm:logupper}. For each fixed weight, $p$ and $A$, this Hamiltonian is chosen once before the sequence of initial data.
\end{theorem}

\begin{definition}[Sharpness of a logarithmic stability scale]\label{rem:scope}
Fix $p$, the weight $f$, its width $s$, and $A>1$. A scale $\Phi$
is sharp for the stated Hamiltonian class if every fixed equation
$H$ in the class satisfies
$T_A^+(u_0;H)\geq\exp(c_H\Phi(L))$ for all sufficiently small data,
and there exist a fixed equation $H_*$ in the class and a sequence
of data with $L_n\to\infty$ such that
$T_A^+(u_0^{(n)};H_*)\leq\exp(C_{H_*}\Phi(L_n))$.
The constants $c_H,C_{H_*}>0$ may depend on the equations and the
fixed parameters. The equation $H_*$ is chosen after fixing
$p,f,s,A$ and before choosing the sequence of initial data.
\end{definition}

The lower bound has concrete consequences for three equations on the
circle:
\begin{equation}\label{eq:intro-models}
 \begin{aligned}
  \ii\psi_t&=-\psi_{xx}+V*\psi+a(|\psi|^2)\psi,\\
  \ii\psi_t&=(-\partial_x^2+m)^\eta\psi+a(|\psi|^2)\psi,\\
  \psi_{tt}+g^2\psi_{xxxx}+m\psi
     &=-F_\psi(\psi,\psi_x)+\partial_xF_v(\psi,\psi_x).
 \end{aligned}
\end{equation}
Here $a$ is real analytic with $a(0)=0$, and $F$ is real analytic
with Taylor expansion starting at degree three. For fractional NLS
we fix a noninteger $\eta>1/2$; for the beam equation we fix $m>0$
and vary the scalar metric $g>0$. The convolution potentials are
real and even, with the product probability law specified in
Section~\ref{sec:apps}.

For each fixed nonlinearity, almost every parameter in the indicated
family gives the following stability times, where
$L=\log(1/\varepsilon)$. All regularity parameters are fixed as
$\varepsilon\to0$. The constants are chosen after fixing the
potential, the noninteger exponent $\eta$ and the mass or metric in
a compact positive parameter interval:
\begin{center}\small
\renewcommand{\arraystretch}{1.3}
\begin{tabular}{@{}l c c c@{}}
\toprule
Equation and parameter & $p$ & Gevrey time & Logarithmic time\\
\midrule
Convolution NLS, $V$ & $1$
 & $\exp(c_A L^2/\log L)$ & $\exp(c_A L^{2q/(q+1)})$\\
Fractional NLS, $m>0$ & $2$
 & $\exp(c_A L^{3/2}/\sqrt{\log L})$
 & $\exp(c_A L^{3q/(2q+1)})$\\
Beam equation, $g>0$ & $2$
 & $\exp(c_A L^{3/2}/\sqrt{\log L})$
 & $\exp(c_A L^{3q/(2q+1)})$\\
\bottomrule
\end{tabular}
\end{center}
For either weight and every fixed $A>1$, there are
$c_A,\varepsilon_A>0$ such that
\[
 0<\varepsilon=\|u(0)\|_s<\varepsilon_A
 \quad\Longrightarrow\quad
 \|u(t)\|_s<A\varepsilon\qquad(|t|\leq T),
\]
with $T$ given in the table. At $q=2$, the logarithmic times become
$\exp(c_A L^{4/3})$ for convolution NLS and
$\exp(c_A L^{6/5})$ for fractional NLS and beam.
For the two NLS equations $u=\psi$;
for the beam equation we use
$u=2^{-1/2}(A_g^{1/4}\psi+\ii A_g^{-1/4}\psi_t)$,
$A_g=g^2\partial_x^4+m$, so that
$\|u\|_s^2=(\|A_g^{1/4}\psi\|_s^2+
\|A_g^{-1/4}\psi_t\|_s^2)/2$.
The Gevrey weight is $f(x)=x^\theta$, $0<\theta<1$; the logarithmic
weight is \eqref{eq:logweight}, with $q>1$.
Theorems~\ref{ex1thm1}--\ref{ex3thm2} give the six precise
statements; the preceding calculations in Section~\ref{sec:apps}
verify their hypotheses. The two nonlocal equations constructed
later provide the complementary upper bounds in
Theorem~\ref{thm:joint}.

\subsection{Contributions and the main steps of the proofs}
Theorem~\ref{thm:joint} has two complementary contributions.
First, one quantitative normal-form theorem gives the stated
stability times for every sufficiently small initial datum in a
broad class of Hamiltonian PDEs, with the spatial regularity fixed.
The three applications illustrate how independent Fourier parameters
and scalar physical parameters enter the same theorem. Second, we
construct fixed admissible equations whose solutions reach the
prescribed amplification on matching scales. To our knowledge,
these are the first examples establishing sharpness for both of
these fixed Fourier regularities under the same order-dependent
arithmetic assumptions as a general lower-bound theorem.
Together they determine the optimal logarithmic stability scales
for the class specified in Definitions~\ref{def:coef}
and~\ref{rem:scope} and Assumption~\ref{ass:freq}. In the case $p=1$, the two scales are
\[
 \Phi_{G,1}(L)=\frac{L^2}{\log L},\qquad
 \Phi_{U,1,q}(L)=L^{2q/(q+1)}.
\]
For every fixed $0<\theta<1$, the Gevrey theorem gives Bourgain's
dependence on $L$ in \eqref{eq:Bourgain-benchmark} for the weights
$\ee^{s\br j^\theta}$. The proof uses their strict subadditivity
constant $C_f=2^\theta-1<1$. The logarithmic theorem supplies the
corresponding scale at weaker Fourier regularity.
The lower bound for every sufficiently small datum and the upper
examples determine these scales up to positive multiplicative
constants in $\log T$. For each constructed equation, we also compute
the leading constant in its first-amplification time:
$\log T_n/\Phi(L_n)\to C_G$ in Theorem~\ref{thm:gevreyupper}
and $\log T_n/\Phi(L_n)\to C_U$ in Theorem~\ref{thm:logupper}.
The transfers producing these limits describe the dynamics that
limit the general stability estimate.

Biasco--Massetti--Procesi
\cite[Theorem 1.1 and Remark 1.2]{BMP20} proved stability in a fixed
weighted Hilbert norm, with a displayed logarithmic time of order
$L^{1+\theta/4}$. Bernier--Gr\'ebert
\cite[Theorem 1.1]{BG25} reach $L^2/(4\log L)$ for all small
low-Sobolev data with a special dyadic-block multiplier.
Bernier--Camps--Gr\'ebert--Wang \cite[Theorem 1.1]{BCGW24}
reach $cL^2/\log L$ for Schr\"odinger--Poisson on asymptotically
large sets of data in a fixed weighted $\ell^1$ Gevrey space.
Subsection~\ref{subsec:comparison} records the norms, parameter laws
and initial-data sets associated with these time scales.

We first prove the lower bound by retaining the dependence on the
normalization order throughout the construction. The main issue is to
control two different errors at the same time. Terms of high polynomial
degree become small when restricted to the physical ball, whereas terms
containing at least three high Fourier factors become small through the
spatial weight. We carry the latter terms through the canonical maps
using the vector-field transport identity. This keeps their original
gain even though composition need not preserve the number of high
factors. The resulting bounds allow the degree and the Fourier cutoff
to be chosen together as the initial norm tends to zero.

Resonance requires a separate treatment. On the circle, even frequencies
have the exact blocks $\{n,-n\}$. We retain the two-high terms belonging
to a common block and use the equality of the two Fourier weights to
derive their contribution to the energy identity. The remaining
two-high terms are eliminated using momentum conservation, a buffer
cutoff and the gaps between distinct absolute indices. The estimates
retain the order dependence of the small divisors, represented by
$p$ in \eqref{eq:scales}. The applications verify both the nonlinear
coefficient estimates and the required arithmetic for convolution NLS,
fractional NLS and a beam equation.

For the upper bounds, we construct infinitely many disjoint finite
modules inside one equation. Their frequencies must satisfy prescribed
small-divisor relations without creating additional relations that
violate the lower theorem. We achieve this by controlling all integer
combinations of the implanted detunings. Within each module, a quartic
action term cancels its detuning along a selected invariant phase. We
then derive an increasing scalar action transfer and reconstruct all
coordinate phases. The modules are exactly invariant, so the first
norm crossing and its time follow from explicit integrals. In the
logarithmic construction the small signal amplitudes contribute to the
leading time asymptotic alongside the degree and the frequency size;
keeping this contribution yields the matching exponent.

\subsection{Organization of the paper}
Section~\ref{sec:setting} introduces the function spaces, Hamiltonian
conventions and analytic norms. Section~\ref{sec:res} classifies the monomials by
their high factors and proves the divisor estimates.
Section~\ref{sec:iteration} solves the homological equation and gives
the complete inductive construction. In Section~\ref{sec:normal} we
pass to the physical ball and align the smallness conditions, including
the explicit degree--cutoff inversion. Section~\ref{sec:stability}
proves conjugacy on the generator domain, derives the full-energy
estimate and establishes the stability theorem.
Section~\ref{sec:apps} verifies the three PDE applications.

Section~\ref{sec:implant} constructs the common spectrum and proves
the analytic convergence and invariance of the modules.
Sections~\ref{sec:Gupper} and~\ref{sec:Uupper} each contain one
complete example, from its integer geometry and coordinate equations
to phase locking, the first norm crossing and the time asymptotic.
Section~\ref{sec:upper-comparison} realizes both examples as fixed
nonlocal PDEs and compares their upper bounds with the stability theorem.
Appendix~\ref{app:bookkeeping} records the choice of constants,
Appendix~\ref{app:technical} contains the weight, determinant,
sublevel and inversion lemmas, and Appendix~\ref{app:calculus}
proves the analytic estimates stated in Section~\ref{sec:setting}.
Appendix~\ref{app:supplementary} contains the finite integral
normalization identities, the separate low- and high-mode proofs,
and the extension to fixed polynomial factors and equivalent radial weights.

\section{Setting}\label{sec:setting}
All Fourier indices belong to $\mathbb Z$. The small parameter is the
initial norm $\varepsilon$, and the perturbation is analytic in the
phase variables on a fixed ball. Its Taylor expansion starts at degree
three, so its vector field is $O(\varepsilon^2)$ on a ball of radius
$\varepsilon$. We prescribe spatial regularity separately through the
Fourier weight. The integer $d$ denotes the normalization degree.
Unless otherwise indicated, constants depend only on the fixed weight,
frequency and analytic data, and are independent of $d$, the cutoff
$N$, the construction label $n$ and the initial radius.

\subsection{Fourier coordinates and weights}
For a periodic function write
\[
 u(x)=\sum_{j\in\ZZ}u_j\ee^{\ii jx},\qquad
 u_j=\frac1{2\pi}\int_{\mathbb T}u(x)\ee^{-\ii jx}\,dx,
 \qquad \br j=\max(c_0,|j|),\quad c_0\geq2.
\]
The physical Fourier norm is
\begin{equation}\label{eq:physicalnorm}
 \nrm u_s^2=\sum_{j\in\ZZ}\ee^{2sf(\br j)}|u_j|^2.
\end{equation}
For a convolution kernel we use
\[
 (V*u)(x)=\frac1{2\pi}\int_{\mathbb T}V(x-y)u(y)\,dy,
 \qquad (V*u)_j=V_ju_j.
\]
The same notation denotes the diagonal multiplier when only its
Fourier coefficients are specified.
The complexified phase space has index set $\mathcal Z=\ZZ\times\{+1,-1\}$, $J=(j,\sigma)$ and $\bar J=(j,-\sigma)$, with norm
\begin{equation}\label{eq:complexnorm}
 \nrm z_s^2=\frac12\sum_{J\in\mathcal Z}\ee^{2sf(\br j)}|z_J|^2.
\end{equation}
On the real slice $z_{(j,+)}=u_j$, $z_{(j,-)}=\bar u_j$, this equals \eqref{eq:physicalnorm} exactly. We use $u$ also for a doubled vector when working with analytic Hamiltonians,
and put $B_s(r)=\{u:\nrm u_s<r\}$. The coordinate action is
$I_j=u_{(j,+)}u_{(j,-)}$, equal to $|u_j|^2$ on the real slice.
All dynamical norms and amplification times below are evaluated on that slice.

\begin{assumption}[Weight]\label{ass:weight}
The function $f:[c_0,\infty)\to(0,\infty)$ is increasing and unbounded, and there is $0<C_f<1$ such that, if $x_* =\max(x_1,\ldots,x_r)$ and $x_i\geq c_0$, then
\begin{equation}\label{eq:strictsub}
 f(x_1+\cdots+x_r)\leq f(x_*)+C_f\sum_{i\ne *}f(x_i).
\end{equation}
Fix $s_0>0$ so that
\begin{equation}\label{eq:sumweight}
 S_0^2:=2\sum_{J\in\mathcal Z}\ee^{-2(1-C_f)s_0 f(\br j)}<\infty.
\end{equation}
Set $b=\max(2,2S_0)$ and, for $s>s_0$, $\nu=(1-C_f)(s-s_0)>0$.
\end{assumption}
For Gevrey weights take $f(x)=x^\theta$, $0<\theta<1$, with
$C_f=2^\theta-1$. For logarithmic weights we use
\begin{equation}\label{eq:logweight}
 f(x)=[\log(x+\kappa_*)]^q,\qquad q>1,\quad \kappa_*>\ee^q.
\end{equation}
For this weight $C_f=q/\log\kappa_*$.
Lemma~\ref{lem:weights-full} proves the subadditivity inequality
and summability for both choices and every $s_0>0$.

For an ordered tuple $\mathcal J=(J_1,\ldots,J_m)$ set $\mathcal M(\mathcal J)=\sum_i\sigma_i j_i$ and $\mathcal I_m=\{\mathcal J:\mathcal M(\mathcal J)=0\}$. Each homogeneous Hamiltonian is written with symmetric ordered coefficients,
\[
 F_m=\sum_{\mathcal J\in\mathcal I_m}F_{\mathcal J}u_{J_1}\cdots u_{J_m},\qquad c_m(F)=\sup_{\mathcal I_m}|F_{\mathcal J}|.
\]
Reality means $\overline{F_{\mathcal J}}=F_{\bar{\mathcal J}}$. An ordinary collected monomial is divided by the number of its distinct permutations to obtain these coefficients.

\subsection{Hilbert spaces, real structure and Hamiltonian conventions}
Write $w_s(j)=\exp(sf(\br j))$. The space used below is
\[
 W_s=\left\{(u_j)_{j\in\ZZ}:\sum_jw_s(j)^2|u_j|^2<\infty\right\},
 \qquad (u,v)_s=\sum_jw_s(j)^2u_j\bar v_j.
\]
Multiplication by $w_s$ identifies $W_s$ isometrically with
$\ell^2(\mathbb Z)$, hence $W_s$ is complete. Density and the
embedding for $s'>s$ follow from
\[
 \|u-\Pi_{\leq N}u\|_s^2
       =\sum_{|j|>N}w_s(j)^2|u_j|^2\longrightarrow0,
 \qquad
 \|u\|_s^2
       =\sum_j e^{-2(s'-s)f(\br j)}w_{s'}(j)^2|u_j|^2
       \leq\|u\|_{s'}^2.
\] With
$\Pi_{\leq N}u=(\mathbf1_{|j|\leq N}u_j)_j$ and
$\Pi_{>N}=\Id-\Pi_{\leq N}$, both projections are orthogonal and
\begin{equation}\label{eq:space-tail}
 \|\Pi_{>N}u\|_s^2
 \leq\ee^{-2(s'-s)f(N)}\|u\|_{s'}^2.
\end{equation}
Indeed each term in the left sum is the corresponding $s'$-weighted
term multiplied by $\exp(-2(s'-s)f(\br j))$, bounded above by the
displayed factor. This also proves compactness of the inclusion by
approximation with the finite-rank projections $\Pi_{\leq N}$.

For either weight in this paper, $f(x)/\log x\to\infty$. Thus for
every $a\geq0$ there is a finite $C_{a,s}$ with
$\br j^a\leq C_{a,s}w_s(j)$. The Fourier series of a vector in
$W_s$ consequently defines a smooth periodic function. For example,
\[
 \sum_j|j|^a|u_j|
 \leq\left(\sum_j\br j^{2a}\ee^{-2sf(\br j)}\right)^{1/2}\|u\|_s.
\]
The series in parentheses converges by comparison with
$\sum_{j\ne0}|j|^{-2}$ after increasing the finite starting index.
Uniform convergence of each differentiated Fourier series proves the
asserted smoothness. We denote these particular Hilbert spaces by
$W^G_{s,\theta}$ and $W^U_{s,q}$ when the choice of weight matters.

The complexified space is the Hilbert direct sum in
\eqref{eq:complexnorm}. Its antilinear involution is
$(\mathfrak c z)_{(j,\sigma)}=\overline{z_{(j,-\sigma)}}$; the fixed
points form the physical real phase space. The unweighted symplectic
form on that space is
\[
 \Omega_{\rm symp}(v,h)=\ii\sum_j(v_j\bar h_j-h_j\bar v_j).
\]
Cauchy--Schwarz gives
\[
 |\Omega_{\rm symp}(v,h)|
 \leq2\sum_j|v_j||h_j|
 \leq2\|v\|_{\ell^2}\|h\|_{\ell^2}
 \leq2\|v\|_s\|h\|_s.
\]
For the coordinate vector $e_j$,
\[
 \Omega_{\rm symp}(v,e_j)=-2\operatorname{Im}v_j,\qquad
 \Omega_{\rm symp}(v,\ii e_j)=2\operatorname{Re}v_j.
\]
Thus $\Omega_{\rm symp}(v,h)=0$ for all $h$ implies $v=0$,
and the form is weak symplectic on $W_s$. The convention
$\Omega_{\rm symp}(X_F,h)=dF[h]$ gives
$(X_F)_j=-\ii\partial_{\bar u_j}F$.
For real analytic Hamiltonians with a $W_s$-valued vector field, the
complexified formula is
\begin{equation}\label{eq:ham-conventions}
 (X_F)_{(j,\sigma)}=-\sigma\ii\partial_{u_{(j,-\sigma)}}F,
 \qquad
 \{F,G\}=dF[X_G]
 =-\ii\sum_{j,\sigma}\sigma\partial_{u_{(j,\sigma)}}F
                              \partial_{u_{(j,-\sigma)}}G.
\end{equation}
The differential identity gives
$d(F\circ\Phi_G^t)/dt=\{F,G\}\circ\Phi_G^t$ whenever the indicated
flow exists. The absolute convergence and differentiability required
here are provided by the coefficient estimates below.

\subsection{Homogeneous coefficients and projections}
Homogeneous expansions and the separation of low and high Fourier
factors provide a common setup for Hamiltonian normal forms; see
\cite[Sections 3--4]{BFM24} and \cite[Sections 5--6]{BFLM25}.
We introduce both decompositions here, with degree-dependent estimates
to be used in the iteration.

For $m\geq3$, let $\mathcal P_m$ consist of the real, zero-momentum
homogeneous coefficient arrays described above with $c_m(F)<\infty$.
We use the algebraic direct sum
$\mathcal P_{a,b}=\bigoplus_{m=a}^b\mathcal P_m$ for $a\leq b$,
and $\mathcal P_{a,b}=\{0\}$ for $a>b$. Infinite-degree Hamiltonians belong to the analytic coefficient
space defined below. The order
$\ord F\geq a$ means that every coefficient of degree smaller than
$a$ vanishes. Degree projections will be denoted by $\Pi_mF=F_m$,
$\Pi_{[a,b]}F=\sum_{m=a}^bF_m$ and
$\Pi_{\geq a}F=\sum_{m\geq a}F_m$.
We also write $\Pi_{<a}F=\sum_{3\leq m<a}F_m$.

For a tuple $\mathcal J\in\mathcal I_m$, set
$h_N(\mathcal J)=\#\{i:|j_i|>N\}$. The projector
$\mathscr H_{\ell,N}$ keeps precisely the monomials with
$h_N=\ell$, and $\mathscr H_{\geq3,N}$ keeps those with
$h_N\geq3$. These are coefficient projectors, in contrast to the
phase-space projections $\Pi_{\leq N}$ and $\Pi_{>N}$.
In particular
\begin{equation}\label{eq:polynomial-splits}
 \Pi_{[3,d]}P=(\Id-\mathscr H_{\geq3,N})\Pi_{[3,d]}P
                  +\mathscr H_{\geq3,N}\Pi_{[3,d]}P.
\end{equation}
Counting factors includes multiplicity: a factor $u_j^a$ contributes
$a$ to $h_N$ if $|j|>N$. The count is invariant under permutations
and under conjugation. Consequently every projector used here preserves
the symmetric coefficient convention and reality.

If a collected monomial has multiplicities $(\alpha_J)$ with
$\sum\alpha_J=m$, its number of ordered occurrences is
$m!/\prod_J\alpha_J!$. Its symmetric ordered coefficient is therefore
its collected coefficient multiplied by $\prod_J\alpha_J!/m!$.
All suprema and all coefficientwise divisions in the paper use this
ordered convention. In particular, for a set $\mathscr A$ of tuples
and a coefficient multiplier $\mathscr D$,
\[
 c_m(\Pi_{\mathscr A}F)\leq c_m(F),\qquad
 \sup_{\mathcal J\in\mathcal I_m}|D_{\mathcal J}|\leq D
 \quad\Longrightarrow\quad c_m(\mathscr DF)\leq D c_m(F).
\]

\subsection{Analytic coefficient majorants}
The weighted Hamiltonian majorants in \cite[Definitions 1.1--1.2]{BMP20}
control vector-field evaluation and coefficient projections.
Here we use a weighted sum of the ordered coefficient suprema;
Lemma~\ref{lem:cut} gives its comparison with vector-field norms.
We state the estimates needed for the iteration below and give their
proofs in Appendix~\ref{app:calculus}. As in
\cite[Appendix B]{BMP20}, this separates the analytic calculus from
the normal-form construction while retaining its constants explicitly.
\begin{definition}[Coefficient-majorant analytic class]\label{def:coef}
For $F=\sum_{m\geq3}F_m$ put
\begin{equation}\label{eq:coefnorm}
 \coef F_r=\sum_{m\geq3} m^3 b^{m-2}c_m(F)r^{m-2}.
\end{equation}
We require $\coef P_R\leq M$ for fixed $R,M>0$.
For a coefficient projector $\Pi_{\mathscr A}$ and a diagonal
coefficient multiplier $\mathscr D$ with entries
$|D_{\mathcal J}|\leq D$,
\[
 c_m(\Pi_{\mathscr A}F)\leq c_m(F),\qquad
 \coef{\Pi_{\mathscr A}F}_r\leq\coef F_r,\qquad
 \coef{\mathscr DF}_r\leq D\coef F_r.
\]
We separately use the ordinary complex-ball vector-field seminorm
\[
 \vf F_{r,s}=r^{-1}\sup_{\nrm u_s<r}\nrm{X_F(u)}_s.
\]
The coefficient norm is the stronger norm used for polynomials and generators; transported remainders may be estimated in the ordinary seminorm.
\end{definition}

\subsection{Weighted evaluation and the cutting estimates}
\begin{lemma}[Weighted evaluation and high-mode gain]\label{lem:cut}
For $s\geq s_0$, $\vf F_{r,s}\leq C\coef F_r$.
For the two cutting estimates fix $s>s_0$ and
$\nu=(1-C_f)(s-s_0)$. If each monomial of $Q$ has at least three factors with $|j|>N$, then
\begin{equation}\label{eq:cutgain}
 \vf Q_{r,s}\leq C\ee^{-\nu f(N)}\coef Q_r.
\end{equation}
If every monomial has two such factors, then
\begin{equation}\label{eq:lowgain}
 r^{-1}\sup_{\nrm u_s<r}\nrm{\Pi_{\leq N}X_Q(u)}_s
 \leq C\ee^{-\nu f(N)}\coef Q_r.
\end{equation}
\end{lemma}
The proof is given in Appendix~\ref{app:cut-calculus}.

\begin{lemma}[Evaluation of Hamiltonians and analytic tails]\label{lem:scalar-evaluation}
If $\coef F_R<\infty$, the coefficient series defines a holomorphic
Hamiltonian and vector field on $B_s(R)$. Their sums over homogeneous
degrees and all fixed-order Fr\'echet derivatives converge uniformly
on each $B_s(t)$, $0<t<R$, with derivatives measured in multilinear
operator norm. For $0<t<R$ and
$\ord F\geq h\geq3$,
\begin{equation}\label{eq:tail-complete}
 \coef F_t\leq(t/R)^{h-2}\coef F_R,
 \qquad \sup_{B_s(t)}|F|\leq C t^2\coef F_t.
\end{equation}
If $\coef P_R\leq M$, then for $t\leq R$
$\coef P_t\leq Mt/R$, and
$\coef{\Pi_{\geq d+1}P}_t\leq M(t/R)^{d-1}$.
\end{lemma}
The proof is given in Appendix~\ref{app:evaluation-calculus}.

\subsection{Poisson brackets and Lie transformations}
\begin{lemma}[Poisson algebra and Lie series]\label{lem:lie}
With the Hamiltonian convention \eqref{eq:ham-conventions}, for homogeneous degrees $m,n\geq3$,
\begin{equation}\label{eq:coeffbracket}
 c_{m+n-2}(\{F_m,G_n\})\leq2mn\,c_m(F)c_n(G).
\end{equation}
Consequently, for $0<\rho<a$,
\begin{equation}\label{eq:liecauchy}
 \coef{\{F,G\}}_{a-\rho}\leq C\frac a\rho\coef F_a\coef G_a,
 \qquad
 \frac1{\ell!}\coef{\ad_G^\ell F}_{a-\rho}
 \leq\left(C\frac a\rho\coef G_a\right)^\ell\coef F_a.
\end{equation}
If $\eta=Ca\coef G_a/\rho$ is small, the real canonical time-one flow is defined on $B_s(a-\rho)$ and satisfies
\begin{align}\label{eq:flows}
 \sup\nrm{\Phi_G-\Id}_s&\leq Ca\coef G_a,&
 \sup\nrm{D\Phi_G^{\pm1}}&\leq1+C\eta,\\
 \coef{F\circ\Phi_G}_{a-\rho}&\leq(1+C\eta)\coef F_a.\nonumber
\end{align}
In an ordinary vector-field norm the transport rule is
\begin{equation}\label{eq:weaktransport}
 \vf{E\circ\Phi_G}_{a-\rho,s}
 \leq \frac a{a-\rho}(1+C\eta)\vf E_{a,s}.
\end{equation}
\end{lemma}
The proof is given in Appendix~\ref{app:lie-calculus}.

\begin{lemma}[Finite Lie expansion with integral remainder]\label{lem:integral-lie}
Under the same flow smallness condition, for every integer $\ell_*\geq0$,
\begin{align}\label{eq:integral-lie-full}
 F\circ\Phi_G^t
 &=\sum_{\ell=0}^{\ell_*}\frac{t^\ell}{\ell!}\ad_G^\ell F
 +\frac1{\ell_*!}\int_0^t(t-v)^\ell_*
       (\ad_G^{\ell_*+1}F)\circ\Phi_G^v\,dv,\quad 0\leq t\leq1.
\end{align}
Its last term at $t=1$ has coefficient norm at most
$\eta^{\ell_*+1}(1-\eta)^{-1}\coef F_a$ on $B_s(a-\rho)$, after
increasing the absolute constant used in $\eta$.
If $\ord F\geq h$ and $\ord G\geq k$, it has order at least
$h+(\ell_*+1)(k-2)$.
\end{lemma}
The proof is given in Appendix~\ref{app:integral-calculus}.

\begin{remark}[Local and nonlocal nonlinearities]
For a local analytic translation-invariant potential,
$c_m(P)\leq C_1C_2^m$. If each input multiplier has modulus at
most $C_3$, the resulting coefficients satisfy
\[
 c_m(P_{\rm mult})\leq C_1(C_2C_3)^m,\qquad
 \coef{P_{\rm mult}}_r
 \leq C_1(C_2C_3)^2\sum_{m\geq3}m^3(bC_2C_3r)^{m-2}<\infty
 \quad(bC_2C_3r<1).
\]
The sparse constructions below satisfy $\sup_m c_m(P)<\infty$;
the same calculation with a fixed smaller radius places them in
Definition~\ref{def:coef}.
\end{remark}

\section{Resonance and normal form}\label{sec:res}
\subsection{Frequency assumptions and exact blocks}
Let $H=H_0+P$, $H_0=\sum_{j\in\ZZ}\omega_j u_{(j,+)}u_{(j,-)}$. For $\mathcal J\in\mathcal I_m$, write $\Omega(\mathcal J)=\sum\sigma_l\omega_{j_l}$.
\begin{assumption}[Even frequencies on the circle]\label{ass:freq}
The frequencies are real and even: $\omega_j=\lambda_{|j|}$. There are fixed $\beta>1$, $p\geq1$, $\tau,C>0$, $0<\gamma<1$ and $\delta,c_{\rm gap}>0$ such that $|\omega_j|\leq C\br j^\beta$, $\omega_j\geq C^{-1}|j|^\beta$ for large $|j|$, and
\begin{equation}\label{eq:radialgap}
 |\lambda_n-\lambda_k|\geq c_{\rm gap}(\br n^\delta+\br k^\delta)
 \quad(n,k\in\NN_0,\ n\ne k).
\end{equation}
For every $m\geq3$, $K\geq m$ and zero-momentum tuple $\mathcal J\in\mathcal I_m$ supported below $K$, either
\begin{equation}\label{eq:nr}
 |\Omega(\mathcal J)|\geq\gamma^m(CmK)^{-\tau m^p},
\end{equation}
or the tuple admits a pairing in which each pair has opposite signs and equal absolute spatial indices.
\end{assumption}
The unavoidable equality $\omega_n=\omega_{-n}$ is explicitly included. The gap \eqref{eq:radialgap} excludes coincidences between distinct absolute indices. Define once and for all
\begin{equation}\label{eq:blocks}
 \mathcal B_0=\{0\},\qquad \mathcal B_n=\{n,-n\}\quad(n\geq1).
\end{equation}
For indices in different blocks,
\begin{equation}\label{eq:sep}
 |j-k|+|\omega_j-\omega_k|\geq c_{\rm gap}(\br j^\delta+\br k^\delta).
\end{equation}
For indices in the same block $f(\br j)=f(\br k)$ exactly. These equal weights enter the energy identity in Section~\ref{sec:stability}. A quantitative elementary verification of \eqref{eq:radialgap} for the applications is given in Lemma~\ref{lem:1dgap}.

\subsection{Classification by the number of high factors}
Fix $d\geq4$ and $N\geq\max(d,N_0)$. Relabel the factors of a
tuple so that $|j_1|\geq\cdots\geq|j_m|$, and put
$h=h_N(\mathcal J)$. The sets below are invariant under the original
ordering of the tuple. Set
\begin{equation}\label{eq:buffer}
 M_{d,N}=\left\lceil C_0(dN)^{B_0}\right\rceil,
 \qquad B_0\geq\max(1,\beta/\delta),
\end{equation}
where $C_0$ is chosen so that $M_{d,N}\geq4dN$ and
$c_{\rm gap}M_{d,N}^{\delta}\geq4C_\omega dN^\beta+2$.
Here $C_\omega$ bounds $|\omega_j|/\br j^\beta$; enlarging $N_0$
makes $N\geq c_0$ and all frequencies with $|j|>N$ positive.

We use the number $h$ of high factors to decide which monomials to
retain, eliminate or estimate as a remainder. The alternatives are
listed in Table~\ref{tab:resonance-cases}. In the rows with $h=2$,
the first two factors are the high factors.
\begin{table}[htbp]
\centering\small
\setlength{\tabcolsep}{5pt}
\renewcommand{\arraystretch}{1.25}
\begin{tabular}{@{}c p{0.48\textwidth} c l@{}}
\toprule
$h$ & Further conditions & Class & Treatment\\
\midrule
$0$ & $\Omega=0$ & $\mathscr R_0$ & Retain\\
$0$ & $\Omega\ne0$ & $\mathscr N_0$ & Eliminate\\
$1$ & All tuples & $\mathscr N_1$ & Eliminate\\
$2$ & $\sigma_1=\sigma_2$ & $\mathscr N_{21}$ & Eliminate\\
$2$ & $\sigma_1=-\sigma_2$, $\Omega=0$
    & $\mathscr R_{2,\mathrm{ex}}$ & Retain\\
$2$ & $\sigma_1=-\sigma_2$, $\Omega\ne0$,
      $\min(|j_1|,|j_2|)>M_{d,N}$, $|j_1|=|j_2|$
    & $\mathscr R_{2,\mathrm{bl}}$ & Retain\\
$2$ & $\sigma_1=-\sigma_2$, $\Omega\ne0$, and either
      $\min(|j_1|,|j_2|)\leq M_{d,N}$ or $|j_1|\ne|j_2|$
    & $\mathscr N_{22}$ & Eliminate\\
$\geq3$ & All tuples & $\mathscr Q_{\geq3}$ & High remainder\\
\bottomrule
\end{tabular}
\caption{Classification of zero-momentum monomials through degree $d$.}
\label{tab:resonance-cases}
\end{table}
For $h=1$, exact resonance would require an opposite-sign factor with
the same absolute index as the single high factor, which is absent.
For $h=2$ and $\sigma_1=\sigma_2$, the high factors can be paired
neither with the low factors nor with one another. Assumption~\ref{ass:freq}
therefore excludes exact resonance in $\mathscr N_1$ and
$\mathscr N_{21}$. For two opposite high signs we keep all exactly
resonant terms, together with the nonexact terms lying above the buffer
in one block. The remaining opposite-sign terms form $\mathscr N_{22}$.
Monomials with at least three high factors are assigned to
$\mathscr Q_{\geq3}$ regardless of their divisor.
These alternatives give the disjoint partition
\begin{equation}\label{eq:res-partition}
 \mathcal I_m=\mathscr R_0\sqcup\mathscr R_{2,\mathrm{ex}}
 \sqcup\mathscr R_{2,\mathrm{bl}}\sqcup\mathscr N_0
 \sqcup\mathscr N_1\sqcup\mathscr N_{21}\sqcup\mathscr N_{22}
 \sqcup\mathscr Q_{\geq3}.
\end{equation}
Empty classes make no contribution to this partition. The classification
distinguishes the algebraic condition $\Omega=0$ from the two-high
terms retained for their dynamical structure.
For any union $\mathscr A$ of these classes, we write
$\Pi_{\mathscr A}$ for the coefficient projector onto $\mathscr A$.

\begin{definition}[Buffered $N$-cutting normal form]\label{def:buffered-normal}
A polynomial of degrees $3,\ldots,d$ is in buffered $N$-cutting
normal form if its coefficients are supported on
$\mathscr R=\mathscr R_0\cup\mathscr R_{2,\mathrm{ex}}
\cup\mathscr R_{2,\mathrm{bl}}$. Write it as
$Z=Z_0+Z_{2,\mathrm{ex}}+Z_{\rm blk}$ according to these three
disjoint supports. In particular it has degree at most two in
$u^{>N}$, and $Z_0$ depends only on $u^{\leq N}$.
\end{definition}

\subsection{Divisors in each nonresonant case}
\begin{lemma}[Extended divisor estimate]\label{lem:div}
For every tuple in
$\mathscr N=\mathscr N_0\cup\mathscr N_1\cup\mathscr N_{21}
\cup\mathscr N_{22}$ of degree $3\leq m\leq d$,
\begin{equation}\label{eq:Delta}
 |\Omega(\mathcal J)|\geq\Delta_{d,N}^{-1},\qquad
 \Delta_{d,N}=\gamma^{-d}(C_{\rm div}dN)^{C_{\rm div}d^p}.
\end{equation}
The constant $C_{\rm div}$ depends only on the fixed frequency data.
\end{lemma}
\begin{proof}
For $\mathscr N_0$, all factors are at most $N$, so the hypothesis
gives $|\Omega|\geq\gamma^m(CmN)^{-\tau m^p}$.
For $\mathscr N_1$, momentum gives
$|j_1|=|\sum_{i=2}^m\sigma_1\sigma_i j_i|\leq(m-1)N$.
Application of the same hypothesis at $K=dN$ gives
$|\Omega|\geq\gamma^m(CmdN)^{-\tau m^p}$.
One can also verify the high-frequency alternative directly:
$|\sum_{i=2}^m\sigma_i\omega_{j_i}|\leq C_\omega dN^\beta$,
whereas $\omega_{j_1}\geq c_\omega|j_1|^\beta$ for
$|j_1|>N$. Once
$|j_1|\geq(2C_\omega d/c_\omega)^{1/\beta}N$, the divisor
has modulus at least $C_\omega dN^\beta$. The momentum bound
provides a finite cutoff also in the complementary range.

For $\mathscr N_{21}$ both high frequencies enter with the same
sign. Put $J=\max(|j_1|,|j_2|)$ and
$K_1=(4C_\omega d/c_\omega)^{1/\beta}N$.
If $J\geq K_1$, positivity and growth imply
\[
 |\Omega|\geq\omega_{j_1}+\omega_{j_2}
       -\left|\sum_{i=3}^m\sigma_i\omega_{j_i}\right|
 \geq c_\omega J^\beta-C_\omega dN^\beta\geq1.
\]
If $J<K_1$, then $\beta>1$ gives
\[
 \max_{1\leq i\leq m}|j_i|\leq K_1
       \leq C d^{1/\beta}N\leq CdN,\qquad
 |\Omega|\geq\gamma^m(CmdN)^{-\tau m^p},
\]
where we use the exclusion of exact resonance in
$\mathscr N_{21}$ and Assumption~\ref{ass:freq}.

It remains to consider opposite high signs. Momentum now implies
\begin{equation}\label{eq:two-high-momentum}
 |j_1-j_2|\leq(m-2)N\leq dN.
\end{equation}
If $\min(|j_1|,|j_2|)\leq M_{d,N}$, the triangle inequality
gives $\max(|j_1|,|j_2|)\leq M_{d,N}+dN\leq2M_{d,N}$.
Apply the divisor hypothesis at $K=2M_{d,N}$.
If both moduli exceed $M_{d,N}$, a tuple in $\mathscr N_{22}$
must have distinct absolute indices. The gap assumption gives
\begin{align*}
 |\Omega|&\geq|\omega_{j_1}-\omega_{j_2}|
                  -C_\omega dN^\beta\\
 &\geq c_{\rm gap}(\br {j_1}^{\delta}+\br {j_2}^{\delta})
                  -C_\omega dN^\beta\geq1.
\end{align*}
This proves a uniform lower bound in all the unbounded cases.

Every finite cutoff used above is at most $C(dN)^{B_1}$ for a
fixed $B_1$. Because $m\leq d$, $0<\gamma<1$ and $p\geq1$,
\[
 \gamma^m(CmC(dN)^{B_1})^{-\tau m^p}
 \geq\gamma^d(C_{\rm div}dN)^{-C_{\rm div}d^p},
\]
after choosing $C_{\rm div}$ once. The right side is at most one,
so it also covers the unbounded cases. This proves the lemma.
\end{proof}

\section{Iteration lemma}\label{sec:iteration}
\subsection{The homological equation}
\begin{lemma}[Homological equation]\label{lem:homo}
Let $Q\in\mathcal P_{k,d}$ have at most two high factors per
monomial. There are unique $Z^*$ supported on $\mathscr R$ and
$G$ supported on $\mathscr N$ satisfying
\begin{equation}\label{eq:homological-full}
 \{H_0,G\}+Q=Z^*.
\end{equation}
Both have degrees $k,\ldots,d$, are real and have zero momentum.
For every radius at which the norms are defined,
$\coef{Z^*}_r\leq\coef Q_r$ and
$\coef G_r\leq\Delta_{d,N}\coef Q_r$.
\end{lemma}
\begin{proof}
With the convention \eqref{eq:ham-conventions}, each monomial satisfies
\[
 \{H_0,u^{\mathcal J}\}
 =\ii\!\left(\sum_{\ell=1}^m\sigma_\ell\omega_{j_\ell}\right)
                      u^{\mathcal J}
 =\ii\Omega(\mathcal J)u^{\mathcal J}.
\]
We therefore define the coefficients on the two disjoint supports by
\begin{equation}\label{eq:homo-coefficients}
 (Z^*_{\mathcal J},G_{\mathcal J})=
 \begin{cases}
  (Q_{\mathcal J},0),&\mathcal J\in\mathscr R,\\[2pt]
  (0,-Q_{\mathcal J}/(\ii\Omega(\mathcal J))),
       &\mathcal J\in\mathscr N.
 \end{cases}
\end{equation}
The assumption on the number of high factors leaves no other
coefficients of $Q$. Lemma~\ref{lem:div} ensures that every division
in the second line is defined. For each $m=k,\ldots,d$,
\begin{align*}
 c_m(Z^*)&\leq c_m(Q),&
 c_m(G)&\leq\Delta_{d,N}c_m(Q),\\
 \sum_{m=k}^d m^3b^{m-2}c_m(Z^*)r^{m-2}
      &\leq\coef Q_r,&
 \sum_{m=k}^d m^3b^{m-2}c_m(G)r^{m-2}
      &\leq\Delta_{d,N}\coef Q_r.
\end{align*}
These are the claimed norm estimates. The coefficient identity
$\ii\Omega(\mathcal J)G_{\mathcal J}+Q_{\mathcal J}
=Z^*_{\mathcal J}$ also proves uniqueness once the supports have
been prescribed.

To check reality, conjugate a tuple by reversing all its signs. Then
$\Omega(\bar{\mathcal J})=-\Omega(\mathcal J)$, the two supports
are unchanged, and on $\mathscr N$ we have
\[
 G_{\bar{\mathcal J}}
 =-\frac{\overline{Q_{\mathcal J}}}{-\ii\Omega(\mathcal J)}
 =\overline{-\frac{Q_{\mathcal J}}{\ii\Omega(\mathcal J)}}
 =\overline{G_{\mathcal J}}.
\]
The same relation for $Z^*$ follows directly from its first line in
\eqref{eq:homo-coefficients}. Since no indices or degrees are changed,
both polynomials retain zero momentum and degrees $k,\ldots,d$.
\end{proof}

\subsection{Auxiliary radius and inductive estimates}
In this section abbreviate $\Delta=\Delta_{d,N}$. Let $C_P\geq1$
be fixed so that $\coef P_{4a}\leq C_Pa$ whenever $4a\leq R$;
one may take $C_P\geq4M/R$. Choose $K$ sufficiently large and put
\begin{equation}\label{eq:auxradius}
 a=(Kd^2\Delta)^{-1},\quad 4a<R/4,\quad
 \rho_k=4a-\frac{(k-3)a}{d-2}\quad(3\leq k\leq d+1).
\end{equation}
Thus $\rho_3=4a$, $\rho_{d+1}=3a$, and the loss per step is
$\ell_a=a/(d-2)$. Fix $C_*\geq1$ as in the proof below and set
$\mu_a=C_*d^2\Delta a=C_*/K\leq1/32$.

\begin{lemma}[Order-uniform iteration]\label{lem:iteration-full}
For each $k=3,\ldots,d+1$ there is a canonical map
$\mathcal T^{(k)}:B_s(\rho_k)\to B_s(4a)$ and a decomposition
\begin{equation}\label{eq:inductdecomp}
 H\circ\mathcal T^{(k)}
 =H_0+Z^{(k)}+P_k+R_{k,T}+E_k.
\end{equation}
Here $\mathcal T^{(3)}=\Id$, $Z^{(3)}=0$, $P_{d+1}=0$,
$Z^{(k)}$ is a cumulative normal polynomial of degrees $3,\ldots,d$,
$P_k\in\mathcal P_{k,d}$ has at most two high factors per monomial,
$\ord R_{k,T}\geq d+1$, and $\ord E_k\geq3$.
$Z^{(k)}=\sum_{j=3}^{k-1}Z_j^*$ records the normal terms accumulated
before stage $k$; each $Z_j^*$ has degrees between $j$ and $d$.
For $3\leq k\leq d$ the generator $G_k\in\mathcal P_{k,d}$ obeys
\begin{align}
 \coef{P_k}_{\rho_k}&\leq C_Pa\mu_a^{k-3},&
 \coef{G_k}_{\rho_k}&\leq C_P\Delta a\mu_a^{k-3},\label{eq:Pkaux}\\
 \coef{Z^{(k)}}_{\rho_k}&\leq2C_Pa,&
 \sum_{j=3}^{k}\eta_j&\leq C d\Delta a,\label{eq:etasum}
\end{align}
where $\eta_k=C_L\rho_k\coef{G_k}_{\rho_k}/\ell_a$ and
$C_L$ is a fixed Lie-calculus constant. Uniformly in $k$,
\begin{equation}\label{eq:aux-remainders-full}
 \coef{R_{k,T}}_{\rho_k}+\coef{E_k}_{\rho_k}\leq Ca,
 \qquad \vf{E_k}_{\rho_k,s}\leq Ca\ee^{-\nu f(N)}.
\end{equation}
The maps are tangent to the identity and satisfy
$\sup_{B_s(\rho_k)}\|\mathcal T^{(k)}-\Id\|_s\leq C\Delta a^2$.
\end{lemma}

\begin{proof}
\emph{Initial decomposition.}
Set
\[
 P_3=(\Id-\mathscr H_{\geq3,N})\Pi_{[3,d]}P,\quad
 E_3=\mathscr H_{\geq3,N}\Pi_{[3,d]}P,\quad
 R_{3,T}=\Pi_{\geq d+1}P.
\]
The three projections have disjoint supports, so
\[
 \max\{c_m(P_3),c_m(E_3),c_m(R_{3,T})\}\leq c_m(P).
\]
Lemma~\ref{lem:scalar-evaluation} then gives
\begin{align}\label{eq:initial-iteration-bounds}
 \max\{\coef{P_3}_{4a},\coef{E_3}_{4a}\}
       &\leq\coef P_{4a}\leq\frac{4M}{R}a\leq C_Pa,\nonumber\\
 \coef{R_{3,T}}_{4a}
       &\leq M\left(\frac{4a}{R}\right)^{d-1}\leq C_Pa,\nonumber\\
 \vf{E_3}_{4a,s}
       &\leq C\ee^{-\nu f(N)}\coef{E_3}_{4a}
        \leq CC_Pa\ee^{-\nu f(N)}.
\end{align}
Here we have used $4a/R<1$ and $C_P\geq4M/R$ in the second
line, and Lemma~\ref{lem:cut} in the third. The degree projections
give
\[
 P_3\in\mathcal P_{3,d},\qquad
 \ord E_3\geq3,\qquad \ord R_{3,T}\geq d+1.
\]
In particular, the initial Taylor tail enters the induction with
\[
 \coef{R_{3,T}}_{4a}\leq
       (4a/R)^{d-1}\coef P_R,
\]
before any normalization step is performed.

\emph{Homological solution and the radius margin.}
Suppose the induction is proved at stage $k$. Solve
$\{H_0,G_k\}+P_k=Z_k^*$ by Lemma~\ref{lem:homo} and denote
the time-one flow by $\Phi_k$. Since
$\rho_k/\ell_a\leq4(d-2)$, the inductive bounds give
\begin{equation}\label{eq:eta-step-full}
 \eta_k\leq C_1C_Pd\Delta a\mu_a^{k-3}.
\end{equation}
Choose $C_*$ larger than the fixed constants occurring in this proof
and then $K\geq32C_*$. This ensures $\eta_k\leq1/16$ and
$C\rho_k\coef{G_k}_{\rho_k}\leq\ell_a/8$ for every $k$.
To see the radius margin explicitly, the flow estimate gives
\begin{align}\label{eq:step-domain-estimates}
 \sup_{B_s(\rho_{k+1})}\|\Phi_k-\Id\|_s
  &\leq C\rho_k\coef{G_k}_{\rho_k}
   \leq CC_P\Delta a^2\mu_a^{k-3}\leq\ell_a/8,\nonumber\\
 \frac{C\rho_k\coef{G_k}_{\rho_k}}{\ell_a}
  &\leq CC_Pd\Delta a\mu_a^{k-3}
   =\frac{CC_P}{Kd}\mu_a^{k-3}.
\end{align}
The last expression is uniformly small after $K$ is fixed. The same
bound holds for the intermediate times $0\leq t\leq1$ of the flow.
Consequently $\|\Phi_k^t(u)\|_s<\rho_k$ whenever
$\|u\|_s<\rho_{k+1}=\rho_k-\ell_a$. Lemma~\ref{lem:lie}
and its differential estimate then give
\begin{equation}\label{eq:step-differentials}
 \sup_{B_s(\rho_{k+1})}\|D\Phi_k-\Id\|\leq C\eta_k,
 \qquad
 \sup_{B_s(\rho_{k+1})}
       \max\{\|D\Phi_k\|,\|(D\Phi_k)^{-1}\|\}
       \leq1+C\eta_k.
\end{equation}
The constants used to choose $C_*$ and $K$ do not depend on the
stage $k$ or the final degree $d$.

\emph{Lie expansions of the quadratic and nonlinear terms.}
Write $\ad_kF=\{F,G_k\}$. The terms containing $H_0$ are
rewritten through their first bracket:
\begin{align}
 H_0\circ\Phi_k
 &=H_0+(Z_k^*-P_k)
   +\sum_{j\geq1}\frac{\ad_k^j(Z_k^*-P_k)}{(j+1)!},
       \label{eq:H0-Lie-expanded}\\
 Z^{(k)}\circ\Phi_k
 &=Z^{(k)}+\sum_{j\geq1}\frac{\ad_k^j Z^{(k)}}{j!},
       \label{eq:Z-Lie-expanded}\\
 P_k\circ\Phi_k
 &=P_k+\sum_{j\geq1}\frac{\ad_k^jP_k}{j!}.
       \label{eq:P-Lie-expanded}
\end{align}
The series on the right are norm convergent because
$Z_k^*-P_k$, $Z^{(k)}$ and $P_k$ have bounded coefficient norms.
The first line follows by integrating
$d(H_0\circ\Phi_k^t)/dt=(Z_k^*-P_k)\circ\Phi_k^t$;
this identity is justified on the linear generator domain in
Subsection~\ref{subsec:flow} and then by continuity.
Combining the three lines gives the complete new contribution
\begin{equation}\label{eq:Sk}
 S_k=\sum_{j\geq1}\ad_k^j
 \left(\frac{Z^{(k)}}{j!}+\frac{jP_k+Z_k^*}{(j+1)!}\right).
\end{equation}
In particular, the coefficient of $\ad_k^jP_k$ is
$1/j!-1/(j+1)!=j/(j+1)!$.

\emph{Orders and finite integral tails.}
For each $j\geq1$, the orders in the preceding formulas obey
\begin{align*}
 \ord\ad_k^j Z^{(k)}&\geq3+j(k-2),\\
 \ord\ad_k^j P_k, \ord\ad_k^j Z_k^*&\geq k+j(k-2).
\end{align*}
The first lower bound is at least $k+1$; the second is at least
$2k-2\geq k+1$. Thus $\ord S_k\geq k+1$.
The recurrence uses the convergent expression \eqref{eq:Sk} and its
coefficient projections. Appendix~\ref{app:finite-step} writes the
same step as finite Lie polynomials plus integral remainders:
\eqref{eq:finite-tail-orders} and \eqref{eq:Sk-finite-projections}
account explicitly for every term of degree greater than $d$.

The lowest new homogeneous terms are obtained before the
degree-$d$ projection. Writing
$G_{k,m}=\Pi_mG_k$, and using the analogous notation for $P_k$
and $Z_k^*$, we find
\begin{align}\label{eq:leading-step-brackets}
 \Pi_4S_3
   &=\frac12\{P_{3,3}+Z^*_{3,3},G_{3,3}\},\nonumber\\
 \Pi_{k+1}S_k
   &=\{\Pi_3Z^{(k)},G_{k,k}\},\qquad k\geq4.
\end{align}
Indeed $Z^{(3)}=0$; for $k\geq4$, the other first brackets
start at degree $2k-2>k+1$, and the second brackets involving
$Z^{(k)}$ start at $2k-1>k+1$. This also shows why the cumulative
normal polynomial must be included in the one-step estimate.

\emph{The polynomial estimate.}
We apply the iterated-bracket bound at radii $\rho_k$ and
$\rho_{k+1}$. For each $j\geq1$, its two contributions satisfy
\begin{align*}
 \frac1{j!}\coef{\ad_k^jZ^{(k)}}_{\rho_{k+1}}
   &\leq\eta_k^j\coef{Z^{(k)}}_{\rho_k},\\
 \frac1{(j+1)!}\coef{\ad_k^j(jP_k+Z_k^*)}_{\rho_{k+1}}
   &\leq\frac{\eta_k^j}{j+1}
                (j\coef{P_k}_{\rho_k}+\coef{Z_k^*}_{\rho_k})
    \leq\eta_k^j\coef{P_k}_{\rho_k}.
\end{align*}
Summing these bounds, and keeping a harmless factor two for the
$P_k$ terms, gives
\begin{align}\label{eq:onestep}
 \coef{S_k}_{\rho_{k+1}}
 &\leq\sum_{j\geq1}\eta_k^j
       (\coef{Z^{(k)}}_{\rho_k}+2\coef{P_k}_{\rho_k})\nonumber\\
 &\leq\frac{\eta_k}{1-\eta_k}
       (\coef{Z^{(k)}}_{\rho_k}+2\coef{P_k}_{\rho_k})\nonumber\\
 &\leq C_2C_Pd\Delta a\coef{P_k}_{\rho_k}
 \leq\mu_a\coef{P_k}_{\rho_k}.
\end{align}
For the third inequality use
$\eta_k\leq C_1d\Delta\coef{P_k}_{\rho_k}$ and
$\coef{Z^{(k)}}_{\rho_k}+2\coef{P_k}_{\rho_k}\leq4C_Pa$.
The final inequality follows from $C_*d^2\geq C_2C_Pd$,
which holds for every $d\geq4$ after fixing $C_*$.
At $k=3$ this reads
$\coef{S_3}_{\rho_4}\leq C_2C_P^2d\Delta a^2$.
For $k\geq4$, insertion of the induction hypothesis gives
$\coef{S_k}_{\rho_{k+1}}\leq C_Pa\mu_a^{k-2}$.

\emph{New decomposition.}
Define the new terms and the two intermediate remainders by
\begin{align*}
 P_{k+1}&=(\Id-\mathscr H_{\geq3,N})\Pi_{[k+1,d]}S_k,\\
 Q_{k+1,>}&=\mathscr H_{\geq3,N}\Pi_{[k+1,d]}S_k,\\
 B_{k+1,T}&=\Pi_{\geq d+1}S_k,\\
 Z^{(k+1)}&=Z^{(k)}+Z_k^*,\\
 R_{k+1,T}&=R_{k,T}\circ\Phi_k+B_{k+1,T},\\
 E_{k+1}&=E_k\circ\Phi_k+Q_{k+1,>}.
\end{align*}
Set $\mathcal T^{(k+1)}=\mathcal T^{(k)}\circ\Phi_k$.
The three series expansions prove \eqref{eq:inductdecomp} at
stage $k+1$. The degree projections and the high-factor projection
are contractive. In particular, \eqref{eq:onestep} implies
\begin{equation}\label{eq:new-projected-bounds}
 \max\{\coef{P_{k+1}}_{\rho_{k+1}},
        \coef{Q_{k+1,>}}_{\rho_{k+1}},
        \coef{B_{k+1,T}}_{\rho_{k+1}}\}
 \leq\coef{S_k}_{\rho_{k+1}}
 \leq C_Pa\mu_a^{k-2}.
\end{equation}
The finite representation \eqref{eq:Sk-finite-projections} in
Appendix~\ref{app:finite-step} additionally identifies these
projections: $P_{k+1}$ and $Q_{k+1,>}$ come from its polynomial
$\mathscr S_k$, while $B_{k+1,T}$ includes the terms above degree
$d$ and the three integral tails.
The cumulative normal polynomial satisfies
\[
 \coef{Z^{(k+1)}}_{\rho_{k+1}}
 \leq\sum_{j=3}^k\coef{P_j}_{\rho_j}
 \leq C_Pa\sum_{j=0}^{k-3}\mu_a^j
 \leq\frac{C_Pa}{1-\mu_a}\leq2C_Pa.
\]
Summing \eqref{eq:eta-step-full} gives
$\sum_{j=3}^k\eta_j\leq C_1C_Pd\Delta a/(1-\mu_a)$,
which is \eqref{eq:etasum}.

\emph{Taylor remainder and its accumulated bound.}
Put $T_k=\coef{R_{k,T}}_{\rho_k}$ and
$t_j=1+C\eta_j$. Composition and \eqref{eq:onestep} imply
\begin{equation}\label{eq:Trec}
 T_{k+1}\leq t_kT_k+C_Pa\mu_a^{k-2}.
\end{equation}
Iterating this scalar inequality, with an empty product equal to one,
gives
\begin{equation}\label{eq:Trec-solved}
 T_k\leq\left(\prod_{j=3}^{k-1}t_j\right)T_3
 +C_Pa\sum_{i=3}^{k-1}\mu_a^{i-2}
                         \prod_{j=i+1}^{k-1}t_j.
\end{equation}
To bound the accumulated products, set
\begin{equation}\label{eq:accumulated-flow-loss}
 \Theta_k=C\sum_{j=3}^{k-1}\eta_j
 \leq CC_1C_Pd\Delta a\,
           \frac{1-\mu_a^{k-3}}{1-\mu_a}
 \leq\frac{CC_1C_P}{Kd(1-\mu_a)}.
\end{equation}
Here $C$ is the fixed constant in $t_j=1+C\eta_j$. Since
$1+x\leq\ee^x$ for $x\geq0$, every product in
\eqref{eq:Trec-solved} is at most $\ee^{\Theta_k}$. Hence
\begin{align}\label{eq:Trec-geometric}
 T_k
 &\leq\ee^{\Theta_k}
       \left[T_3+C_Pa\sum_{i=3}^{k-1}\mu_a^{i-2}\right]\nonumber\\
 &\leq\ee^{\Theta_k}
       \left[M\left(\frac{4a}{R}\right)^{d-1}
       +C_Pa\mu_a\frac{1-\mu_a^{k-3}}{1-\mu_a}\right]
 \leq Ca.
\end{align}
This bound is uniform in $k,d,N$. We also verify the order under
composition. Since $\ord G_k\geq k$, its vector field has order
at least $k-1$, and the flow equation implies
$\Phi_k(u)=u+O(u^{k-1})$. If
$F=\sum_{m\geq d+1}F_m$ and $\Psi(u)=u+O(u^2)$, the Taylor
expansion of each homogeneous $F_m$ gives
\[
 F_m(\Psi(u))=F_m(u)
 +\sum_{\ell=1}^{m}\frac1{\ell!}
     D^\ell F_m(u)[\Psi(u)-u,\ldots,\Psi(u)-u].
\]
Each summand with $\ell\geq1$ has order at least
$(m-\ell)+2\ell=m+\ell$, because $D^\ell F_m$ has degree
$m-\ell$ and each of its $\ell$ arguments has order at least two.
Thus $\ord(F\circ\Psi)\geq d+1$. Applying this to
$R_{k,T}\circ\Phi_k$ and adding $B_{k+1,T}$, whose order is
given by \eqref{eq:Sk-finite-projections}, proves inductively that
$\ord R_{k,T}\geq d+1$.

\emph{High-frequency remainder and transport.}
For the coefficient norm set $e_k^\#=\coef{E_k}_{\rho_k}$.
The composition bound and the definition of $Q_{k+1,>}$ give
\[
 e_{k+1}^\#\leq t_ke_k^\#+C_Pa\mu_a^{k-2},\qquad
 e_3^\#\leq C_Pa.
\]
We expand this recurrence as
\begin{align}\label{eq:Ecoef-geometric}
 e_k^\#
 &\leq\left(\prod_{j=3}^{k-1}t_j\right)e_3^\#
    +C_Pa\sum_{i=3}^{k-1}\mu_a^{i-2}
                        \prod_{j=i+1}^{k-1}t_j\nonumber\\
 &\leq\ee^{\Theta_k}
      \left[e_3^\#+C_Pa\mu_a
                  \frac{1-\mu_a^{k-3}}{1-\mu_a}\right]
 \leq Ca.
\end{align}
For its smaller vector-field norm write
$e_k=\vf{E_k}_{\rho_k,s}$.
We use the high-factor estimate on $Q_{k+1,>}$ and the
canonical transport formula on $E_k\circ\Phi_k$. Pointwise,
\begin{equation}\label{eq:E-pointwise-transport}
 X_{E_k\circ\Phi_k}(u)
   =(D\Phi_k(u))^{-1}X_{E_k}(\Phi_k(u)),
   \qquad u\in B_s(\rho_{k+1}).
\end{equation}
Consequently, \eqref{eq:step-differentials} gives
\[
 \frac1{\rho_{k+1}}
 \sup_{B_s(\rho_{k+1})}\|X_{E_k\circ\Phi_k}\|_s
 \leq\frac{1+C\eta_k}{\rho_{k+1}}
                  \sup_{B_s(\rho_k)}\|X_{E_k}\|_s
 =\frac{\rho_k}{\rho_{k+1}}t_ke_k.
\]
On the other hand, \eqref{eq:new-projected-bounds} and
Lemma~\ref{lem:cut} imply
\[
 \vf{Q_{k+1,>}}_{\rho_{k+1},s}
 \leq C\ee^{-\nu f(N)}\coef{Q_{k+1,>}}_{\rho_{k+1}}
 \leq CC_Pa\ee^{-\nu f(N)}\mu_a^{k-2}.
\]
Adding these two estimates yields
\begin{equation}\label{eq:Erec}
 e_{k+1}\leq\frac{\rho_k}{\rho_{k+1}}t_ke_k
          +CC_Pa\ee^{-\nu f(N)}\mu_a^{k-2}.
\end{equation}
Iterating as above yields
\begin{align}\label{eq:Erec-solved}
 e_k&\leq\frac{\rho_3}{\rho_k}
             \left(\prod_{j=3}^{k-1}t_j\right)e_3\nonumber\\
 &\quad+CC_Pa\ee^{-\nu f(N)}
       \sum_{i=3}^{k-1}\mu_a^{i-2}
           \frac{\rho_{i+1}}{\rho_k}
                    \prod_{j=i+1}^{k-1}t_j.
\end{align}
The products of successive radius losses telescope:
\[
 \prod_{j=3}^{k-1}\frac{\rho_j}{\rho_{j+1}}
   =\frac{\rho_3}{\rho_k}\leq\frac43,
 \qquad
 \prod_{j=i+1}^{k-1}\frac{\rho_j}{\rho_{j+1}}
   =\frac{\rho_{i+1}}{\rho_k}\leq\frac43.
\]
Using \eqref{eq:accumulated-flow-loss} in
\eqref{eq:Erec-solved}, we therefore obtain
\begin{align}\label{eq:Erec-geometric}
 e_k
 &\leq\frac43\ee^{\Theta_k}
    \left[e_3+CC_Pa\ee^{-\nu f(N)}\mu_a
                \frac{1-\mu_a^{k-3}}{1-\mu_a}\right]\nonumber\\
 &\leq Ca\ee^{-\nu f(N)}.
\end{align}
In particular the loss is bounded independently of the number of
steps. For the order, $E_3$ starts at degree three,
$E_k\circ\Phi_k$ keeps this lower degree by the composition
calculation above, and $Q_{k+1,>}$ starts at degree $k+1$.
Thus $\ord E_k\geq3$ at every stage.

\emph{Composition and completion of the induction.}
We write the nested composition explicitly. For
$u\in B_s(\rho_k)$ set $v_k=u$ and
$v_j=\Phi_j(v_{j+1})$ for $j=k-1,\ldots,3$. The image
inclusions proved above give $v_j\in B_s(\rho_j)$ and
$\mathcal T^{(k)}(u)=v_3$. Hence
\begin{align}\label{eq:composition-telescope}
 \|\mathcal T^{(k)}(u)-u\|_s
  &\leq\sum_{j=3}^{k-1}\|v_j-v_{j+1}\|_s\nonumber\\
  &\leq\sum_{j=3}^{k-1}C\rho_j\coef{G_j}_{\rho_j}
   \leq CC_P\Delta a^2
          \frac{1-\mu_a^{k-3}}{1-\mu_a}
   \leq C\Delta a^2.
\end{align}
The chain rule and \eqref{eq:step-differentials} also yield
\begin{equation}\label{eq:composition-differentials}
 \max\{\|D\mathcal T^{(k)}(u)\|,
        \|(D\mathcal T^{(k)}(u))^{-1}\|\}
 \leq\prod_{j=3}^{k-1}(1+C\eta_j)
 \leq\ee^{\Theta_k},
\end{equation}
after choosing the fixed $C$ in \eqref{eq:accumulated-flow-loss}
to dominate the differential constants. Each $G_j$ is real, so
its flow preserves the real slice. Since $X_{G_j}(0)=DX_{G_j}(0)=0$,
the flow and chain-rule identities give
\[
 \mathcal T^{(k)}(0)=0,\qquad
 D\mathcal T^{(k)}(0)=\Id,
 \qquad
 (\mathcal T^{(k)})^*\Omega_{\rm symp}
   =\Phi_{k-1}^*\cdots\Phi_3^*\Omega_{\rm symp}
   =\Omega_{\rm symp},
\]
where $\Omega_{\rm symp}$ denotes the symplectic form fixed in
Section~\ref{sec:setting}. Together with
\eqref{eq:new-projected-bounds}, \eqref{eq:Trec-geometric} and
\eqref{eq:Erec-geometric}, these identities close the induction.
At $k=d+1$ the projection onto degrees $d+1,\ldots,d$ is empty,
so $P_{d+1}=0$ and all the asserted terminal estimates hold.
\end{proof}

\section{Normal form lemma and parameter alignment}\label{sec:normal}
\subsection{The two truncations and their balance}
There are two truncations. The integer $d$ is the largest retained
Hamiltonian degree, and $N$ separates low and high Fourier factors.
Terms containing at least three factors above $N$ gain
$\exp(-\nu f(N))$ in the nonlinear vector field. For the remaining
monomials, momentum and frequency separation reduce homological
division to a finite cutoff enlarged by a fixed power of $dN$.
This gives
\[
 \Delta_{d,N}=\gamma^{-d}(C_{\rm div}dN)^{C_{\rm div}d^p},
\]
as a common upper bound for the divisors' reciprocals through degree
$d$. A buffer cutoff $M_{d,N}$ isolates the two-high normal terms.
Their high equations preserve block masses, while their low outputs
still have the exponential high-frequency gain.

The construction keeps a polynomial to be normalized, a cumulative
normal polynomial, a Taylor remainder and a transported high-frequency
remainder at every step. On the auxiliary ball
$a=(Kd^2\Delta_{d,N})^{-1}$ the polynomial decreases geometrically.
Its Taylor remainder has degree at least $d+1$, and restriction to
the physical ball produces the factor
$r(Cd^2\Delta_{d,N}r)^{d-2}$. The other remainder is transported using
the differential of the canonical map. Its original high-frequency
gain is retained in norm, and all radius ratios telescope. This gives
the two quantities to be balanced:
\[
 (d-2)\{L-\log(Cd^2\Delta_{d,N})\},\qquad \nu f(N),
 \qquad L=\log(1/r).
\]
Choosing $d^p\log(dN)$ to be a sufficiently small multiple of $L$ and
$f(N)$ comparable to $dL$ yields the two scales stated above.

\subsection{The normal form on the physical ball}
\begin{theorem}[Normal form before optimization]\label{thm:nf}
There are fixed $C,K>1$ such that, with $\Delta=\Delta_{d,N}$ and
$a=(Kd^2\Delta)^{-1}$ as in \eqref{eq:auxradius},
there is a real analytic canonical map $\mathcal T_d$, constructed
for these fixed $d,N$ on the auxiliary ball, with the following
simultaneous estimates for every $0<r\leq a/4$. The map and its
inverse are defined on $B_s(2r)$, are tangent to the identity, and
\begin{equation}\label{eq:nf}
 H\circ\mathcal T_d=H_0+Z_d+R_T+E_d.
\end{equation}
The polynomial $Z_d$ has degrees $3,\ldots,d$ and the stated buffered normal form. Moreover,
\begin{align}
 \sup_{B_s(2r)}\nrm{\mathcal T_d^{\pm1}-\Id}_s&\leq C\Delta r^2,\label{eq:nearid}\\
 \vf{R_T}_{2r,s}&\leq Cr(Cd^2\Delta r)^{d-2},\label{eq:taylorfinal}\\
 \vf{E_d}_{2r,s}&\leq Cr\ee^{-\nu f(N)},\label{eq:highfinal}\\
 \coef{Z_d}_{2r}&\leq Cr.\label{eq:Zfinal}
\end{align}
In particular the Taylor bound also implies the useful coarser estimate
\begin{equation}\label{eq:originalscale}
 \vf{R_T}_{2r,s}\leq C r^{d-2}\gamma^{-d^2}(CdN)^{Cd^{p+1}}.
\end{equation}
The nonexact block part $Z_{\rm blk}$ has no output in $N<|j|\leq M_{d,N}$, and its low output satisfies
\begin{equation}\label{eq:blocklow}
 \sup_{B_s(2r)}\nrm{\Pi_{\leq N}X_{Z_{\rm blk}}}_s\leq Cr^2\ee^{-\nu f(N)}.
\end{equation}
\end{theorem}

\begin{proof}
Apply Lemma~\ref{lem:iteration-full} at $k=d+1$. Set
$\mathcal T_d=\mathcal T^{(d+1)}$, $Z_d=Z^{(d+1)}$,
$R_T=R_{d+1,T}$ and $E_d=E_{d+1}$.
Thus the subscript on $E_d$ denotes the terminal normalization degree.
The lemma proves \eqref{eq:nf} on $B_s(3a)$ and gives
\begin{equation}\label{eq:terminal-auxiliary}
 \coef{Z_d}_{3a}+\coef{R_T}_{3a}+\coef{E_d}_{3a}\leq Ca,
 \quad \ord R_T\geq d+1,
 \quad \vf{E_d}_{3a,s}\leq Ca\ee^{-\nu f(N)}.
\end{equation}
It also gives $\sup_{B_s(3a)}\|\mathcal T_d-\Id\|_s
\leq C\Delta a^2$.

We first construct the inverse on a common ball. Put
$F=\mathcal T_d-\Id$. Cauchy's estimate with margin $a/2$ gives
\[
 \sup_{B_s(5a/2)}\|DF\|
 \leq\frac2a\sup_{B_s(3a)}\|F\|_s\leq C\Delta a\leq\frac14,
 \qquad
 \sup_{B_s(3a)}\|F\|_s\leq C\Delta a^2\leq a/4,
\]
after the fixed enlargement of $K$. For $y\in B_s(2a)$ and
$w,v\in\overline B_s(5a/2)$,
\[
 \|y-F(w)\|_s<2a+a/4<5a/2,
 \qquad
 \|(y-F(w))-(y-F(v))\|_s\leq\tfrac14\|w-v\|_s.
\]
Starting with $w_0=y$ and $w_{n+1}=y-F(w_n)$, we obtain
\[
 \|w_{n+1}-w_n\|_s\leq4^{-n}\|F(y)\|_s,
 \qquad
 \|w_{n+r}-w_n\|_s\leq\tfrac43\,4^{-n}\|F(y)\|_s.
\]
Thus $w_n\to w(y)$ uniformly on smaller $y$-balls, with
$w(y)=y-F(w(y))$ and $\mathcal T_d(w(y))=y$. Each iterate is
holomorphic; the locally uniform limit is holomorphic by the same
Cauchy calculation as in Appendix~\ref{app:lie-calculus}. Uniqueness
follows from the contraction inequality, and
\[
 \mathcal T_d^{-1}(y)=w(y),\qquad
 \sup_{B_s(2a)}\|\mathcal T_d^{-1}-\Id\|_s
 \leq\sup_{B_s(5a/2)}\|F\|_s\leq C\Delta a^2.
\]
Since $F(0)=DF(0)=0$, the identities for $w$ also give
$\mathcal T_d^{-1}(0)=0$ and $D\mathcal T_d^{-1}(0)=\Id$.

Now let $0<r\leq a/4$. For a coefficient series of order $h$,
Lemma~\ref{lem:scalar-evaluation} gives
\begin{equation}\label{eq:homoscale}
 \coef F_t\leq(t/u)^{h-2}\coef F_u\qquad(0<t<u).
\end{equation}
Use $h=d+1$, $u=3a$, $t=2r$ in the Taylor estimate:
\begin{align}\label{eq:T-physical-derivation}
 \vf{R_T}_{2r,s}
 &\leq C\coef{R_T}_{2r}
 \leq Ca\left(\frac{2r}{3a}\right)^{d-1}\nonumber\\
 &=\frac{2C}{3}r\left(\frac{2r}{3a}\right)^{d-2}
 \leq Cr(Cd^2\Delta r)^{d-2}.
\end{align}
The field $X_{E_d}$ is holomorphic and has a zero of order at least
two. Lemma~\ref{lem:radial-full} therefore implies
\begin{align}\label{eq:E-physical-derivation}
 \sup_{B_s(2r)}\|X_{E_d}\|_s
 &\leq\left(\frac{2r}{3a}\right)^2
                  \sup_{B_s(3a)}\|X_{E_d}\|_s\nonumber\\
 &\leq\left(\frac{2r}{3a}\right)^2Ca^2\ee^{-\nu f(N)}
 \leq Cr^2\ee^{-\nu f(N)}.
\end{align}
Division by $2r$ proves \eqref{eq:highfinal}. For the maps and
the normal polynomial, Lemma~\ref{lem:radial-full} and
\eqref{eq:homoscale} give
\begin{align*}
 \sup_{B_s(2r)}\|\mathcal T_d-\Id\|_s
 &\leq(2r/(3a))^2C\Delta a^2\leq C\Delta r^2,\\
 \sup_{B_s(2r)}\|\mathcal T_d^{-1}-\Id\|_s
 &\leq(r/a)^2C\Delta a^2\leq C\Delta r^2,\\
 \coef{Z_d}_{2r}&\leq\frac{2r}{3a}\coef{Z_d}_{3a}\leq Cr.
\end{align*}
These are \eqref{eq:nearid} and \eqref{eq:Zfinal}.

For comparison with the stepwise estimates, put
$\widehat\rho_k=(r/a)\rho_k$. Since $P_k,G_k$ have order at
least $k$, the auxiliary bounds imply
\begin{align}\label{eq:physicalPk}
 \coef{P_k}_{\widehat\rho_k}
 &\leq(r/a)^{k-2}C_Pa\mu_a^{k-3}
   =C_Pr(C_*d^2\Delta r)^{k-3},\nonumber\\
 \coef{G_k}_{\widehat\rho_k}
 &\leq C_P\Delta r(C_*d^2\Delta r)^{k-3}.
\end{align}
Expanding the first bound displays all degree and cutoff factors:
\begin{equation}\label{eq:physicalPk-expanded}
 \coef{P_k}_{\widehat\rho_k}
 \leq C_PC_*^{k-3}d^{2k-6}r^{k-2}
 \gamma^{-d(k-3)}
 (C_{\rm div}dN)^{C_{\rm div}(k-3)d^p}.
\end{equation}
For the generator the corresponding expression is
\[
 \coef{G_k}_{\widehat\rho_k}
 \leq C_PC_*^{k-3}d^{2k-6}r^{k-2}
    \gamma^{-d(k-2)}(C_{\rm div}dN)^{C_{\rm div}(k-2)d^p}.
\]
Similarly, \eqref{eq:T-physical-derivation} gives
\begin{align*}
 \vf{R_T}_{2r,s}
 &\leq Cr^{d-1}(Cd^2)^{d-2}\Delta^{d-2}\\
 &=Cr^{d-1}(Cd^2)^{d-2}\gamma^{-d(d-2)}
                   (C_{\rm div}dN)^{C_{\rm div}(d-2)d^p}\\
 &\leq Cr^{d-2}\gamma^{-d^2}(CdN)^{Cd^{p+1}}.
\end{align*}
For the last inequality we use $r<1$, $d(d-2)\leq d^2$, and
\[
 (d-2)\log(Cd^2)
 +C_{\rm div}(d-2)d^p\log(C_{\rm div}dN)
 \leq Cd^{p+1}\log(CdN).
\]
This proves \eqref{eq:originalscale}.

Every monomial of $Z_{\rm blk}$ has two factors above $M_{d,N}$
and its other factors below $N$. It consequently has no output in
$N<|j|\leq M_{d,N}$. Its low output has both high factors still
present. Applying \eqref{eq:lowgain} and \eqref{eq:Zfinal} gives
$\sup_{B_s(2r)}\|\Pi_{\leq N}X_{Z_{\rm blk}}\|_s
\leq Cr^2\ee^{-\nu f(N)}$, as required.
\end{proof}

\subsection{Choice of the degree and the cutoff}
The dependence on the order is now explicit enough to identify the
admissible normalization degree. On the auxiliary ball,
\begin{equation}\label{eq:decisive-order-loss}
 \mu_a=C_*d^2\Delta a=C_*/K,\qquad
 \sum_{k=3}^d\eta_k\leq Cd\Delta a\leq C/(Kd),\qquad
 \prod_{k=3}^d\frac{\rho_k}{\rho_{k+1}}\leq4/3.
\end{equation}
Thus the accumulated transport cost is at most
$(4/3)\exp(C/(Kd))$. Restriction from $a=(Kd^2\Delta)^{-1}$
to $2r$ then gives exactly \eqref{eq:taylorfinal} and
\eqref{eq:highfinal}: the factor $\Delta$ enters the auxiliary
radius once, while the Taylor remainder acquires the power $d-2$
only after restriction. These two steps lead to the simultaneous
requirements
\begin{equation}\label{eq:decisive-balances}
 \log\Delta=O\!\left(d^p\log(dN)+d|\log\gamma|\right)
       \leq cL,\qquad f(N)\asymp dL,
 \qquad L=\log(1/r).
\end{equation}
For $p=1$, the first condition allows $d\asymp L/\log L$
in Gevrey regularity. Thus the normalization order reaches
$L/\log L$; fixed powers $L^\alpha$, $\alpha<1$, occur in the
displayed estimates of \cite{FG13,BMP20}.
An additional factor of $d$ in the logarithmic divisor loss would
change the first balance to $d^{p+1}\log(dN)\lesssim L$, giving
\[
 d_G\asymp(L/\log L)^{1/(p+1)},\qquad
 d_U\asymp L^{(q-1)/((p+1)q+1)},
\]
and hence strictly smaller exponents $d_GL$ and $d_UL$.
The transported vector-field norm and the geometric recurrence in
Section~\ref{sec:iteration} are what allow
\eqref{eq:decisive-balances} with the original arithmetic exponent
$p$. The next proposition solves these two constraints, including
all smallness conditions.

\begin{proposition}[Admissible parameters]\label{prop:parameters}
Let $L=\log(1/r)$. If
\begin{equation}\label{eq:smallbalance}
 Cd^p\log(CdN)+d|\log\gamma|+C\log d\leq L/4,
\end{equation}
then the two remainders satisfy
\[
 \vf{R_T}_{2r,s}+\vf{E_d}_{2r,s}
 \leq Cr\,[\exp(-cdL)+\exp(-\nu f(N))]
 \leq2Cr\exp\!\left(-\min\{cdL,\nu f(N)\}\right).
\]

For Gevrey weights one may take
\begin{equation}\label{eq:Gparameters}
 d=\lfloor a_0(L/\log L)^{1/p}\rfloor,\qquad
 N=\lceil(b_0dL)^{1/\theta}\rceil.
\end{equation}
For logarithmic weights one may take
\begin{equation}\label{eq:Uparameters}
 d=\lfloor a_0L^{(q-1)/(pq+1)}\rfloor,\qquad
 N=\left\lceil\exp((b_0dL)^{1/q})\right\rceil.
\end{equation}
Here $b_0>0$ is fixed and $a_0$ sufficiently small depending on $b_0$ and the fixed data. These choices satisfy $N\geq d$, \eqref{eq:smallbalance}, $r\leq a/4$, and $f(N)\asymp dL\asymp\Phi(L)$.
\end{proposition}
\begin{proof}
We first insert the divisor loss into the Taylor estimate:
\[
 \log(Cd^2\Delta r)
 \leq-L+Cd^p\log(CdN)+d|\log\gamma|+C\log d
 \leq-3L/4.
\]
Since $d-2\geq d/2$ for $d\geq4$, it follows that
\[
 (Cd^2\Delta r)^{d-2}
 \leq\exp\!\left(-\frac34(d-2)L\right)
 \leq\exp(-3dL/8).
\]
Together with \eqref{eq:highfinal}, this gives the stated remainder
bound. Taking the fixed $C$ large enough to include $4K$ also gives
$\log(4Kd^2\Delta)\leq L$, which is $r\leq a/4$.

For \eqref{eq:Gparameters}, direct substitution gives
$\log(dN)\asymp\log L$ and $d^p\log(dN)\leq Ca_0^pL$.
For \eqref{eq:Uparameters}, set
$\alpha=(q-1)/(pq+1)$ and $\mathfrak h=(b_0dL)^{1/q}$.
The ceiling in $N$ satisfies
\[
 0\leq\log N-\mathfrak h
 \leq\log(1+\ee^{-\mathfrak h})\leq\ee^{-\mathfrak h}.
\]
Consequently
\[
 d^p\log N
 \leq b_0^{1/q}d^{p+1/q}L^{1/q}+d^p\ee^{-\mathfrak h}
 \leq Cb_0^{1/q}a_0^{p+1/q}L,
 \qquad p\alpha+(1+\alpha)/q=1.
\]
The terms $d|\log\gamma|+C\log d$ are $o(L)$ in both cases. The assertions follow by reducing $a_0$. All integer-part errors are bounded relative errors, since $d,N\to\infty$. In the logarithmic case, substitution also gives $\log N\asymp d^{(p+1)/(q-1)}$, which also expresses the degree-cutoff relation directly.
\end{proof}

\subsection{Alignment of all smallness requirements}
The logarithm entering the Taylor remainder is
\begin{equation}\label{eq:loss-log-complete}
 \log(Cd^2\Delta r)
 =-L+d|\log\gamma|
     +C_{\rm div}d^p\log(C_{\rm div}dN)+2\log d+\log C.
\end{equation}
The restriction $r\leq a/4$ is the inequality
$\log(4Kd^2\Delta)\leq L$. The flow smallness on the auxiliary
ball is instead $C_*d^2\Delta a=C_*/K\leq1/32$ and has already
been secured by a fixed choice of $K$. Thus it remains to choose
$d,N$ so that \eqref{eq:loss-log-complete} is negative by a fixed
fraction of $L$. Under \eqref{eq:smallbalance},
\[
 (Cd^2\Delta r)^{d-2}\leq
 \exp\left(-\frac34(d-2)L\right)
 \leq\exp\left(-\frac38dL\right).
\]
At the same time $r\leq a/4$, $\Delta r\to0$, and
$\Delta r^2/r\leq C\exp(-3L/4)$, after decreasing the initial-radius
threshold if necessary. In particular the coordinate changes have
relative displacement tending to zero uniformly over the data ball.
The balanced size to be used in the dynamics is
\begin{equation}\label{eq:balanced-remainder}
 \mathfrak e_{d,N}(r)
 =(Cd^2\Delta r)^{d-2}+\ee^{-\nu f(N)}
 \leq2\exp\left(-\min\{3dL/8,\nu f(N)\}\right).
\end{equation}
The same choices of $d,N$ therefore satisfy the normal-form
domain condition and control both remainders.

\subsection{Explicit Gevrey degree--cutoff inversion}
The parameter choices in Proposition~\ref{prop:parameters} can also
be obtained from an exact balance equation. Fix positive constants
$A_0,B_0'$ and solve, initially for real $d,N$,
\begin{equation}\label{eq:exact-G-balance}
 L=A_0d^p\log(dN),\qquad
 N^\theta=B_0'd^{p+1}\log(dN).
\end{equation}
$A_0$ may be made as large as required by
\eqref{eq:smallbalance}; the second constant fixes the relative
sizes of the two remainders. Set $Y=\theta\log(dN)$.
Substitution of $N=\ee^{Y/\theta}/d$ gives successively
\begin{align}\label{eq:Lambert-negative-complete}
 \ee^Y&=\frac{B_0'}{\theta}d^{p+\theta+1}Y,
 &Y\ee^{-Y}&=\frac{\theta}{B_0'}d^{-(p+\theta+1)},\nonumber\\
 Y&=-W_{-1}\!\left(-\frac{\theta}{B_0'}d^{-(p+\theta+1)}\right),
 &N&=d^{-1}\exp(Y/\theta),\\
 L&=\frac{A_0}{\theta}d^pY,
 &N^\theta&=\frac{B_0'}{A_0}dL.\nonumber
\end{align}
The negative branch is chosen because $Y\to\infty$.
Lemma~\ref{lem:Lambert-both} below gives
\[
 Y=(p+\theta+1)\log d+\log\log d+O(1).
\]
Writing $B=A_0(p+\theta+1)/\theta$ yields
$L=Bd^p\log d(1+o(1))$ and therefore
\[
 d=(p/B)^{1/p}(L/\log L)^{1/p}(1+o(1)),\qquad
 N^\theta\asymp\frac{L^{1+1/p}}{(\log L)^{1/p}}.
\]
This proves the endpoint logarithmic factor for every fixed
$p\geq1$. It also recovers the comparison with the weaker powers:
for $0<a<1/p$,
\[
 \frac{N^\theta}{L^{1+a}}
 \asymp\frac{L^{1/p-a}}{(\log L)^{1/p}}\longrightarrow\infty.
\]
In particular at $p=1$ the endpoint scale is $L^2/\log L$.
Integer choices are supplied by \eqref{eq:Gparameters}, so the exact
continuous inversion is compatible with the integer normal-form
construction.

\subsection{Explicit logarithmic degree--cutoff inversion}
For the logarithmic weight set $H=\log N$. At large $N$,
$f(N)=(H+O(N^{-1}))^q=H^q(1+o(1))$.
For this weight the continuous balance has the form
\[
 L=A_0d^p(H+\log d),\qquad
 H^q=B_0'd^{p+1}(H+\log d).
\]
The second relation implies $H\geq c d^{(p+1)/(q-1)}$.
Consequently $\log d/H\to0$; division by $H$ then gives
\begin{align}\label{eq:log-inversion-complete}
 H^{q-1}&=B_0'd^{p+1}(1+o(1)),\nonumber\\
 L&=A_0(B_0')^{1/(q-1)}
              d^{p+(p+1)/(q-1)}(1+o(1)),\\
 d&\asymp L^{(q-1)/(pq+1)},\qquad
 H^q\asymp dL\asymp L^{1+(q-1)/(pq+1)}.\nonumber
\end{align}
There is no factor of $\log d$ at the last scale. For the actual
integer choices \eqref{eq:Uparameters}, put
$\alpha=(q-1)/(pq+1)$ and $\mathfrak h=(b_0dL)^{1/q}$.
Then $\log N=\mathfrak h+O(\ee^{-\mathfrak h})$ and
\[
 d^p\log N\leq Cb_0^{1/q}a_0^{p+1/q}L,
 \quad d^p\log d=o(L),\quad d|\log\gamma|=o(L),
\]
because $p\alpha+(1+\alpha)/q=1$ and $p\alpha<1$.
These estimates verify \eqref{eq:smallbalance} upon reducing $a_0$.
For the Gevrey choices, $\log(dN)\leq C\log L$ and the same
verification gives $d^p\log(dN)\leq Ca_0^pL$.
In both cases $N/d\to\infty$, so the cutoff requirement $N\geq d$
also holds for every sufficiently small radius.

\section{Stability time}\label{sec:stability}
We now turn the normal-form estimates into a bound for the original
solutions. After proving conjugacy on the domain of the unbounded
linear generator, we use the block structure of the retained terms
to estimate the full physical norm. Continuation on a ball whose
radius is a fixed multiple of the initial norm then gives the
stability interval. Appendix~\ref{app:component-dynamics} supplies
the separate low-mode estimate and the high-mode isometric propagator.
\subsection{The linear generator and conjugacy of solution flows}\label{subsec:flow}
We justify the Hamiltonian changes of coordinates for the semilinear
flow. On the doubled space define
\[
 \mathscr L u=(-\sigma\ii\omega_j u_J)_J,\qquad
 D(\mathscr L)=\left\{u:\tfrac12\sum_{j,\sigma}
                       w_s(j)^2\omega_j^2|u_J|^2<\infty\right\}.
\]
For either weight, the bound
\[
 |H_0(u)|\leq\sup_j\frac{|\omega_j|}{w_s(j)^2}\|u\|_s^2
\]
is finite because $f(x)/\log x\to\infty$. Thus $H_0$ is a
continuous quadratic Hamiltonian on this phase space, and its
Hamiltonian vector field is represented by the unbounded generator
$\mathscr L$ on the displayed domain.

For $U(t)u=(\ee^{-\sigma\ii\omega_jt}u_J)_J$,
\begin{align*}
 \|U(t)u\|_s^2
 &=\tfrac12\sum_{j,\sigma}w_s(j)^2
                   |\ee^{-\sigma\ii\omega_jt}u_J|^2=\|u\|_s^2,\\
 \|U(t)u-u\|_s^2
 &=\tfrac12\sum_{j,\sigma}w_s(j)^2
           |\ee^{-\sigma\ii\omega_jt}-1|^2|u_J|^2\longrightarrow0.
\end{align*}
The second limit uses
$|\ee^{-\sigma\ii\omega_jt}-1|^2\leq4$ and
$\sum_{j,\sigma}w_s(j)^2|u_J|^2=2\|u\|_s^2<\infty$.
For $u\in D(\mathscr L)$,
\[
 \left|\frac{\ee^{-\sigma\ii\omega_jt}-1}{t}\right|
 \leq|\omega_j|,\qquad
 \frac{U(t)u-u}{t}\longrightarrow
       (-\sigma\ii\omega_j u_J)_J=\mathscr Lu\quad\text{in }W_s.
\]
The coordinate error tends to zero and satisfies
\[
 \left|\frac{\ee^{-\sigma\ii\omega_jt}-1}{t}
                         +\sigma\ii\omega_j\right|^2|u_J|^2
 \leq4\omega_j^2|u_J|^2,\qquad
 2\sum_{j,\sigma}w_s(j)^2\omega_j^2|u_J|^2=4\|\mathscr Lu\|_s^2<\infty.
\]
Conversely, if $(U(t)u-u)/t\to v$ in $W_s$, then
\[
 v_J=-\sigma\ii\omega_j u_J,\qquad
 \tfrac12\sum_{j,\sigma}w_s(j)^2\omega_j^2|u_J|^2=\|v\|_s^2<\infty.
\]
Thus $D(\mathscr L)$ is exactly the generator domain.

Let $B$ be a nonlinear vector field analytic on a ball $B_s(R_*)$.
For $0<r_*<\rho<R_*$ put
$M_\rho=\sup_{B_s(\rho)}\|B\|_s$ and
$L_\rho=\sup_{B_s(\rho)}\|DB\|$; these are finite for the
fields used here by the proved coefficient bounds. Given
$\|u_*\|_s\leq r_*$ at time $t_0$, consider the closed set
\[
 \mathcal X=\left\{v\in C([t_0-\tau,t_0+\tau],W_s):
 \sup_t\|v(t)-U(t-t_0)u_*\|_s\leq\rho-r_*\right\}.
\]
For a curve $z$, define the Duhamel map by
\begin{equation}\label{eq:Duhamel-full}
 (\Gamma z)(t)=U(t-t_0)u_*+
                \int_{t_0}^t U(t-\xi)B(z(\xi))\,d\xi.
\end{equation}
Isometry of $U$ gives
\[
 \|\Gamma z-U(\cdot-t_0)u_*\|_{C_tW_s}\leq\tau M_\rho,
 \quad\|\Gamma z-\Gamma\widetilde z\|_{C_tW_s}
       \leq\tau L_\rho\|z-\widetilde z\|_{C_tW_s}.
\]
Thus $\tau\leq\min\{(\rho-r_*)/M_\rho,1/(2L_\rho)\}$
makes $\Gamma$ a self-map and a contraction. A zero denominator
imposes no restriction in this minimum. The fixed point is the
unique local mild solution. The same estimates give, on any common
interval contained in the $\rho$-ball,
\begin{equation}\label{eq:mild-cont-dependence}
 \|u(t)-v(t)\|_s\leq\ee^{L_\rho|t-t_0|}\|u(t_0)-v(t_0)\|_s.
\end{equation}
They also give a uniform time for restarting the solution at any
point of a fixed smaller ball. This proves the continuation criterion
used below.

Equip $D(\mathscr L)$ with
$\|u\|_{\mathscr D}=\|u\|_s+\|\mathscr L u\|_s$.
Frequency growth and the finitely many low indices make this norm
equivalent to $\|u\|_s+\|\br j^\beta u_J\|_s$.
The diagonal group is isometric for the graph norm and strongly
continuous there, because both $U(t)u$ and
$\mathscr L U(t)u=U(t)\mathscr L u$ converge in $W_s$.
For a momentum-preserving $m$-homogeneous vector field,
\[
 \br {j_{\rm out}}^\beta
 \leq\left(\sum_{i=1}^{m-1}\br {j_i}\right)^\beta
 \leq(m-1)^{\beta-1}\sum_{i=1}^{m-1}\br {j_i}^\beta.
\]
Apply the weighted convolution estimate separately to each input
carrying the additional polynomial weight. The other inputs use
exactly \eqref{eq:weight-ratio-full}. Summing over the input positions
and the differentiated coefficient gives
\begin{equation}\label{eq:graph-homogeneous}
 \|X_{F_m}(u)\|_{\mathscr D}
 \leq C m^{\beta+3}b^{m-2}c_m(F)
                  \|u\|_s^{m-2}\|u\|_{\mathscr D}.
\end{equation}
For $0<r<R_F$, where $\coef F_{R_F}<\infty$, put
\[
 A_F(r)=C\sum_{m\geq3}m^{\beta+4}b^{m-2}c_m(F)r^{m-2}<\infty.
\]
For $r<r'<R_F$, the comparison is
\[
 A_F(r)\leq C\left(\sup_{m\geq3}m^{\beta+1}(r/r')^{m-2}\right)
                    \coef F_{r'}<\infty.
\]
Differentiating one input in \eqref{eq:graph-homogeneous} gives,
for $\|u\|_s\leq r$ and $u,h\in D(\mathscr L)$,
\begin{align}\label{eq:graph-tame-differential}
 \|X_F(u)\|_{\mathscr D}&\leq A_F(r)\|u\|_{\mathscr D},\nonumber\\
 \|DX_F(u)h\|_{\mathscr D}
 &\leq A_F(r)\left(\|h\|_{\mathscr D}
          +r^{-1}\|u\|_{\mathscr D}\|h\|_s\right).
\end{align}
In the second line the polynomial weight falls either on the
replaced input $h$ or on one of the remaining inputs $u$.
Integration along the segment from $v$ to $u$ consequently yields
\begin{align}\label{eq:graph-difference}
 \|X_F(u)-X_F(v)\|_{\mathscr D}
 \leq A_F(r)\bigl[&\|u-v\|_{\mathscr D}\nonumber\\
 &+r^{-1}(\|u\|_{\mathscr D}+\|v\|_{\mathscr D})\|u-v\|_s\bigr],
\end{align}
when $\|u\|_s,\|v\|_s\leq r$. Thus the field is locally
Lipschitz in the graph norm. The Duhamel contraction in that norm
constructs graph-regular solutions. Its integral takes values in
$D(\mathscr L)$ continuously, so differentiation in $W_s$ proves
that these solutions are classical.

On an interval where $\|u(t)\|_s\leq r$, the graph-space
Duhamel equation and \eqref{eq:graph-tame-differential} give
\[
 \|u(t)\|_{\mathscr D}
 \leq \|u(0)\|_{\mathscr D}
       +A_F(r)\int_0^{|t|}\|u(\operatorname{sgn}(t)v)\|_{\mathscr D}\,dv
 \leq e^{A_F(r)|t|}\|u(0)\|_{\mathscr D}.
\]
The last inequality follows by multiplying the integral upper
bound by $e^{-A_F(r)|t|}$ and differentiating. This prevents
loss of graph regularity on a compact interval inside the $r$-ball.
The same argument for the autonomous fields $X_{G_k}$ and
$-X_{G_k}$ proves graph preservation for each component flow.
Their differentials preserve the graph space by the variational
equation and the second bound in \eqref{eq:graph-tame-differential}.

Graph preservation of the inverse normalizing map can now be checked
on its full common domain. For $y\in B_s(2a)$, apply the flows
$\Phi_3^{-1},\Phi_4^{-1},\ldots,\Phi_d^{-1}$ in that order.
Their displacements have total size at most
$C\Delta a^2\sum_{k=3}^d\mu_a^{k-3}<a/4$.
Every intermediate point therefore lies in $B_s(9a/4)$, inside
all the component domains $B_s(\rho_{k+1})$ since
$\rho_{k+1}\geq3a$.
Cancellation of successive inverse flows gives
\begin{equation}\label{eq:inverse-flow-graph}
 \mathcal T_d^{-1}
   =\Phi_d^{-1}\circ\cdots\circ\Phi_3^{-1}
          \quad\hbox{on }B_s(2a).
\end{equation}
Indeed the right side is a preimage under
$\mathcal T_d=\Phi_3\circ\cdots\circ\Phi_d$ and lies in the
ball where the inverse constructed in Theorem~\ref{thm:nf} is unique.
It preserves $D(\mathscr L)$, as does its differential, by the
componentwise graph estimates.

On $D(\mathscr L)$, differentiation along $\Phi_k^t$ proves
$d(H_0\circ\Phi_k^t)/dt=\{H_0,G_k\}\circ\Phi_k^t$.
The homological equation replaces this bracket by the bounded
Hamiltonian $Z_k^*-P_k$, so integration gives exactly
\eqref{eq:H0-Lie-expanded}. For $w\in D(\mathscr L)$ in the common domain, the canonical
chain rule gives the vector identity
\begin{equation}\label{eq:classical-conjugacy-explicit}
 D\mathcal T_d(w)\bigl[\mathscr Lw+X_{Z_d+R_T+E_d}(w)\bigr]
 =\mathscr L\mathcal T_d(w)+X_P(\mathcal T_d(w)).
\end{equation}
Both sides belong to $W_s$ by graph preservation. To verify the
identity, pair each side with $D\mathcal T_d(w)h$ in the
symplectic form, use canonicity and differentiate
$H\circ\mathcal T_d=H_0+Z_d+R_T+E_d$ in the direction $h$.
The resulting pairings agree for every $h\in W_s$; invertibility
of $D\mathcal T_d$ and nondegeneracy of $\Omega_{\rm symp}$ give equality.
The ordinary chain rule in time proves conjugacy of the classical
flows for every fixed finite normalization.
Take graph-regular approximations of any initial datum. Inequality
\eqref{eq:mild-cont-dependence}, first on a local interval and then
on finitely many successive intervals, proves uniform convergence
of their mild solutions on each common compact existence interval.
The analytic maps and their inverses are continuous in $W_s$.
Passing to the limit proves the same flow conjugacy for the mild
solutions. All quantitative $W_s$ estimates used in the stability
theorem therefore hold for these mild solutions.

\subsection{The weighted energy identity}
We estimate the full physical norm through its quadratic energy.
Appendix~\ref{app:component-dynamics} gives the corresponding
low-mode estimate and the isometric high-mode propagator.
Write $Z_d=Z_{\rm res}+Z_{\rm blk}$, placing all exact resonances in the first term, and set $M_N=M_{d,N}$. The block energy is
\begin{equation}\label{eq:energy}
 \mathcal E_{s,N}(w)=\sum_{|j|\leq M_N}\ee^{2sf(\br j)}|w_j|^2
 +\sum_{n>M_N}\ee^{2sf(\br n)}(|w_n|^2+|w_{-n}|^2).
\end{equation}
In the one-dimensional setting it satisfies the exact identity
\begin{equation}\label{eq:energyequiv}
 \mathcal E_{s,N}(w)=\nrm w_s^2.
\end{equation}
The energy is independent of the split at $M_N$.
For a $W_s$-valued Hamiltonian field $X_F$, its bracket with this
energy is the continuous directional derivative
\begin{equation}\label{eq:energy-directional-bracket}
 \{\mathcal E_{s,N},F\}(w):=d\mathcal E_{s,N}(w)[X_F(w)]
   =2\operatorname{Re}\sum_j w_s(j)^2\bar w_j(X_F(w))_j.
\end{equation}
Cauchy--Schwarz gives
\[
 2\sum_jw_s(j)^2|w_j||(X_F(w))_j|
 \leq2\|w\|_s\|X_F(w)\|_s<\infty,
 \qquad \mathcal E_{s,N}(U(t)w)=\mathcal E_{s,N}(w).
\]

\begin{lemma}[Commutation and the high-mode action equations]\label{lem:energy}
The Hamiltonian $H_0+Z_{\rm res}$ commutes with \eqref{eq:energy}. The output of $Z_{\rm blk}$ above $M_N$ makes no contribution to its derivative. Consequently, on $B_s(2r)$,
\begin{equation}\label{eq:energydiff}
 \left|\frac{d}{dt}\sqrt{\mathcal E_{s,N}(w(t))}\right|
 \leq Cr^2\left[(Cd^2\Delta r)^{d-2}+\ee^{-\nu f(N)}\right].
\end{equation}
\end{lemma}
\begin{proof}
For an exact-resonant monomial $w^{\mathcal J}$ the directional
bracket contributes the factor
\[\ii\sum_{i=1}^m\sigma_i w_s(j_i)^2.\]
The opposite-sign pairing at equal absolute indices makes this
sum zero. Applying Lemma~\ref{lem:cut} to the absolute coefficient
series justifies summing these cancellations in
\eqref{eq:energy-directional-bracket}. This proves commutation at
all indices, including the pair $n,-n$.

Every monomial of $Z_{\rm blk}$ has two high factors in the same
block and with opposite signs. Put $w^<=\Pi_{\leq N}w$.
On the real slice we obtain the representation
\[
 Z_{\rm blk}=\sum_{n>M_N}\sum_{j,k\in\mathcal B_n}
           \bar w_j A_{jk}(w^<)w_k,\qquad A_{jk}=\overline{A_{kj}}.
\]
The Hermitian identity follows by conjugating and collecting each
monomial. If $j\ne k$ in $\mathcal B_n$, then $j=-k$ and momentum
would require
\[
 2n=|j-k|\leq(d-2)N<M_N,
\]
which contradicts $n>M_N$. Hence $A_{jk}=0$ for $j\ne k$.
Since $A_{jj}$ is real, the contribution to each high action is
\[
 \left.\dot I_j\right|_{Z_{\rm blk}}
   =2\operatorname{Re}(-\ii A_{jj}|w_j|^2)=0
       \qquad(|j|>M_N).
\]
There is no output for $N<|j|\leq M_N$. Thus its entire
contribution to the energy derivative is the low sum
\[
 \{\mathcal E_{s,N},Z_{\rm blk}\}
 =2\operatorname{Re}\sum_{|j|\leq N}
       w_s(j)^2\bar w_j(X_{Z_{\rm blk}})_j.
\]
Put $\mathcal R_d=R_T+E_d$. Exact commutation of the remaining
normal terms gives the full identity
\begin{align}\label{eq:energy-commutator-expanded}
 \frac{d}{dt}\mathcal E_{s,N}(w(t))
 &=2\operatorname{Re}\sum_{|j|\leq N}
       w_s(j)^2\bar w_j(X_{Z_{\rm blk}})_j\nonumber\\
 &\quad+2\operatorname{Re}\sum_{j\in\ZZ}
       w_s(j)^2\bar w_j(X_{\mathcal R_d})_j.
\end{align}
Both sums converge absolutely by Cauchy--Schwarz. By
\eqref{eq:blocklow}, \eqref{eq:taylorfinal} and \eqref{eq:highfinal},
\[
 \left|\frac{d}{dt}\mathcal E_{s,N}(w(t))\right|
 \leq C\sqrt{\mathcal E_{s,N}(w(t))}\,r^2\mathfrak e_{d,N}(r).
\]
Division by $2\sqrt{\mathcal E_{s,N}}$ at its positive values
proves \eqref{eq:energydiff}.
For graph-regular solutions, differentiate the quadratic form
and use the absolutely convergent sums in
\eqref{eq:energy-commutator-expanded}. Apply the resulting bound to
$(\mathcal E_{s,N}+h)^{1/2}$, integrate between two fixed times
and let $h\downarrow0$.
For a mild solution, approximate its initial value by vectors in
$D(\mathscr L)$. On each compact interval contained in the
$2r$-ball, the corresponding solutions converge uniformly in $W_s$
by \eqref{eq:mild-cont-dependence}. The nonlinear fields in
\eqref{eq:energy-commutator-expanded} are continuous and bounded
on a common smaller ball. Cauchy--Schwarz therefore gives uniform
convergence of the integrated right-hand side, and continuity of
$\mathcal E_{s,N}$ gives convergence of the left-hand side.
The integrated estimate passes to the limit, proving the assertion
for mild solutions as well.
\end{proof}

\subsection{Continuation and the fixed-amplification theorem}
\begin{proof}[Proof of the lower-bound part of Theorem~\ref{thm:joint}]
Fix $A>1$ and let
$\delta_A=\min\{(A-1)/8,1/8\}$. Choose a fixed $B\geq8$
and put $r=B\varepsilon$, $L_r=\log(1/r)$.
For all sufficiently small $\varepsilon$, Proposition~\ref{prop:parameters}
chooses $d,N$ satisfying every hypothesis of Theorem~\ref{thm:nf}.
Let $w_0=\mathcal T_d^{-1}(u_0)$. The inverse bound gives
\[
 \|w_0\|_s\leq\varepsilon+C\Delta\varepsilon^2
                  \leq(1+\delta_A)\varepsilon,
\]
after reducing the fixed initial-data threshold. The bound is uniform
over all $u_0$ on the sphere of radius $\varepsilon$.

The parameter estimates give a fixed $c_*>0$ such that
$\mathfrak e_{d,N}(r)\leq C\exp(-c_*\Phi(L_r))$.
Define $T_*=\exp(c_*\Phi(L_r)/2)$ and let $T_b$ be the
supremum of times $T>0$ for which the transformed mild solution
exists on $[0,T]$ and has $\|w(t)\|_s<r$ there.
For $0\leq t<\min(T_b,T_*)$, Lemma~\ref{lem:energy} implies
\begin{align*}
 \|w(t)\|_s
 &\leq(1+\delta_A)\varepsilon
               +CB^2\varepsilon^2T_*\ee^{-c_*\Phi(L_r)}\\
 &\leq(1+2\delta_A)\varepsilon<r/2,
\end{align*}
when $CB^2\varepsilon\exp(-c_*\Phi(L_r)/2)\leq\delta_A$.
If $T_b\leq T_*$, continuity excludes exit through the sphere
of radius $r$, while the uniform local continuation time on the
$r/2$-ball excludes a finite maximal existence endpoint.
Therefore $T_b>T_*$ and the displayed bound holds up to $T_*$.

Return to $u(t)=\mathcal T_d(w(t))$.
The near-identity estimate at these points gives
$\|u(t)-w(t)\|_s\leq C\Delta(1+2\delta_A)^2\varepsilon^2
\leq\delta_A\varepsilon$. Thus
\[
 \|u(t)\|_s\leq(1+3\delta_A)\varepsilon<A\varepsilon
 \quad(0\leq t\leq T_*).
\]
The same argument starting with negative time proves the estimate on
$[-T_*,0]$. Since $L_r=\log(1/\varepsilon)-\log B$ and each
of the two explicit functions $\Phi$ satisfies
$\Phi(L-\log B)/\Phi(L)\to1$, one may decrease the positive
constant in the exponent and obtain the stated time
$\exp(c_A\Phi(\log(1/\varepsilon)))$.
This proves existence, uniqueness and the prescribed norm bound on
the whole interval for every sufficiently small datum.
\end{proof}

\section{Applications to three equations on the circle}\label{sec:apps}
Stability near plane waves
is studied in \cite{FGL13,W15}, while \cite{FM24} treats lifespan
and high Sobolev norms for the completely resonant NLS. Further
quantitative results concern
Klein--Gordon equations with potential \cite{P18}, the beam equation
\cite{FM23}, the derivative wave equation \cite{SLW23}, and the
Schr\"odinger--Poisson equation \cite{BCGW24}.

For each equation the proof has three inputs: an analytic coefficient
bound, a divisor estimate with its degree dependence, and a gap between
distinct absolute frequencies. The coefficient calculation
\eqref{eq:NLS-majorant-explicit} serves both NLS equations, while
\eqref{eq:beam-majorant-full} serves the beam equation.
Proposition~\ref{prop:nlsnr} gives $p=1$ for convolution NLS.
Lemma~\ref{lem:powers} and Proposition~\ref{prop:scalar} give
$p=2$ for both fractional NLS and the metric beam equation; their
respective substitutions are recorded in
\eqref{eq:scalar-application-nodes} below.
Lemma~\ref{lem:1dgap} supplies the geometric gap once the spectrum
has been identified. We verify these inputs before stating each pair
of stability conclusions.

We apply the normal-form theorem to three models with analytic nonlinearities.
For each model we identify its Fourier Hamiltonian and verify the
frequency conditions on a full-measure parameter set. All tuple counts are over integer indices $|j|\leq N$.
In each stability statement $\varepsilon>0$ denotes the initial norm
and $L=\log(1/\varepsilon)$. The zero datum gives the identically
zero solution.

\subsection{Comparison with existing stability theorems}\label{subsec:comparison}
We compare the hypotheses together with the time, keeping every
regularity parameter fixed as the initial size tends to zero.
For a positive sequence $w$ write
\[
 \|u\|_{\ell^2(w)}^2=\sum_jw_j^2|u_j|^2,
 \qquad L=\log(1/\varepsilon),\qquad \iota_j=\max(1,|j|).
\]
The following table records the leading scale of $\log T$ in the
specified theorems. Additive $O(L)$ terms and fixed positive
multipliers are suppressed. The initial sets and parameter laws
are described immediately below the table.
\begin{center}\small
\renewcommand{\arraystretch}{1.25}
\begin{tabular}{@{}p{.24\textwidth}p{.44\textwidth}p{.24\textwidth}@{}}
\toprule
Result & Equation and fixed regularity & Scale of $\log T$\\
\midrule
\cite[Theorem 1.1]{FG13}
 & NLS; analytic strip widths $2\rho$ initially and $\rho/2$ finally
 & $L^{1+\beta}$, $0<\beta<1$\\
\cite[Theorem 1.1]{BMP20}
 & NLS; $\ell^2(\iota_j^{p_0}\ee^{s\iota_j^\theta})$,
 $p_0>1/2$, $0<\theta<1$
 & $L^{1+\theta/4}$\\
\cite[Theorem 1.1]{BG25}
 & NLS on $\mathbb T^d$; $H^{s_0}$, $s_0>d/2$
 & $L^2/(4\log L)$\\
\cite[Theorem 1.1]{BCGW24}
 & Schr\"odinger--Poisson; $2\sum_j\ee^{\sigma|j|^\theta}|u_j|$,
 $\sigma>0$, $0<\theta<1$
 & $c_{\sigma,\theta}L^2/\log L$\\
\cite[Theorem 1.4]{FM23}
 & Beam; $\mathcal N_{s,p_0}$ in \eqref{eq:FM-cauchy-norm},
 $p_0>3/2$, $1<q\leq2$
 & $L(\log L)^{(q-1)/2}$\\
Theorem~\ref{thm:joint}
 & The stated Hamiltonian class; $W^G_{s,\theta}$ or $W^U_{s,q}$
 & $\Phi_{G,p}(L)$ or $\Phi_{U,p,q}(L)$\\
\bottomrule
\end{tabular}
\end{center}

Faou--Gr\'ebert \cite[Theorem 1.1]{FG13} choose independent
signed Fourier coefficients with the product law
\[
 V_j=R(1+|j|)^{-m_0}v_j,\qquad
 v_j\sim\operatorname{Unif}[-1/2,1/2],\qquad m_0>d/2,
\]
where $R>0$ is fixed. For almost every such potential, every
sufficiently small datum
in the supremum norm on the strip of width $2\rho$ is controlled
in the strip of width $\rho/2$. In the Gevrey case $a=0$,
Biasco--Massetti--Procesi \cite[Theorem 1.1]{BMP20} bound the
solution in the same weighted Hilbert norm as the initial datum. Their full statement
also permits a factor $\ee^{a|j|}$, $a\geq0$. Their independently
variable frequencies are indexed by signed integers, with
\[
 \omega_j=j^2+\iota_j^{-\rho}v_j,\qquad \rho\geq0,
 \qquad v_j\sim\operatorname{Unif}[-1/2,1/2].
\]
They fix a frequency vector in the product nonresonance set
\[
 |\omega\cdot\ell|>
 \gamma\prod_j(1+|\ell_j|^2\iota_j^{2+\rho})^{-1}
 \qquad(0\ne\ell\in\ZZ^{\ZZ},\ \|\ell\|_1<\infty).
\]
Their Remark 1.2 discusses larger exponents than the one displayed
in Theorem 1.1. Proposition~\ref{prop:radial-weights}
below proves our lower bound directly in the polynomially weighted
Gevrey norm in the table. For even potentials,
Proposition~\ref{prop:nlsnr} supplies the arithmetic estimate using
independent variables indexed by absolute modes. Our NLS
application has a translation-invariant nonlinearity; their theorem
also allows analytic dependence on $x$.

The all-data theorem of Bernier--Gr\'ebert
\cite[Theorem 1.1]{BG25} uses a bounded multiplier that is constant
on dyadic blocks. In the Fourier normalization of
\cite[equation (1.7)]{BG25}, its law is
\[
 V_k=X_n\quad(k\in B_n),\qquad
 X_n\sim\operatorname{Unif}[0,1]\quad\text{independently in }n.
\]
Smallness
is required only in $H^{s_0}$; if a datum also belongs to $H^s$,
the conclusion controls that higher norm by a fixed multiple of
its initial value. The special block law allows the displayed
time at Sobolev regularity. For the parameter-free
Schr\"odinger--Poisson equation, Bernier--Camps--Gr\'ebert--Wang
\cite[Theorem 1.1]{BCGW24} instead use the dependence of the
nonlinear frequencies on the initial actions. Their admissible
sets $\Theta_\varepsilon$ are asymptotically full in the
finite-dimensional density and product-probability senses of
\cite[Propositions 1.3 and 1.5]{BCGW24}. Their theorem preserves
the fixed weighted $\ell^1$ Gevrey norm, with a factor two.
In Theorem~\ref{thm:joint}, every sufficiently small datum is
admissible once the external arithmetic data have been fixed.

Feola--Massetti \cite[Theorem 1.4, equation (1.19)]{FM23} treat
$\psi_{tt}+\psi_{xxxx}+m\psi+F'(\psi)=0$ with analytic $F$,
for a set of masses in $[1,2]$ whose complement has measure
$O(\gamma)$. Their logarithmic regularity and the initial
and final Cauchy norms are fixed. The calculation in
Subsection~\ref{subsec:beam-comparison} applies our theorem to
this same mass family and these same norms; it gives
\[
 \log T\geq cL^{3q/(2q+1)},\qquad
 \frac{L^{3q/(2q+1)}}{L(\log L)^{(q-1)/2}}
 =\frac{L^{(q-1)/(2q+1)}}{(\log L)^{(q-1)/2}}
 \longrightarrow\infty.
\]
The beam application also allows the gradient nonlinearity
in \eqref{eq:intro-models} and every fixed $q>1$.

To complement these lower estimates, we construct frequencies in
Lemma~\ref{lem:implant} that satisfy both the divisor hypothesis of
Theorem~\ref{thm:joint} and prescribed transfer relations. The
resulting equations have solutions with
$\log T_A^+(u_0^{(n)};H)\leq C_A\Phi(L_n)$, as proved in
Sections~\ref{sec:Gupper} and~\ref{sec:Uupper}. Applying the lower
theorem to each of these fixed equations gives the matching scale.

\subsection{Elementary one-dimensional frequency gaps}\label{subsec:clusters}
\begin{lemma}[Gaps between distinct absolute indices]\label{lem:1dgap}
Let $\beta>1$. Suppose $h\in C^1([x_0,\infty))$ and $h'(x)\geq c x^{\beta-1}$ there. Let $\lambda_n=h(n)+v_n$ for large integers $n$, where $\sup_n|v_n|\leq V_*<\infty$. If the sequence $(\lambda_n)_{n\geq0}$ has no repeated values, then \eqref{eq:radialgap} holds with $\delta=\beta-1$. The gap constant may depend on the finitely many low values and on the fixed data.
\end{lemma}
\begin{proof}
Increase $x_0$ so that $x_0\geq c_0$ and
$\lambda_n=h(n)+v_n$ for every integer $n\geq x_0$.
For integers $n>k\geq x_0$ with $k\geq n/2$,
\[
 h(n)-h(k)\geq c\int_k^n x^{\beta-1}\,dx
 \geq c(n-k)(n/2)^{\beta-1}\geq c2^{1-\beta}n^{\beta-1}.
\]
For $x_0\leq k<n/2$, we instead have
\[
 h(n)-h(k)\geq c\int_{n/2}^n x^{\beta-1}\,dx
 =\frac{c}{\beta}(1-2^{-\beta})n^\beta.
\]
In both cases, for sufficiently large $n$,
\[
 \lambda_n-\lambda_k
 \geq h(n)-h(k)-2V_*
 \geq c_1n^{\beta-1}
 \geq\frac{c_1}{2}(\br n^{\beta-1}+\br k^{\beta-1}).
\]
For $0\leq k<x_0$, the same conclusion follows from
\begin{align*}
 \lambda_n-\lambda_k
 &\geq h(x_0)+\frac{c}{\beta}(n^\beta-x_0^\beta)
           -V_*-\max_{0\leq j<x_0}|\lambda_j|\\
 &\geq\frac{c}{2\beta}n^\beta
 \geq c_2(\br n^{\beta-1}+\br k^{\beta-1})
\end{align*}
when $n$ is sufficiently large. Choose one integer $n_0$ for
which these estimates hold whenever $n\geq n_0$. The remaining
pairs satisfy
\[
 c_{\rm low}:=\min_{0\leq k<n<n_0}
 \frac{|\lambda_n-\lambda_k|}
      {\br n^{\beta-1}+\br k^{\beta-1}}>0.
\]
Thus \eqref{eq:radialgap} holds with
$\delta=\beta-1$ and
$c_{\rm gap}=\min(c_1/2,c_2,c_{\rm low})$.
For $j\in\mathcal B_n$, $k\in\mathcal B_\ell$, $n\ne\ell$,
\[
 |j-k|+|\omega_j-\omega_k|
 \geq|\lambda_n-\lambda_\ell|
 \geq c_{\rm gap}(\br j^{\beta-1}+\br k^{\beta-1}),
\]
whereas $f(\br j)-f(\br k)=0$ if $j,k\in\mathcal B_n$.
This proves both the separation and the equality of weights
required for the block decomposition.
\end{proof}

\subsection{Convolution NLS: full-measure verification with \texorpdfstring{$p=1$}{p=1}}
Consider, on $\mathbb T$,
\begin{equation}\label{eq:nls}
 \ii\partial_t\psi=-\partial_x^2\psi+V*\psi+a(|\psi|^2)\psi,
 \qquad a(z)=\sum_{h\geq1}a_hz^h\text{ analytic near }0.
\end{equation}
The coefficients $a_h$ are real. The Hamiltonian potential is $\int_{\mathbb T}A(|\psi|^2)\,dx/(2\pi)$, $A'=a$, and satisfies Definition~\ref{def:coef}. Choose real even Fourier coefficients $V_j=\br j^{-\rho}v_{[j]}$, where $[j]=\{j,-j\}$ and the $v_{[j]}$ are independent uniform variables in $[-1/2,1/2]$. Any fixed $\rho>0$ is allowed for the multiplier statement; one may take $\rho>1$ if a continuous convolution kernel is desired.

\subsubsection{Fourier Hamiltonian and its coefficient norm}
Let $a(z)=\sum_{h\geq1}a_hz^h$ and take
$A(z)=\sum_{h\geq1}a_hz^{h+1}/(h+1)$ with $A(0)=0$.
Using the normalized circle integral, the Fourier expansion of the
potential is
\begin{equation}\label{eq:NLS-potential-coefficients}
 P=\sum_{h\geq1}\frac{a_h}{h+1}
  \sum_{j_1+\cdots+j_{h+1}=k_1+\cdots+k_{h+1}}
  u_{j_1}\cdots u_{j_{h+1}}\bar u_{k_1}\cdots\bar u_{k_{h+1}}.
\end{equation}
Each summand has zero momentum. The real coefficients $a_h$
give reality. Symmetrization of this ordered expression gives
$c_{2h+2}(P)\leq|a_h|/(h+1)$ and $c_m(P)=0$ for odd $m$.
If $|a_h|\leq A_*\zeta^{-h}$, then
\begin{equation}\label{eq:NLS-majorant-explicit}
 \coef P_r
 \leq8A_*\sum_{h\geq1}(h+1)^2
                   (b^2r^2/\zeta)^h<\infty
 \qquad(b^2r^2<\zeta).
\end{equation}
This provides a fixed analytic radius and perturbation bound for
both convolution NLS and fractional NLS.
Differentiating \eqref{eq:NLS-potential-coefficients} with respect
to $\bar u_j$ cancels $h+1$ against its denominator and gives the
$j$th Fourier coefficient of $a(|\psi|^2)\psi$. Thus the
Hamiltonian convention gives \eqref{eq:nls}, and the initial-data
norm is \eqref{eq:physicalnorm}.

\subsubsection{A simultaneous full-measure divisor bound}
\begin{proposition}[The divisor bound for convolution NLS]\label{prop:nlsnr}
For almost every such potential there is $\gamma(V)>0$ and a fixed $\tau$ for which, simultaneously for every $\ell\geq3$, $N\geq\ell$ and every nonzero reduced frequency combination of length at most $\ell$ below $N$,
\[
 |\Omega|\geq\gamma(V)(C\ell N)^{-\tau\ell}.
\]
The only exact frequency classes are $\{j,-j\}$ (with $\{0\}$ separately), and all hypotheses of Theorem~\ref{thm:joint} hold with $p=1$.
\end{proposition}
\begin{proof}
Group equal blocks $[j]$ and write the reduced divisor as
$\Omega=\sum_{[j]}c_{[j]}(j^2+\br j^{-\rho}v_{[j]})$.
If $c_{[j_*]}\ne0$, conditioning on all other parameters gives
\begin{align*}
 \left|\partial_{v_{[j_*]}}\Omega\right|
 &=|c_{[j_*]}|\br{j_*}^{-\rho}\geq cN^{-\rho},\\
 \mathbb P\bigl(|\Omega|<h\mid(v_{[j]})_{[j]\ne[j_*]}\bigr)
 &\leq\frac{2h}{|c_{[j_*]}|\br{j_*}^{-\rho}}
  \leq CN^\rho h.
\end{align*}
The number of signed tuples of lengths at most $\ell$ is bounded by
\[
 \sum_{a=1}^{\ell}[2(2N+1)]^a\leq(CN)^\ell.
\]
For $h=\gamma(C\ell N)^{-\tau\ell}$, the union bound is therefore
\[
 C\gamma\sum_{\ell\geq3}\sum_{N\geq\ell}
       N^\rho(CN)^\ell(C\ell N)^{-\tau\ell}\leq C'\gamma.
\]
Choose $C\geq1$ to dominate the tuple-count constant and
$\tau\geq\rho+6$. For $\ell\geq3$ and $N\geq\ell$,
\[
 N^\rho(CN)^\ell(C\ell N)^{-\tau\ell}
 \leq N^{\rho-(\tau-1)\ell}
 \leq N^{-5\ell}\leq2^{-\ell}N^{-2}.
\]
This proves convergence of the displayed double sum. If
$\mathcal G_\gamma$ is the set on which all the divisor
inequalities hold, then
$\mathbb P(\mathcal G_\gamma^c)\leq C'\gamma$.
The nesting $\mathcal G_{2^{-j}}\subset\mathcal G_{2^{-j-1}}$
gives
\[
 \mathbb P\!\left(\left(\bigcup_{j\geq1}\mathcal G_{2^{-j}}\right)^c\right)
 \leq\inf_{j\geq1}C'2^{-j}=0.
\]
For $V\in\mathcal G_{2^{-j}}$, set $\gamma(V)=2^{-j}$; this
one constant works for every degree and cutoff.
A nonzero reduced combination of order one or two can also be
represented with length at most $\ell=3$; increasing $N$ includes
any fixed support. The same inequalities therefore exclude a zero
frequency and equalities between distinct representatives. Thus exact resonances have precisely the pairing required in \eqref{eq:nr}, and frequency equality implies equal modulus. The spectrum grows quadratically. Lemma~\ref{lem:1dgap}, with $h(x)=x^2$, gives \eqref{eq:radialgap} with $\delta=1$. The exact equal-frequency classes are the blocks \eqref{eq:blocks}.
\end{proof}
For a tuple of length $m\geq3$, take $\ell=m$ in the
proposition. Since $0<\gamma(V)<1$,
$\gamma(V)(CmN)^{-\tau m}\geq\gamma(V)^m(CmN)^{-\tau m}$,
which is \eqref{eq:nr} with $p=1$.
The corresponding times are $\exp(c_A L^2/\log L)$ in the
Gevrey norm and $\exp(c_A L^{2q/(q+1)})$ in the logarithmic norm.

\subsubsection{Comparison with product divisor conditions}
Fix $\mu_1,\mu_2>1$ and the decay exponent $\rho>0$.
For a finitely supported integer vector $c$ on the absolute-index
blocks, let
$R_c=\sum|c_j|\leq\ell$ and suppose its support is below $N\geq c_0$.
Put $a_{\rm prod}=\mu_1+\mu_2+\rho>1$. For every nonzero coefficient,
\[
 1+|c_j|^{\mu_1}\br j^{\mu_2+\rho}
 \leq(1+|c_j|\br j)^{a_{\rm prod}}
 \leq(1+\br j)^{a_{\rm prod}|c_j|}.
\]
For the first inequality use $|c_j|,\br j\geq1$,
 $|c_j|^{\mu_1}\br j^{\mu_2+\rho}\leq(|c_j|\br j)^{a_{\rm prod}}$
 and $1+x^{a_{\rm prod}}\leq(1+x)^{a_{\rm prod}}$.
 The second inequality is the integer binomial inequality
$(1+x)^b\geq1+bx$ for $b=|c_j|$.
Multiplication therefore gives
\begin{equation}\label{eq:product-to-cutoff-full}
 \prod_{j:c_j\ne0}(1+|c_j|^{\mu_1}\br j^{\mu_2+\rho})
 \leq(1+N)^{a_{\rm prod}R_c}\leq(C\ell N)^{a_{\rm prod}\ell}.
\end{equation}
Thus a bound
 $|\lambda\cdot c|\geq\gamma\prod_{j:c_j\ne0}
 (1+|c_j|^{\mu_1}\br j^{\mu_2+\rho})^{-1}$
 implies $|\lambda\cdot c|\geq\gamma(C\ell N)^{-a_{\rm prod}\ell}$.
This is the deterministic product-to-cutoff implication for
conditions of the type used in \cite{B05,BMP20}; on the block
indices it yields $p=1$. Under the signed product law,
$V_j$ and $V_{-j}$ are independent continuous random variables,
so
\[
 \mathbb P(V_j=V_{-j})=0\qquad(j>0).
\]
The even subfamily therefore has zero measure under that law.
For the even potentials considered here we use the product law on
absolute modes; Proposition~\ref{prop:nlsnr} gives the cutoff bound
on a full-measure set for this law.

\subsubsection{The two stability conclusions}
\begin{theorem}[Convolution NLS in Gevrey regularity]\label{ex1thm1}
Fix $0<\theta<1$, $s>0$ and $A>1$. For almost every even
potential of Proposition~\ref{prop:nlsnr}, there are
$c_A,\varepsilon_A>0$ such that every datum with
$0<\|\psi_0\|^G_{s,\theta}=\varepsilon<\varepsilon_A$ has a unique
solution of \eqref{eq:nls} on
$|t|\leq\exp(c_AL^2/\log L)$, $L=\log(1/\varepsilon)$,
and $\|\psi(t)\|^G_{s,\theta}<A\varepsilon$ there.
\end{theorem}
\begin{theorem}[Convolution NLS in logarithmic regularity]\label{ex1thm2}
Fix $q>1$, $\kappa_*>\ee^q$, $s>0$ and $A>1$.
For the same full-measure set of even potentials and all sufficiently
small $\varepsilon=\|\psi_0\|^U_{s,q}$, the unique solution of
\eqref{eq:nls} satisfies $\|\psi(t)\|^U_{s,q}<A\varepsilon$
on $|t|\leq\exp(c_AL^{2q/(q+1)})$.
\end{theorem}
\begin{proof}[Proof of Theorems~\ref{ex1thm1} and \ref{ex1thm2}]
Choose $s_0=s/2$ in the weight estimates. Formula
\eqref{eq:NLS-majorant-explicit} provides the fixed analytic data,
Proposition~\ref{prop:nlsnr} supplies the divisor bound with $p=1$,
and Lemma~\ref{lem:1dgap} supplies $\beta=2$, $\delta=1$.
The Fourier Hamiltonian was identified in
\eqref{eq:NLS-potential-coefficients}. All inputs of
Theorem~\ref{thm:joint} are therefore satisfied for this fixed
equation. Substitution of $p=1$ gives
$\Phi_{G,1}(L)=L^2/\log L$ and
$\Phi_{U,1,q}(L)=L^{2q/(q+1)}$, with the norms stated above.
Its existence and fixed-amplification conclusions prove both theorems.
\end{proof}

\subsection{A scalar-parameter sublevel estimate with its order dependence}\label{subsec:sublevel}
For fractional NLS and the metric beam equation, all frequencies depend
on one scalar parameter. We extract a nonzero derivative of each divisor
from a determinant of separated powers. Its lower bound has a loss
quadratic in the number of distinct frequencies. We keep this dependence
in the sublevel estimate and in the sum of exceptional sets; the
resulting divisor threshold has $p=2$ in \eqref{eq:nr}. This is the
arithmetic input used in the stability times for these two models.
\begin{lemma}[Separated powers]\label{lem:powers}
Let $I\Subset(0,\infty)$ and $\eta\notin\mathbb Z$ be fixed. Suppose
\[
 0\leq a_1<\cdots<a_k\leq C N^b,
 \qquad |a_i-a_j|\geq c N^{-b'},
\]
where $b,b',c,C$ are fixed and $N\geq2$. For $F(x)=\sum_{j=1}^k c_j(a_j+x)^\eta$, $c_j\in\ZZ\setminus\{0\}$, there are constants independent of $k,N,c_j$ such that
\begin{align}
 \max_{0\leq r\leq k-1}|F^{(r)}(x)|&\geq\delta_{k,N}:=(CkN)^{-Ck^2},\label{eq:wronskianlower}\\
 |\{x\in I:|F(x)|<h\}|&\leq(CkN)^{Ck}h^{1/k}.\label{eq:powerssub}
\end{align}
\end{lemma}
\begin{proof}
Write $I=[x_-,x_+]$ with $x_->0$, and set $x_j=a_j+x$.
For $0\leq r\leq k-1$ the derivative matrix is
$A_{rj}=(\eta)_r x_j^{\eta-r}$, where
$(\eta)_0=1$ and $(\eta)_r=\eta(\eta-1)\cdots(\eta-r+1)$.
Factoring $(\eta)_r$ from row $r$ and $x_j^{\eta-k+1}$
from column $j$ leaves the powers $x_j^{k-1-r}$.
The determinant of this last matrix is, up to sign, the Vandermonde
product $\prod_{i<j}(x_j-x_i)$. Since $x_j-x_i=a_j-a_i$,
\begin{equation}\label{eq:wronskianformula}
 |\det A|=\left|\prod_{r=0}^{k-1}(\eta)_r\right|
              \prod_{j=1}^k x_j^{\eta-k+1}
              \prod_{i<j}|a_j-a_i|.
\end{equation}
All estimates that follow are uniform for $x\in I$.

Because $\eta$ is fixed and noninteger, there is $c_\eta>0$
with $|\eta-r|\geq c_\eta$ for all nonnegative integers $r$.
It follows that
$|\prod_{r=0}^{k-1}(\eta)_r|
\geq\min(1,c_\eta)^{k(k-1)/2}$.
The factors $x_j$ satisfy $x_-\leq x_j\leq C_I N^b$.
For either sign of $\eta-k+1$ this gives
$\prod_jx_j^{\eta-k+1}\geq(C_I N)^{-C_Ik^2}$,
after enlarging a fixed constant to cover the finitely many small
$k$. Node separation gives
$\prod_{i<j}|a_j-a_i|\geq(cN^{-b'})^{k(k-1)/2}$.
Combining the three bounds proves
\begin{equation}\label{eq:Wronskian-determinant-bound-full}
 |\det A|\geq(CkN)^{-Ck^2}.
\end{equation}

The entries have the uniform bound
\[
 |A_{rj}|\leq(|\eta|+k)^r
       \max(1,x_-^{-k-|\eta|})(C_I N^b)^{\max(\eta,0)}
 \leq S_{k,N}:=(Ck)^kN^C.
\]
Dividing each row by $kS_{k,N}$ produces rows of $\ell^1$ norm
at most one. Lemma~\ref{lem:detector-full}, with the nonzero
integer vector $c$, gives
\begin{equation}\label{eq:Wronskian-transversality-full}
 \max_{0\leq r<k}|F^{(r)}(x)|
 \geq\frac{|\det A|\|c\|_1}
              {k^{3/2}(kS_{k,N})^{k-1}}
 \geq(CkN)^{-Ck^2}.
\end{equation}
Equivalently, the cofactor formula gives the explicit inverse bound
\[
 \|A^{-1}\|_{\infty\to\infty}
 \leq\frac{k(k-1)!S_{k,N}^{k-1}}{|\det A|}
 \leq(CkN)^{Ck^2}.
\]
Choose the constant in $\delta_{k,N}$ to cover these estimates.
This proves \eqref{eq:wronskianlower}.

We next partition the parameter interval using a zero count.
For any noninteger $\lambda$ the first $j+1$ functions of
$1,(a_1+x)^\lambda,\ldots,(a_k+x)^\lambda$ have Wronskian
of absolute value
\begin{equation}\label{eq:powers-constant-Wronskian-full}
 \left|\prod_{r=1}^{j}(\lambda)_r\right|
 \prod_{i=1}^{j}(a_i+x)^{\lambda-j}
                       \prod_{i<l\leq j}|a_l-a_i|>0.
\end{equation}
Expansion along the constant column and the same Vandermonde
factorization prove this formula. Lemma~\ref{lem:ECT-full}
therefore bounds the zeros of each nontrivial combination of these
$j+1$ functions by $j$.
For each $0\leq r<k$, $F^{(r)}$ is a combination of
$(a_i+x)^{\eta-r}$ with nonzero coefficients
$c_i(\eta)_r$.
Since $\eta-r$ is noninteger, each function
$F^{(r)}\pm\delta_{k,N}/2$ has at most $k$ zeros on $I$.
The strict positivity of the Wronskian with the constant column
also excludes its being identically zero.

All these threshold points partition $I$ into at most $2k^2+1$
open intervals, apart from endpoints and the threshold points
themselves, which have measure zero. On a given interval the
truth value of each inequality
$|F^{(r)}|\geq\delta_{k,N}/2$ is constant.
At any one point some derivative has modulus at least
$\delta_{k,N}$, by \eqref{eq:wronskianlower}. Thus the same
derivative has modulus at least $\delta_{k,N}/2$ throughout
that whole interval.

Suppose first that $0<h<\delta_{k,N}/2$. If the selected
derivative has order zero, the sublevel set in this interval is
empty. If its order is $1\leq r<k$, Lemma~\ref{lem:sublevel-full}
gives
\[
 |\{|F|<h\}\cap I_*|
 \leq2(r+1)(2h/\delta_{k,N})^{1/r}
 \leq4k(h/\delta_{k,N})^{1/k}.
\]
The second inequality uses $h/\delta_{k,N}<1$ and $r<k$.
After summing over the at most $2k^2+1$ intervals, the bound is
$Ck^3\delta_{k,N}^{-1/k}h^{1/k}
\leq(CkN)^{Ck}h^{1/k}$.
For $h\geq\delta_{k,N}/2$ the interval-length bound
$|I|\leq C_I(2h/\delta_{k,N})^{1/k}$ proves the same estimate.
When $k=1$, the order-zero bound alone gives an empty sublevel
set for $h<\delta_{1,N}$ and the last interval-length argument
handles the remaining values of $h$. This proves
\eqref{eq:powerssub} in every case.
\end{proof}

The next proposition is the common arithmetic step for
\eqref{eq:fnls} and \eqref{eq:beam}. To distinguish the physical
parameters from the variable in Lemma~\ref{lem:powers}, the two
substitutions are
\begin{equation}\label{eq:scalar-application-nodes}
 \begin{aligned}
 \text{Fractional NLS:}\quad
 &x=m,\qquad a_j=n_j^2,\qquad \text{power }\eta;\\
 \text{Beam equation:}\quad
 &x=m/\zeta^2,\qquad a_j=\bar g^2n_j^4,
       \qquad \text{power }1/2,\qquad g=\zeta\bar g.
 \end{aligned}
\end{equation}
The nonlinearities are handled separately by
\eqref{eq:NLS-majorant-explicit} and \eqref{eq:beam-majorant-full}.
The proposition below supplies exactly the exponent $\ell^2$ in
the divisor loss, hence $p=2$ in Assumption~\ref{ass:freq}.

\begin{proposition}[The scalar divisor bound for fractional NLS and beam]\label{prop:scalar}
For families of the preceding form obtained from signed frequency tuples of length $\ell$, with at most $(CN)^\ell$ tuples below $N$, there is a fixed $\tau>0$ and a full-measure set of parameters $x$ such that, for some $\gamma(x)>0$,
\begin{equation}\label{eq:scalarNR}
 |F(x)|\geq\gamma(x)^\ell(C\ell N)^{-\tau\ell^2}
\end{equation}
simultaneously for all $\ell\geq3$, $N\geq\ell$ and all nonzero reduced combinations.
\end{proposition}
\begin{proof}
Let $C_{\rm sub}$ denote the fixed coefficient of $k$ in the
exponent of \eqref{eq:powerssub}. For a fixed divisor and
$0<\gamma<1$, substitute $h=\gamma^\ell(C\ell N)^{-\tau\ell^2}$
into that estimate. Since $k\leq\ell$, its bad-set measure is at most
\[
 \gamma^{\ell/k}(C\ell N)^{C_{\rm sub}k-\tau\ell^2/k}
 \leq\gamma(C\ell N)^{-(\tau-C_{\rm sub})\ell}.
\]
Enlarge the base constant $C$ to dominate the tuple-count constant. Choose
$\tau\geq C_{\rm sub}+4$. Multiplication by the tuple count gives
\[
 C\gamma(CN)^\ell(C\ell N)^{-4\ell}
 \leq C\gamma(C\ell N)^{-3\ell}
 \leq C\gamma2^{-\ell}N^{-2}.
\]
For the exceptional set $\mathcal B_\gamma\subset I$ this gives
\[
 |\mathcal B_\gamma|
 \leq C\gamma\sum_{\ell\geq3}\sum_{N\geq\ell}2^{-\ell}N^{-2}
 \leq C_I\gamma.
\]
Put $\mathcal G_\gamma=I\setminus\mathcal B_\gamma$. Then
\[
 \mathcal G_{2^{-j}}\subset\mathcal G_{2^{-j-1}},\qquad
 \left|I\setminus\bigcup_{j\geq1}\mathcal G_{2^{-j}}\right|
 \leq\inf_{j\geq1}C_I2^{-j}=0.
\] At a point of the union choose one
$j$ with $x\in\mathcal G_{2^{-j}}$ and set $\gamma(x)=2^{-j}$.
This constant works simultaneously for every degree and cutoff.
\end{proof}
The determinant estimate in Lemma~\ref{lem:powers} has a loss
$\delta_{k,N}^{-1}=(CkN)^{O(k^2)}$. Its appearance in the sublevel
bound is through $\delta_{k,N}^{-1/k}=(CkN)^{O(k)}$.
Consequently the single-divisor estimate and the tuple count combine
as
\[
 (CN)^\ell(CkN)^{Ck}
     \bigl[\gamma^\ell(C\ell N)^{-\tau\ell^2}\bigr]^{1/k}
 \leq\gamma(C\ell N)^{-(\tau-C-1)\ell},\qquad k\leq\ell.
\]
The quadratic power $\ell^2$ is therefore sufficient for a summable
union over all degrees and cutoffs. This is the quantitative input
$p=2$ used by both scalar applications; the argument retains the
quadratic determinant loss without adding another factor of the
degree.

\subsection{Fractional NLS with a scalar mass parameter}\label{subsec:fnls}
Consider
\begin{equation}\label{eq:fnls}
 \ii\partial_t\psi=(-\partial_x^2+m)^\eta\psi+a(|\psi|^2)\psi,
 \quad x\in\mathbb T,\quad m\in I\Subset(0,\infty),
 \quad\eta>1/2,\quad\eta\notin\mathbb N.
\end{equation}
The exponent $\eta$ is fixed, and $a$ is the real analytic function in \eqref{eq:nls}. The Hamiltonian perturbation is the same analytic potential as for convolution NLS, so its coefficient-majorant verification applies directly. The linear frequencies are
\[
 \omega_j=(j^2+m)^\eta,\qquad\lambda_n=(n^2+m)^\eta\quad(n\geq0).
\]
They are strictly increasing as a function of the absolute index. Equal frequencies occur precisely within $\{n,-n\}$, and this remaining equality is included in the resonance pairing.

Grouping a signed divisor by absolute index gives separated nodes $a_\nu=n_\nu^2\leq N^2$ with $|a_\nu-a_\mu|\geq1$ for distinct nodes. Lemma~\ref{lem:powers} and Proposition~\ref{prop:scalar} prove \eqref{eq:nr} with $p=2$ outside a single measure-zero set of masses, simultaneously at every degree and cutoff. A vanishing reduced vector is exactly a pairing of opposite signs within the blocks. Every nonzero reduced vector has a nonzero divisor for masses in this full-measure set.

The geometric estimate is explicit. For $x$ large, uniformly for $m$ in the fixed compact interval,
\begin{equation}\label{eq:fractionalderivative}
 \frac{d}{dx}(x^2+m)^\eta=2\eta x(x^2+m)^{\eta-1}
 \asymp x^{2\eta-1}.
\end{equation}
Since $2\eta-1>0$, Lemma~\ref{lem:1dgap} yields
\begin{equation}\label{eq:fractionalgap}
 |(n^2+m)^\eta-(k^2+m)^\eta|
 \geq c_{\eta,I}(\br n^{2\eta-1}+\br k^{2\eta-1})\quad(n\ne k).
\end{equation}
This proof covers the entire stated range $\eta>1/2$, including $1/2<\eta<1$, with $\beta=2\eta$ and $\delta=2\eta-1$. The constants depend on the fixed exponent $\eta$ and compact mass interval $I$. Thus all hypotheses of Theorem~\ref{thm:joint} have been verified, and for every $A>1$ and sufficiently small data,
\begin{equation}\label{eq:p2times}
 T_A^+\geq\exp\left(c_A\frac{L^{3/2}}{(\log L)^{1/2}}\right)
 \quad\text{or}\quad
 T_A^+\geq\exp\left(c_A L^{1+(q-1)/(2q+1)}\right),
\end{equation}
respectively. The same norm bounds and existence times hold in both time directions. Exhausting $(0,\infty)$ by compact mass intervals gives a full-measure set of positive masses for each fixed noninteger $\eta>1/2$.

\subsubsection{The two stability conclusions}
\begin{theorem}[Fractional NLS in Gevrey regularity]\label{ex2thm1}
Fix a noninteger $\eta>1/2$, $0<\theta<1$, $s>0$ and $A>1$.
For almost every $m>0$ there are $c_A,\varepsilon_A>0$ such
that each datum with $0<\varepsilon=\|\psi_0\|^G_{s,\theta}
<\varepsilon_A$ has a unique solution of \eqref{eq:fnls}
with $\|\psi(t)\|^G_{s,\theta}<A\varepsilon$ throughout
$|t|\leq\exp(c_AL^{3/2}/(\log L)^{1/2})$.
\end{theorem}
\begin{theorem}[Fractional NLS in logarithmic regularity]\label{ex2thm2}
Fix a noninteger $\eta>1/2$, $q>1$, $\kappa_*>\ee^q$,
$s>0$ and $A>1$. For the same full-measure set of positive masses,
all sufficiently small data in $W^U_{s,q}$ have a unique solution
of \eqref{eq:fnls} satisfying $\|\psi(t)\|^U_{s,q}<A\varepsilon$
on $|t|\leq\exp(c_AL^{1+(q-1)/(2q+1)})$.
\end{theorem}
\begin{proof}[Proof of Theorems~\ref{ex2thm1} and \ref{ex2thm2}]
For the reduced nodes $a_j=n_j^2$, the constants in
Lemma~\ref{lem:powers} have $b=2$, $b'=0$, and the node gap
is one. Proposition~\ref{prop:scalar} consequently gives the
simultaneous divisor bound with $p=2$. The frequency gap
\eqref{eq:fractionalgap} has $\beta=2\eta>1$ and
$\delta=2\eta-1>0$. Formula~\eqref{eq:NLS-majorant-explicit}
verifies the perturbation, and Lemma~\ref{lem:weights-full} permits
$s_0=s/2$. Inserting $p=2$ into the two functions in
\eqref{eq:scales} gives exactly the two stated exponents.
The dynamical conclusion of Theorem~\ref{thm:joint}, including
its continuation proof, gives the two results.
\end{proof}

\subsection{The beam equation with a one-dimensional metric}\label{subsec:beam}
A constant positive metric parameter is now a scalar $g>0$. Put $\Delta_g=g\partial_x^2$, $Q_g(j)=gj^2$ and $A_g=g^2\partial_x^4+m$ with fixed $m>0$. Consider the variational equation on $\mathbb T$,
\begin{equation}\label{eq:beam}
 \psi_{tt}+g^2\partial_x^4\psi+m\psi
 =-F_\psi(\psi,\psi_x)+\partial_xF_v(\psi,\psi_x),
\end{equation}
where $F(\psi,v)$ is real analytic and vanishes to order at least three. Its Hamiltonian is
\[
 \frac12\int_{\mathbb T}\bigl(\phi^2+g^2\psi_{xx}^2+m\psi^2\bigr)\frac{dx}{2\pi}
       +\int_{\mathbb T}F(\psi,\psi_x)\frac{dx}{2\pi},\qquad \phi=\psi_t.
\]
The complex canonical coordinate is
\begin{equation}\label{eq:beamcoords}
 u=2^{-1/2}(A_g^{1/4}\psi+\ii A_g^{-1/4}\phi),\qquad
 \psi=2^{-1/2}A_g^{-1/4}(u+\bar u).
\end{equation}
The frequencies are $\omega_j=\sqrt{g^2j^4+m}$. Both input multipliers are bounded:
\begin{equation}\label{eq:beammultipliers}
 \sup_j(g^2j^4+m)^{-1/4}\leq m^{-1/4},\qquad
 \sup_j\frac{|j|}{(g^2j^4+m)^{1/4}}\leq g^{-1/2}.
\end{equation}
Expanding $\int F(\psi,\psi_x)$ in these coordinates gives ordered coefficients with at most exponential degree growth. The full gradient nonlinearity therefore satisfies Definition~\ref{def:coef}. Equivalently, the outer $A_g^{-1/4}\partial_x$ in the complex vector field has order zero. The local flow acts at the fixed weight width used by the abstract theorem.

\subsubsection{The beam symplectic change and analytic coefficients}
The transformation \eqref{eq:beamcoords} is canonical. On real
variations $(\dot\psi_1,\dot\phi_1)$ and
$(\dot\psi_2,\dot\phi_2)$, substitution into
$\ii\int(\dot u_1\overline{\dot u_2}
-\dot u_2\overline{\dot u_1})\,dx/(2\pi)$ gives
\[
 \int_{\mathbb T}(\dot\psi_1\dot\phi_2
                     -\dot\psi_2\dot\phi_1)\,\frac{dx}{2\pi}.
\]
Here the self-adjoint multipliers $A_g^{1/4}$ and $A_g^{-1/4}$
cancel in the two cross terms; the pure $\psi$ and pure $\phi$
terms cancel by antisymmetry. The quadratic Hamiltonian becomes
$\sum_j\omega_j|u_j|^2$.
The inverse transformation on Fourier coefficients is
\begin{equation}\label{eq:beam-modal-inverse-full}
 \psi_j=\frac{u_j+\bar u_{-j}}{\sqrt2\,\omega_j^{1/2}},\qquad
 \phi_j=\frac{\omega_j^{1/2}}{\sqrt2\,\ii}
                         (u_j-\bar u_{-j}).
\end{equation}
The reflection $j\mapsto-j$ is an isometry because the weights
are even. In \eqref{eq:beam-modal-inverse-full}, the index $-j$
in $\bar u_{-j}$ is the Fourier index introduced by conjugation.

Write $F(z,v)=\sum_{a+b\geq3}F_{ab}z^av^b$ and choose
$R_F>0,M_F<\infty$ so that
$|F_{ab}|\leq M_FR_F^{-a-b}$.
The coefficients of the two linear forms satisfy
\[
 \frac1{\sqrt2\,\omega_j^{1/2}}\leq\frac{m^{-1/4}}{\sqrt2},
 \qquad
 \frac{|j|}{\sqrt2\,\omega_j^{1/2}}\leq\frac{g^{-1/2}}{\sqrt2}.
\]
Expanding their conjugate choices and using
$(2\pi)^{-1}\int_{\mathbb T}\ee^{\ii kx}\,dx=\mathbf1_{\{k=0\}}$
therefore gives, for $h=a+b$,
\[
 2^h|F_{ab}|\left(\frac{m^{-1/4}}{\sqrt2}\right)^a
                \left(\frac{g^{-1/2}}{\sqrt2}\right)^b
 \leq M_F\left(\frac{\sqrt2\max(m^{-1/4},g^{-1/2})}{R_F}\right)^h.
\]
With $B_{g,m}=\sqrt2\max(m^{-1/4},g^{-1/2})/R_F$, summation
over the $h+1$ pairs $(a,b)$ gives
\begin{equation}\label{eq:beam-majorant-full}
 c_h(P)\leq(h+1)M_FB_{g,m}^{h},\qquad
 \coef P_r\leq M_FB_{g,m}^2
            \sum_{h\geq3}h^3(h+1)(bB_{g,m}r)^{h-2}.
\end{equation}
The right side is finite for $bB_{g,m}r<1$ at the fixed
weight width $s$.
Differentiation of the Hamiltonian gives the vector field
\[
 \ii\partial_tu=A_g^{1/2}u
   +2^{-1/2}A_g^{-1/4}
       \{F_\psi(\psi,\psi_x)-\partial_xF_v(\psi,\psi_x)\}.
\]
The bounded multipliers in \eqref{eq:beammultipliers} include
the outer $A_g^{-1/4}\partial_x$. This is the complex Hamiltonian
equation associated with the full gradient nonlinearity.

\subsubsection{Metric parameter and frequency gaps}
We give the parameter argument in its one-dimensional form. Fix a reference value $\bar g>0$, write $g=\zeta\bar g$, and let $\zeta$ vary in a compact subinterval of $(0,\infty)$. Distinct absolute indices have
\begin{equation}\label{eq:metricnodes}
 a_n=\bar g^2n^4\leq C_{\bar g}N^4,
 \qquad |a_n-a_k|=\bar g^2|n^4-k^4|\geq\bar g^2\quad(n\ne k).
\end{equation}
The separation comes directly from integer fourth powers. Thus the separated-node estimate, including the complete Wronskian calculation \eqref{eq:wronskianformula}, applies to $(a_n+x)^{1/2}$. Proposition~\ref{prop:scalar} gives $p=2$ on a full-measure set of positive $x$, with one positive small-divisor constant for all degrees.

The exact parameter scaling is
\begin{equation}\label{eq:metricscaling}
 \omega_j(\zeta\bar g)=\zeta\sqrt{\bar g^2j^4+m/\zeta^2},
 \qquad x=m/\zeta^2,\qquad \frac{dx}{d\zeta}=-\frac{2m}{\zeta^3}.
\end{equation}
On $\zeta\in[\zeta_-,\zeta_+]\Subset(0,\infty)$ the inverse
is $\zeta=\sqrt{m/x}$, with
$|d\zeta/dx|=\sqrt m/(2x^{3/2})\leq C_{m,\zeta_-,\zeta_+}$.
For a null exceptional set $E$ in the corresponding $x$ interval,
\[
 |\{\zeta:x(\zeta)\in E\}|
   =\int_E\frac{\sqrt m}{2x^{3/2}}\,dx=0.
\]
For a reduced tuple of length $\ell$, Proposition~\ref{prop:scalar}
therefore gives, outside this null set,
\begin{align*}
 \left|\sum_j c_j\omega_j(\zeta\bar g)\right|
 &=\zeta\left|\sum_j c_j(\bar g^2j^4+x)^{1/2}\right|\\
 &\geq\zeta\gamma(x)^\ell(C\ell N)^{-\tau\ell^2}
 \geq[\gamma(x)\min(1,\zeta_-)]^\ell
                                      (C\ell N)^{-\tau\ell^2}.
\end{align*}
Thus the metric change retains $p=2$. Taking $\bar g=1$ and
the countable cover $[1/r,r]$, $r\geq2$, proves the assertion
for almost every $g>0$.

The distinct absolute frequencies are strictly increasing for every $g,m>0$. The derivative
\begin{equation}\label{eq:beamgapderivative}
 \frac{d}{dx}\sqrt{g^2x^4+m}
 =\frac{2g^2x^3}{\sqrt{g^2x^4+m}}\asymp x
\end{equation}
gives \eqref{eq:radialgap} with $\beta=2$, $\delta=1$ by Lemma~\ref{lem:1dgap}. Thus, for the full-measure set of $g$, all hypotheses hold with $p=2$, and the two times in \eqref{eq:p2times} apply to \eqref{eq:beam}.

For real beam data the precise physical norm is
\begin{equation}\label{eq:beamphysicalnorm}
 \nrm u_s^2=\frac12\left(\nrm{A_g^{1/4}\psi}_s^2+
                              \nrm{A_g^{-1/4}\psi_t}_s^2\right).
\end{equation}
Indeed, using the reality of $\psi,\psi_t$ and the even weights,
\[
 \operatorname{Im}(A_g^{1/4}\psi,A_g^{-1/4}\psi_t)_s
 =\operatorname{Im}\sum_jw_s(j)^2\psi_j\overline{(\psi_t)_j}=0,
\]
because the terms at $j$ and $-j$ are complex conjugates.
Expansion of \eqref{eq:beamcoords} then gives
\eqref{eq:beamphysicalnorm}, with the same amplification criterion
as in Theorem~\ref{thm:joint}.

\subsubsection{The two stability conclusions}
\begin{theorem}[The metric beam equation in Gevrey regularity]\label{ex3thm1}
Fix $m>0$, $0<\theta<1$, $s>0$ and $A>1$.
For almost every $g>0$, all sufficiently small real Cauchy data
for \eqref{eq:beam}, of size
\[
 \varepsilon^2=\tfrac12\bigl(\|A_g^{1/4}\psi_0\|_{W^G_{s,\theta}}^2
                     +\|A_g^{-1/4}\phi_0\|_{W^G_{s,\theta}}^2\bigr),
\]
give a unique solution whose corresponding norm remains smaller
than $A\varepsilon$ on the interval
\[|t|\leq\exp(c_AL^{3/2}/(\log L)^{1/2}).\]
\end{theorem}
\begin{theorem}[The metric beam equation in logarithmic regularity]\label{ex3thm2}
Fix $m>0$, $q>1$, $\kappa_*>\ee^q$, $s>0$ and $A>1$.
For the same full-measure set of metrics, sufficiently small real
data with size given by \eqref{eq:beamphysicalnorm} in $W^U_{s,q}$
produce a unique solution whose norm remains smaller than
$A\varepsilon$ for
$|t|\leq\exp(c_AL^{1+(q-1)/(2q+1)})$.
\end{theorem}
\begin{proof}[Proof of Theorems~\ref{ex3thm1} and \ref{ex3thm2}]
Use~\eqref{eq:beam-modal-inverse-full} and~\eqref{eq:beam-majorant-full}. The scalar nodes
$\bar g^2n^4$ have $b=4$, $b'=0$ and gap at least $\bar g^2$.
Proposition~\ref{prop:scalar} and the pullback of null sets under \eqref{eq:metricscaling} give $p=2$ for almost
every $g>0$. Formula~\eqref{eq:beamgapderivative} gives
$\beta=2$, $\delta=1$.
Take $s_0=s/2$ and apply Theorem~\ref{thm:joint}; substitution
of $p=2$ gives the displayed times. The identity
\eqref{eq:beamphysicalnorm} returns its exact norm estimate to
the real Cauchy variables, and the inverse of the canonical
transformation returns the solution to \eqref{eq:beam}.
\end{proof}

\subsubsection{Comparison in the same mass family and Cauchy norm}\label{subsec:beam-comparison}
To compare with \cite[Theorem 1.4]{FM23}, fix $g=1$, take
$F(\psi,v)=F(\psi)$ and vary $m\in[1,2]$. In
Proposition~\ref{prop:scalar} we now use
\[
 a_n=n^4,\qquad x=m,\qquad \eta=\tfrac12,
 \qquad 0\leq a_n\leq N^4,\quad |a_n-a_k|\geq1\ (n\ne k).
\]
The proposition gives $p=2$ for almost every $m\in[1,2]$.
The coefficient calculation \eqref{eq:beam-majorant-full}
and the gap \eqref{eq:beamgapderivative} apply with these
fixed values of $g$ and $m$.

For a comparison at the same mass, let $\mathcal G_\gamma$
be the good set in the proof of Proposition~\ref{prop:scalar},
and let $\mathcal M^{\rm FM}_\gamma$ be the set supplied by
\cite[Theorem 1.4]{FM23}. For every sufficiently small $\gamma>0$,
\begin{align}\label{eq:beam-common-good-set}
 \mathcal M_\gamma
   &:=\mathcal G_\gamma\cap\mathcal M^{\rm FM}_\gamma,\nonumber\\
 |[1,2]\setminus\mathcal M_\gamma|
   &\leq |[1,2]\setminus\mathcal G_\gamma|
        +|[1,2]\setminus\mathcal M^{\rm FM}_\gamma|
    \leq(C_{\rm sc}+C_{\rm FM})\gamma.
\end{align}
Choose $j_0$ so that these estimates hold for
$\gamma=2^{-j}$, $j\geq j_0$, and put
$\mathcal M_*:=\bigcup_{j\geq j_0}\mathcal M_{2^{-j}}$.
Then
\[
 |[1,2]\setminus\mathcal M_*|
 \leq\inf_{j\geq j_0}(C_{\rm sc}+C_{\rm FM})2^{-j}=0.
\]
For each fixed $m\in\mathcal M_*$, choose one $j$ with
$m\in\mathcal M_{2^{-j}}$. Both the earlier time bound and the
bound below then hold at this mass, with their constants and
smallness thresholds fixed after this choice.

For the Cauchy norm of \cite{FM23}, set
\[
 b_j=\max(2,|j|),\qquad
 \lambda_j=[\log(2+\max(1,|j|))]^q,
 \qquad 1<q\leq2,
\]
and, for $s>0$ and $h\in\RR$, define
\begin{equation}\label{eq:FM-cauchy-norm}
 \|v\|_{s,h}^{\rm FM}
 =\left(\sum_j b_j^{2h}\ee^{2s\lambda_j}|v_j|^2\right)^{1/2},
 \qquad
 \mathcal N_{s,p_0}(\psi,\phi)
 =\|\psi\|_{s,p_0+1}^{\rm FM}
                    +\|\phi\|_{s,p_0-1}^{\rm FM},\quad p_0>3/2.
\end{equation}
The corresponding complex norm is
$\|u\|_{\rm FM,c}^2=\sum_jb_j^{2p_0}\ee^{2s\lambda_j}|u_j|^2$.
It is exactly the norm of Proposition~\ref{prop:radial-weights}
with $a=p_0$ and
\[
 \chi_j=(b_j/\br j)^{p_0}
       \exp\!\left(s\lambda_j-s[\log(\br j+\kappa_*)]^q\right).
\]
Indeed, $\chi_j$ is positive and even, and $\chi_j\to1$ as
$|j|\to\infty$, since the mean value theorem gives
\[
 \left|[\log(|j|+2)]^q-[\log(|j|+\kappa_*)]^q\right|
 \leq C_{q,\kappa_*}\frac{(\log|j|)^{q-1}}{|j|}
 \longrightarrow0.
\]
Thus $0<\inf_j\chi_j\leq\sup_j\chi_j<\infty$.

For $\omega_j=(j^4+m)^{1/2}$, $1\leq m\leq2$, we have
$1/4\leq\omega_j/b_j^2\leq4$. The real-slice calculation
in \eqref{eq:beamphysicalnorm} gives
\begin{align*}
 \|u\|_{\rm FM,c}^2
 &=\frac12\sum_jb_j^{2p_0}\ee^{2s\lambda_j}
                  (\omega_j|\psi_j|^2+\omega_j^{-1}|\phi_j|^2),\\
 \frac18\left[(\|\psi\|_{s,p_0+1}^{\rm FM})^2
               +(\|\phi\|_{s,p_0-1}^{\rm FM})^2\right]
 &\leq\|u\|_{\rm FM,c}^2
 \leq2\left[(\|\psi\|_{s,p_0+1}^{\rm FM})^2
               +(\|\phi\|_{s,p_0-1}^{\rm FM})^2\right].
\end{align*}
Consequently
\begin{equation}\label{eq:FM-norm-alignment}
 \tfrac14\mathcal N_{s,p_0}(\psi,\phi)
 \leq\|u\|_{\rm FM,c}
 \leq\sqrt2\mathcal N_{s,p_0}(\psi,\phi).
\end{equation}
For each mass $m\in\mathcal M_*$, apply
Proposition~\ref{prop:radial-weights} with $A=2$. If
$\mathcal N_{s,p_0}(\psi_0,\phi_0)\leq\delta/4$ and
$\delta$ is sufficiently small, then
\begin{align*}
 \mathcal N_{s,p_0}(\psi(t),\psi_t(t))
 &\leq4\|u(t)\|_{\rm FM,c}
 <8\|u_0\|_{\rm FM,c}
 \leq2\sqrt2\delta<8\delta,\\
 |t|&\leq\exp\!\left(c[\log(1/\delta)]^{3q/(2q+1)}\right).
\end{align*}
For every $m\in\mathcal M_*$, this gives the longer stability
interval with the initial-data threshold and Cauchy norm of
\cite[Theorem 1.4]{FM23}.

The six application theorems have the following frequency data.
\begin{center}\small
\begin{tabular}{lccc p{.31\textwidth}}
\toprule
Equation & $p$ & $\beta$ & $\delta$ & Full-measure parameter set\\\midrule
Convolution NLS & $1$ & $2$ & $1$ & Even convolution potentials\\[2pt]
Fractional NLS & $2$ & $2\eta$ & $2\eta-1$ & Masses $m>0$, fixed $\eta$\\[2pt]
Metric beam equation & $2$ & $2$ & $1$ & Metrics $g>0$, fixed $m$\\\bottomrule
\end{tabular}
\end{center}
Theorems~\ref{ex1thm1}--\ref{ex3thm2} hold for each fixed analytic
nonlinearity on the indicated full-measure set, selected independently
of the initial datum. Their constants may depend on this fixed
equation. The solution norms are \eqref{eq:physicalnorm} and its
exact expression \eqref{eq:beamphysicalnorm} in the beam variables.

\section{A common frequency construction and exact invariant modules}\label{sec:implant}
We begin the upper-bound construction with the condition that is common
to both examples: the spectrum must admit the desired transfers and
still satisfy the arithmetic hypothesis of the lower theorem. To this
end we prescribe one small divisor on each of a sequence of disjoint
modules, while controlling every other integer frequency combination.
After constructing the spectrum, we prove the phase-locking identity
and the analytic convergence of the module sum. The choices specific to
each regularity weight are made in the following sections.
We construct the upper examples in the same Fourier phase space on $\mathbb T$. All nonlinear modules have disjoint sets of absolute indices. Frequencies are even, $\omega_j=j^2+V_{|j|}$, and the perturbations below have zero spatial momentum. For a spatial relation $q_n$ define its compression by
$\widehat q_{n,0}=q_{n,0}$ and
$\widehat q_{n,k}=q_{n,k}+q_{n,-k}$ for $k\geq1$.
Then $q_n\cdot\omega=\widehat q_n\cdot\lambda$.
The next lemma constructs the frequencies from these compressed vectors.

\begin{lemma}[Implanting infinitely many small divisors]\label{lem:implant}
Let $\widehat q_n=(\widehat q_{n,j})_{j\geq0}$ be nonzero integer vectors of finite, pairwise disjoint supports, disjoint from $0$. Suppose
\[
 \sum_j \widehat q_{n,j}j^2=0,\qquad \widehat q_{n,\Lambda_n}=1,\qquad
 \Lambda_n=\max\supp \widehat q_n,\qquad D_n=\sum_j|\widehat q_{n,j}|,
\]
where the supports are ordered, $D_n\geq n+3$ and $\Lambda_n\geq2^n$. Assume, for some fixed $\mu>0$, that
\begin{equation}\label{eq:pivotdecay}
 \sum_{j\ne\Lambda_n}|\widehat q_{n,j}|\br j^{-\mu}\longrightarrow0.
\end{equation}
Set $\lambda_j=j^2+V_j$ for $j\geq0$. Fix $p\geq1$ and $\tau_0>0$. After discarding finitely many modules, there are a bounded real sequence $V$, constants $\gamma>0$, $\tau_1>\tau_0$, and $\xi_n\in[1,2]$ such that
\begin{align}
 \widehat q_n\cdot\lambda&=-\kappa_n,
 &\kappa_n&=\xi_n\Lambda_n^{-\tau_0D_n^p},\label{eq:implanted}\\
 |k\cdot\lambda|&\geq\gamma N(k)^{-\tau_1R(k)^p}
 &&(0\ne k\in\ZZ^{(\NN_0)}),\label{eq:fullNR}
\end{align}
where $R(k)=\sum|k_j|$ and $N(k)=\max(2,\max\supp k)$. One can require $3/16\leq V_0\leq5/16$ and $|V_j|\leq1/8$ for $j\geq1$.
\end{lemma}
\begin{proof}
Give each nonpivot coordinate the independent law
\[
 V_j=\bar V_j+\rho\br j^{-\mu}x_j,\qquad x_j\in[-1,1],\qquad
 \bar V_0=1/4,\quad\bar V_j=0\ (j\geq1),\quad0<\rho\leq1/16.
\]
Use independent uniform $\xi_n\in[1,2]$ and define the pivot by
\begin{equation}\label{eq:pivot}
 V_{\Lambda_n}=-\kappa_n-\sum_{j\ne\Lambda_n}\widehat q_{n,j}V_j.
\end{equation}
The quadratic relation and \eqref{eq:pivot} give
\[
 \widehat q_n\cdot\lambda
 =\sum_j\widehat q_{n,j}j^2
       +V_{\Lambda_n}+\sum_{j\ne\Lambda_n}\widehat q_{n,j}V_j
 =-\kappa_n.
\]
Since the supports avoid zero, \eqref{eq:pivotdecay} gives
\[
 |V_{\Lambda_n}|
 \leq2\Lambda_n^{-\tau_0D_n^p}
       +\rho\sum_{j\ne\Lambda_n}|\widehat q_{n,j}|\br j^{-\mu}
       \longrightarrow0.
\]
Discarding finitely many modules therefore ensures
$|V_{\Lambda_n}|\leq1/8$. For free coordinates,
$|V_j|\leq1/16$ when $j\ne0$, and
$3/16\leq1/4-\rho\br0^{-\mu}\leq V_0
\leq1/4+\rho\br0^{-\mu}\leq5/16$.
All probability estimates below use the countable product of the
specified interval laws and one-coordinate conditioning.

For a finitely supported integer vector $k$, eliminating the pivots gives
\begin{equation}\label{eq:eliminate}
 k\cdot\lambda=\sum_jk_jj^2-\sum_nk_{\Lambda_n}\kappa_n
 +\sum_{j\text{ free}}\left(k_j-\sum_n k_{\Lambda_n}\widehat q_{n,j}\right)V_j.
\end{equation}
We estimate the exceptional sets for the following two disjoint classes of integer vectors.

\emph{Transverse case.} If some free coefficient in parentheses is nonzero, it is an integer of modulus at least one. The selected free coordinate has $j\leq N(k)$, even when it was introduced by elimination: every nonpivot of a module is smaller than its pivot. The slope and the conditional probability satisfy
\begin{align*}
 \left|\partial_{x_j}(k\cdot\lambda)\right|
 &=\rho\br j^{-\mu}
       \left|k_j-\sum_n k_{\Lambda_n}\widehat q_{n,j}\right|
 \geq c\rho N(k)^{-\mu},\\
 \mathbb P\!\left(|k\cdot\lambda|<\gamma N(k)^{-\tau_1R(k)^p}
          \,\middle|\,(x_i)_{i\ne j},(\xi_n)_n\right)
 &\leq C_\rho\gamma N(k)^{\mu-\tau_1R(k)^p}.
\end{align*}
For given $R,N$, the signed-atom count is
\[
 \#\{k:R(k)=R,\ N(k)=N\}\leq[2(N+1)]^R\leq(4N)^R.
\]
Choosing $\tau_1>\mu+8$ gives
\begin{equation}\label{eq:transsum}
 \mathbb P(\text{some transverse failure})
 \leq C_\rho\gamma\sum_{R\geq1}\sum_{N\geq2}4^R N^{\mu+R-\tau_1R^p}
 \leq C_\rho'\gamma.
\end{equation}
For the convergence, use $R^p\geq R$ and
$N^\mu\leq N^{\mu R}$ to obtain, with $\tau_1>\mu+8$,
\[
 4^RN^{\mu+R-\tau_1R^p}\leq4^RN^{-7R}
                      \leq2^{-R}N^{-3}\quad(R\geq1,\ N\geq2).
\]
Both sums are therefore finite. The signed-atom count includes
relations of orders one and two as well as the higher orders.

\emph{Kernel case.} Vanishing of every free coefficient is equivalent to
\[
 k=\sum_n A_n\widehat q_n,\qquad A_n=k_{\Lambda_n}\in\ZZ.
\]
Disjoint supports give $R=\sum_n|A_n|D_n$. If $A_\ell$ is
the last nonzero coefficient, then $N=\Lambda_\ell$ and
\[
 \left|\partial_{\xi_\ell}(k\cdot\lambda)\right|
       =|A_\ell|\Lambda_\ell^{-\tau_0D_\ell^p}.
\]
Conditioning on the other parameters gives
\begin{equation}\label{eq:kernelone}
 \mathbb P\!\left(|k\cdot\lambda|<\gamma\Lambda_\ell^{-\tau_1R^p}
                  \,\middle|\,(\xi_n)_{n\ne\ell},(x_j)_{j\text{ free}}\right)
 \leq\frac{C\gamma}{|A_\ell|}
       \Lambda_\ell^{-[\tau_1R^p-\tau_0D_\ell^p]}.
\end{equation}
Put $\Delta\tau=\tau_1-\tau_0>0$. Since $x\mapsto x^p$ has slope at least one on $[1,\infty)$,
\[
 \tau_1R^p-\tau_0D_\ell^p\geq
 \Delta\tau D_\ell^p+\tau_1\left[(|A_\ell|-1)D_\ell+
                         \sum_{n<\ell}|A_n|D_n\right].
\]
Set $z=\Lambda_\ell^{-\tau_1D_\ell}$.
After extracting $\Lambda_\ell^{-\Delta\tau D_\ell^p}$,
the sum over the last nonzero coefficient is bounded by
$2\sum_{a\geq1}z^{a-1}/a\leq2/(1-z)\leq C$.
For an earlier index the exact sum is
$\sum_{A\in\ZZ}\Lambda_\ell^{-\tau_1D_n|A|}
=(1+\Lambda_\ell^{-\tau_1D_n})/(1-\Lambda_\ell^{-\tau_1D_n})$. The sums over all earlier $A_n$ are bounded by
\begin{equation}\label{eq:kernelproduct}
 \prod_{n<\ell}\frac{1+\Lambda_\ell^{-\tau_1D_n}}
                         {1-\Lambda_\ell^{-\tau_1D_n}}
 \leq\prod_{n\geq1}\frac{1+2^{-\tau_1(n+3)}}{1-2^{-\tau_1(n+3)}}<\infty.
\end{equation}
Using $\log((1+x)/(1-x))\leq4x$ for $0\leq x\leq1/2$,
we obtain the uniform product bound
\[
 \log\prod_{n\geq1}
       \frac{1+2^{-\tau_1(n+3)}}{1-2^{-\tau_1(n+3)}}
 \leq4\sum_{n\geq1}2^{-\tau_1(n+3)}
 =\frac{4\,2^{-4\tau_1}}{1-2^{-\tau_1}}<\infty.
\]
Finally,
\begin{equation}\label{eq:kernelsum}
 \mathbb P(\text{some kernel failure})
 \leq C\gamma\sum_{\ell\geq1}\Lambda_\ell^{-\Delta\tau D_\ell^p}
 \leq C\gamma\sum_{\ell\geq1}2^{-\ell\Delta\tau(\ell+3)^p}\leq C_{\tau_0,\tau_1,p}\gamma.
\end{equation}
Take $\tau_1>\max(\tau_0,\mu+8)$ and then $\gamma$ small enough that \eqref{eq:transsum} and \eqref{eq:kernelsum} sum to less than one. A point outside this union exists and is fixed henceforth. This proves all inequalities simultaneously, including every integral combination of the prescribed small-divisor relations.
\end{proof}

Extend the spectrum by $\omega_j=\lambda_{|j|}$.
For a signed tuple of length $m$ below a cutoff $K\geq m$,
let $k_n$ be the number of positive-sign occurrences at absolute
index $n$ minus the number of negative-sign occurrences there.
Then $\Omega=k\cdot\lambda$, $R(k)\leq m$ and $N(k)\leq K$.
If $k\ne0$, reducing the fixed $\gamma$ to be less than one gives
\begin{equation}\label{eq:implant-to-abstract}
 |\Omega|\geq\gamma N(k)^{-\tau_1R(k)^p}
 \geq\gamma K^{-\tau_1m^p}
 \geq\gamma^m(CmK)^{-\tau_1m^p}.
\end{equation}
If $k=0$, the positive- and negative-sign multiplicities agree
at every absolute index. Matching these occurrences gives exactly
the pairing in Assumption~\ref{ass:freq}.

The amplitude restrictions on $V$ also give a global frequency
gap directly. For $n>k\geq1$ and for $n>k=0$, respectively,
\[
 \lambda_n-\lambda_k\geq n^2-k^2-1/4,
 \qquad \lambda_n-\lambda_0\geq n^2-7/16.
\]
Both are at least $\tfrac9{16}(n^2-k^2)\geq\tfrac9{16}n$.
Since $\br n+\br k\leq2c_0n$ for $n>k\geq0$, this proves
\eqref{eq:radialgap} with $\delta=1$ and
$c_{\rm gap}=9/(32c_0)$. Finally $\lambda_n=n^2+O(1)$
gives $\beta=2$. Thus all frequency hypotheses, including their
finite low-index part, hold for the fixed implanted spectrum.

\begin{lemma}[Canonical phase locking]\label{lem:locking}
On a finite module let $q=(q_j)$ be an integer vector and
\[
 H=\sum_j\omega_j I_j+P_{\rm sel}(I)+M+\bar M,
 \quad M=\prod_{q_j>0}u_j^{q_j}\prod_{q_j<0}\bar u_j^{-q_j}.
\]
Use $u_j=\sqrt{I_j}\exp(-\ii\theta_j)$ and $\psi=-\sum q_j\theta_j$. Suppose $q\cdot\omega=-\kappa$, and $q\cdot\nabla P_{\rm sel}=K(I)$, where $q\cdot\nabla K=0$. If $K(I(0))=\kappa$ and $\psi(0)=-\pi/2$, then, while all active actions are positive,
\begin{equation}\label{eq:locked}
 \psi(t)=-\pi/2,\qquad I_j(t)=I_j(0)+q_jx(t),\qquad
 \dot x=2\prod_j(I_j(0)+q_jx)^{|q_j|/2}>0.
\end{equation}
\end{lemma}
\begin{proof}
The substitution $u_j=\sqrt{I_j}\exp(-\ii\theta_j)$ transforms
the symplectic form into $\sum_jd\theta_j\wedge dI_j$.
Thus $\dot I_j=-\partial_{\theta_j}H$ and
$\dot\theta_j=\partial_{I_j}H$.
Writing $\mathcal A(I)=\prod I_j^{|q_j|/2}$ gives $M=\mathcal A\exp(\ii\psi)$. The canonical equations are
\[
 \dot I_j=-2q_j\mathcal A\sin\psi,\qquad
 \dot\psi=\kappa-K-\mathcal A\cos\psi\sum_j\frac{q_j|q_j|}{I_j}.
\]
Moreover $\dot K=-2\mathcal A\sin\psi\,q\cdot\nabla K=0$. The set $K=\kappa$, $\psi=-\pi/2$ is invariant by uniqueness of the smooth finite-dimensional equations in the positive-action region, and the first equations integrate as stated. The targets in the two constructions lie strictly inside this positive-action region.
\end{proof}

\subsection{Analytic direct sums and the semilinear flow}
We record the convergence and flow statement used for both
constructions. Here $W_j=w_s(j)^2$, $\mathcal E_n$ is a finite
coordinate block, and the blocks have disjoint spatial supports.
The number of coordinates in each block is at most eight.

\begin{lemma}[Weight compensation for a monomial output]\label{lem:module-weight}
Let $M$ be a zero-momentum monomial of degree $D\geq3$.
If differentiation with respect to $\bar u_j$ leaves input indices
$j_1,\ldots,j_{D-1}$, repeated according to their multiplicities,
then
\begin{equation}\label{eq:module-weight-compensation}
 W_j^{1/2}\leq\prod_{i=1}^{D-1}W_{j_i}^{1/2}.
\end{equation}
The same assertion holds for the conjugate output.
\end{lemma}
\begin{proof}
Momentum gives $|j|\leq\sum_i|j_i|$ and hence
$\br j\leq\sum_i\br {j_i}$. The strict subadditivity assumption
and $C_f<1$ imply
$f(\br j)\leq f(\sum_i\br {j_i})\leq\sum_i f(\br {j_i})$.
Exponentiation after multiplication by $s$ proves the inequality.
Replacing some factors by direction vectors when differentiating
again leaves exactly the same inequality between their coordinate
weights.
\end{proof}

\begin{proposition}[Convergence of the module Hamiltonians]\label{prop:module-analytic}
Suppose $M_n$ is a zero-momentum monomial of degree $D_n$, with
coefficient one and support in $\mathcal E_n$, where
$D_n\geq c_1n+c_2$, $c_1>0$. The transfer field associated with
$\sum_n(M_n+\bar M_n)$ is real analytic on the physical open unit
ball of $W_s$. Its series and every fixed-order derivative converge
uniformly on each ball of radius $R<1$.
The selector fields used in \eqref{eq:Gham} and \eqref{eq:Uham}
are continuous cubic polynomials, with the same local convergence
property. Their scalar Hamiltonian series converge there as well.
\end{proposition}
\begin{proof}
For the transfer terms, put $z_j=W_j^{1/2}|u_j|$.
We estimate tails in the module index itself. The degree growth
$D_n\geq c_1n+c_2$ provides uniform summability of the transfer
terms, while the selectors below use the increasing coordinate
weights. These are the additional estimates needed for uniform
module convergence in Remark~\ref{rem:Fourier-convergence}.
Each nonzero output contains at most two contributions, one from
$M_n$ and one from its conjugate. Differentiation gives an integer
factor at most $D_n$ and leaves $D_n-1$ inputs. By
Lemma~\ref{lem:module-weight}, for $\|u\|_s\leq R$,
\begin{equation}\label{eq:module-single-field}
 \|X_{M_n+\bar M_n}(u)\|_s\leq C D_n R^{D_n-1}.
\end{equation}
The fixed factor $C$ includes the square root of the number of
coordinates in a block. Taking $a$ directional derivatives chooses
at most $D_n^a$ further positions among these inputs. Thus for
$0\leq a\leq D_n-1$,
\begin{equation}\label{eq:module-single-derivative}
 \|D^aX_{M_n+\bar M_n}(u)[h_1,\ldots,h_a]\|_s
 \leq C_aD_n^{a+1}R^{D_n-1-a}\prod_{i=1}^a\|h_i\|_s.
\end{equation}
For larger $a$ the derivative vanishes.

For $b\geq0$ and $0<t<1$,
\[
 \frac d{dx}\log(x^bt^x)=\frac bx+\log t<0
                       \quad\text{if }x>b/|\log t|.
\]
Hence, for all sufficiently large $n$,
\[
 D_n^bt^{D_n}\leq(c_1n+c_2)^bt^{c_1n+c_2},\qquad
 \frac{(c_1(n+1)+c_2)^bt^{c_1(n+1)+c_2}}
      {(c_1n+c_2)^bt^{c_1n+c_2}}
     \longrightarrow t^{c_1}<1.
\]
Consequently
\begin{equation}\label{eq:degree-series-convergence}
 \sum_nD_n^bt^{D_n}<\infty.
\end{equation}
Orthogonality of the output blocks now gives the uniform tail bound
\begin{equation}\label{eq:module-tail-field}
 \sup_{\|u\|_s\leq R}
 \left\|\sum_{n>N_*}X_{M_n+\bar M_n}(u)\right\|_s^2
 \leq C\sum_{n>N_*}D_n^2R^{2D_n-2}\longrightarrow0.
\end{equation}
Applying the same orthogonality for fixed direction vectors of norm
at most one gives
\begin{equation}\label{eq:module-tail-derivative}
 \sup_{\|u\|_s\leq R}
 \left\|\sum_{n>N_*}D^aX_{M_n+\bar M_n}(u)\right\|_{\mathcal L^a}^2
 \leq C_a\sum_{n>N_*}D_n^{2a+2}R^{2D_n-2-2a}\longrightarrow0.
\end{equation}
For the scalar transfer series,
\[
 \sup_{\|u\|_s\leq R}\sum_{n>N_*}|M_n(u)+\bar M_n(u)|
 \leq2\sum_{n>N_*}R^{D_n}\longrightarrow0,
\]
by \eqref{eq:degree-series-convergence} with $b=0$.

We verify analyticity by a locally convergent Taylor expansion.
Replace $u,\bar u$ by independent vectors $z,\zeta$ and use
the product norm $\max(\|z\|_s,\|\zeta\|_s)$.
The preceding component bounds use absolute values and remain valid
on this complex product ball. Around a center of norm at most
$R_0<1$, take an increment of norm at most $\delta$ with
$R_0+\delta<1$. Each transfer field is a continuous homogeneous
polynomial of degree $D_n-1$. Expanding its inputs gives
\[
 \sum_{a=0}^{D_n-1}\frac1{a!}
  \|D^aX_{M_n+\bar M_n}(z,\zeta)[h^a]\|_s
 \leq CD_n(R_0+\delta)^{D_n-1}.
\]
Summing over $n$ gives
\[
 \sum_n\sum_{a=0}^{D_n-1}\frac1{a!}
       \|D^aX_{M_n+\bar M_n}(z,\zeta)[h^a]\|_s
 \leq C\sum_nD_n(R_0+\delta)^{D_n-1}<\infty.
\]
The double series can therefore be rearranged into the Taylor series
of the total field. Restriction to $\zeta=\bar z$ proves real
analyticity on the physical unit ball.

For the selectors, let $r_n=\|\Pi_{\mathcal E_n}u\|_s$ and
$\underline W_n=\min_{j\in\mathcal E_n}W_j$. Their explicit
coordinate formulas below consist of a fixed number of terms
$\pm c|u_j|^2u_k$, $|c|\leq1$, within one block.
Consequently
\[
 |P_{{\rm sel},n}(u)|\leq C\underline W_n^{-2}r_n^4,
 \qquad
 \|X_{P_{{\rm sel},n}}(u)\|_s
              \leq C\underline W_n^{-1}r_n^3.
\]
Their cubic polarizations obey the same weight factor. For directions
in different input slots, the elementary inequality
$\sum_n\prod_{i=1}^k a_{i,n}^2
\leq\prod_{i=1}^k\sum_n a_{i,n}^2$ bounds the direct sum of
these multilinear maps. Thus, for $0\leq a\leq3$,
\begin{equation}\label{eq:selector-tail-full}
 \sup_{\|u\|_s\leq R}
 \left\|D^a\sum_{n>N_*}X_{P_{{\rm sel},n}}(u)\right\|_{\mathcal L^a}
 \leq C_aR^{3-a}\sup_{n>N_*}\underline W_n^{-1}.
\end{equation}
The constructions satisfy $\underline W_n\to\infty$, so
$\sup_{n>N_*}\underline W_n^{-1}\to0$. For $a>3$,
$D^aX_{P_{{\rm sel},n}}=0$. The scalar estimate is
\[
 \sup_{\|u\|_s\leq R}\sum_{n>N_*}|P_{{\rm sel},n}(u)|
 \leq C\sup_{n>N_*}\underline W_n^{-2}\sum_nr_n^4
 \leq CR^4\sup_{n>N_*}\underline W_n^{-2}\longrightarrow0.
\]
The sum of selector fields is a continuous cubic polynomial and has
the finite Taylor expansion
\[
 X_{P_{\rm sel}}(u+h)
   =\sum_{a=0}^3\frac1{a!}D^aX_{P_{\rm sel}}(u)[h^a].
\]
Together with \eqref{eq:module-tail-derivative}, the case $a=1$
gives, for every $R<1$,
\[
 L_R:=\sup_{B_s(R)}\|DX_P\|<\infty.
\]
For $u,v\in B_s(R)$,
\[
 \|X_P(u)-X_P(v)\|_s
 \leq\int_0^1\|DX_P(v+t(u-v))\|\|u-v\|_s\,dt
 \leq L_R\|u-v\|_s.
\]
\end{proof}

\begin{proposition}[Invariant modules and global small solutions]\label{prop:module-flow}
Suppose the coordinates in each module are partitioned into groups
$G$ satisfying, with one constant $C_W\geq1$,
\[
 \frac{d}{dt}\sum_{j\in G}I_j=0,\qquad
 \frac{\max_{j\in G}W_j}{\min_{j\in G}W_j}\leq C_W^2.
\]
Then the module equations define a unique local mild flow on the
physical unit ball. Throughout its existence interval,
\begin{equation}\label{eq:globalbound}
 \|u(t)\|_s\leq C_W\|u(0)\|_s.
\end{equation}
If $C_W\|u(0)\|_s<1$, it is global. Each module and each union
of modules is exactly invariant, and finite-support data give global
classical solutions.
\end{proposition}
\begin{proof}
The diagonal group and the Duhamel contraction are those proved in
Subsection~\ref{subsec:flow}, now on every fixed radius $\rho<1$ by
Proposition~\ref{prop:module-analytic}. Applying a coordinate
projection to \eqref{eq:Duhamel-full} gives
\[
 u_j(t)=\ee^{-\ii\omega_j(t-t_0)}u_j(t_0)
       +\int_{t_0}^t\ee^{-\ii\omega_j(t-\xi)}(X_P(u(\xi)))_j\,d\xi.
\]
The integrand is continuous, so each coordinate is $C^1$ and
satisfies its classical scalar equation. The equation for a module
only contains that module's coordinates. At its zero vector the
field is zero, so finite-dimensional uniqueness gives invariance.
Coordinates outside all modules make linear rotations.

For one conserved group $G$ let $m_G=\min_{j\in G}W_j$ and
$M_G=\max_{j\in G}W_j$. Its action sum is constant, hence
\[
 \sum_{j\in G}W_jI_j(t)
 \leq M_G\sum_{j\in G}I_j(0)
 \leq(M_G/m_G)\sum_{j\in G}W_jI_j(0).
\]
Summing these nonnegative inequalities over all groups and applying
monotone convergence gives \eqref{eq:globalbound}.
If $r_*=C_W\|u(0)\|_s<1$, fix $r_*<\rho<1$.
The Duhamel local time
$\min\{(\rho-r_*)/M_\rho,1/(2L_\rho)\}$ is positive and
independent of the restart time. Repeated continuation extends the
solution in both directions for all time. On a finite number of
modules the Cartesian equations are polynomial, and conservation
of the positive action sums bounds every coordinate, giving global
finite-dimensional solutions even without the small-data argument.
Finite Fourier support places them in the generator domain and in
every spatial smoothness class.
\end{proof}

The bound \eqref{eq:globalbound} also proves Lyapunov stability of
the origin for each constructed equation. Indeed, for every
$\delta>0$,
\[
 \|u(0)\|_s<\min\{(2C_W)^{-1},\delta/C_W\}
 \quad\Longrightarrow\quad
 \sup_{t\in\mathbb R}\|u(t)\|_s<\delta.
\]
For the trajectories constructed below, $1<A<C_W$ and the first
time at which $\|u(t)\|_s=A\|u(0)\|_s$ is determined by the
scalar transfer equation.

\section{The Gevrey example: construction and amplification time}\label{sec:Gupper}
We choose four Fourier modes whose transfer relation preserves both
momentum and the quadratic dispersion. A small detuning of that relation
will be cancelled by an action-dependent term. The spacing of the modes
is chosen so that this transfer increases the Gevrey norm by the
prescribed factor while keeping the transfer time on the lower-bound
scale.
Fix $s>0$, $0<\theta<1$, $p\geq1$, $A>1$ and $\tau_0>0$. Choose
\begin{equation}\label{eq:Glambda}
 \lambda>\frac{\log A}{s\theta}.
\end{equation}
\subsection{Recursive integer geometry and frequency parameters}
Choose rapidly increasing integers $N_n$ and positive odd integers
$g_n=\lambda N_n^{1-\theta}+O(1)$. Put
\begin{equation}\label{eq:Gindices}
 a_n=(g_n-1)/2,\quad c_n=(g_n+1)/2,\quad
 d_n=-N_n,\quad b_n=N_n+g_n,\quad m_n=2N_n+g_n.
\end{equation}
Choose each $N_n$ so large that $a_n>b_{n-1}$ and $N_n\geq2^n$. Increasing it further handles all finite low-index restrictions. The supports in absolute value are disjoint. Direct calculation gives both identities
\begin{align}
 m_n(a_n-c_n)+b_n-d_n&=0,\label{eq:Gmomentum}\\
 m_n(a_n^2-c_n^2)+b_n^2-d_n^2
 &=-m_ng_n+g_n(2N_n+g_n)=0.\label{eq:Gquadratic}
\end{align}
The integers in \eqref{eq:Gindices} can be chosen recursively with
all the stated requirements. For an integer $N$ let
$g(N)=2\lfloor\lambda N^{1-\theta}/2\rfloor+1$.
Then $g(N)$ is positive and odd for large $N$,
$|g(N)-\lambda N^{1-\theta}|\leq1$, $g(N)\to\infty$ and
$g(N)/N\to0$. Given the preceding module, choose $N_n$ so large
that $g(N_n)>2b_{n-1}+3$, $g(N_n)<N_n/2$,
$N_n\geq2^n$ and the frequency amplitude restrictions following \eqref{eq:pivot} hold. It follows that
\[
 b_{n-1}<a_n<c_n<N_n<b_n,
 \qquad a_n,c_n,N_n,b_n\longrightarrow\infty.
\]
The distinct absolute supports and their ordering follow from these
inequalities. Since $a_n-c_n=-1$, $a_n+c_n=g_n$,
$b_n-d_n=2N_n+g_n=m_n$ and $b_n+d_n=g_n$,
the two resonance identities can also be written as
\[
 m_n(-1)+m_n=0,\qquad
 m_n(-g_n)+(b_n-d_n)(b_n+d_n)=-m_ng_n+m_ng_n=0.
\]
The physical monomial uses the negative index $d_n=-N_n$;
compression of the even frequency at $d_n$ replaces that coordinate
by $N_n$ only in the frequency relation. The momentum calculation
continues to use $d_n$ with its physical sign.

The spatial relation and its absolute-index compression are
\[
 q_n=m_ne_{a_n}-m_ne_{c_n}+e_{b_n}-e_{d_n},\qquad
 \widehat q_n=m_ne_{a_n}-m_ne_{c_n}+e_{b_n}-e_{N_n}.
\]
The compressed vector has $D_n=2m_n+2$ and pivot $\Lambda_n=b_n$. For $\mu>1/(1-\theta)$,
\[
 \sum_{j\ne b_n}|\widehat q_{n,j}|\br j^{-\mu}
 \leq C m_ng_n^{-\mu}+N_n^{-\mu}
 =O(N_n^{1-\mu(1-\theta)})+O(N_n^{-\mu})\longrightarrow0.
\]
The free perturbations are absolutely summable because $\mu>1$. The pivot bounds are
\[|V_{b_n}|\leq C N_n^{1-\mu(1-\theta)}+C\kappa_n.\]
Their sum over $n$ is finite since $N_n\geq2^n$. Thus the frequency perturbation can also be chosen in $\ell^1(\ZZ)$. Lemma~\ref{lem:implant}, including its kernel-product count, supplies a single frequency sequence with
\begin{equation}\label{eq:Gkappa}
 m_n(\omega_{a_n}-\omega_{c_n})+\omega_{b_n}-\omega_{d_n}=-\kappa_n,
 \quad\kappa_n=\xi_n b_n^{-\tau_0(2m_n+2)^p},\quad1\leq\xi_n\leq2.
\end{equation}

The weights $W_j=\exp(2s|j|^\theta)$ may be used on these large indices. They obey
\begin{align}
 E_n:=W_{b_n}/W_{N_n}&\longrightarrow E:=\exp(2s\theta\lambda)>A^2,\label{eq:Gweights}\\
 W_{a_n}/W_{c_n}&\longrightarrow1,&
 W_{c_n}/W_{N_n}&=o(m_n^{-M})\quad\text{for every fixed }M>0.\nonumber
\end{align}
Indeed $b_n^\theta-N_n^\theta\to\theta\lambda$, $c_n^\theta-a_n^\theta=O(g_n^{\theta-1})\to0$, and $c_n^\theta=O(N_n^{\theta(1-\theta)})=o(N_n^\theta)$. These weight bounds will control the two conserved action sums of the Hamiltonian defined next.

Taylor's formula, with $g_n/N_n\to0$, gives
\begin{align}\label{eq:Gweight-errors}
 (N_n+g_n)^\theta-N_n^\theta
 &=\theta N_n^{\theta-1}g_n
             +O(N_n^{\theta-2}g_n^2)\nonumber\\
 &=\theta\lambda+O(N_n^{\theta-1})+O(N_n^{-\theta}).
\end{align}
Moreover
$c_n^\theta-a_n^\theta=\theta\xi_n'{}^{\theta-1}$ for some
$\xi_n'\in(a_n,c_n)$, and this tends to zero.
There is a constant $c>0$ such that
$N_n^\theta-c_n^\theta\geq cN_n^\theta$ for all large $n$.
Consequently, for every fixed $B>0$,
\begin{equation}\label{eq:G-superpolynomial-gap}
 m_n^B\frac{W_{c_n}}{W_{N_n}}
 \leq\exp(B\log m_n-2scN_n^\theta)\longrightarrow0.
\end{equation}
These formulas prove every limit in \eqref{eq:Gweights}, and also
the stronger products with polynomial powers of $m_n$ needed in the
phase and time estimates.

\subsection{Coordinate vector fields and analytic convergence}
Define the fixed Hamiltonian
\begin{equation}\label{eq:Gham}
 H_G=\sum_j\omega_j I_j+
 \frac12\sum_n(I_{a_n}+I_{c_n})(I_{b_n}-I_{d_n})
 +2\operatorname{Re}\sum_n u_{a_n}^{m_n}\bar u_{c_n}^{m_n}u_{b_n}\bar u_{d_n}.
\end{equation}
Reality and momentum follow from \eqref{eq:Gmomentum}; analyticity was checked in Proposition~\ref{prop:module-analytic}. On each module the masses
\[
 K_n=I_{a_n}+I_{c_n},\qquad J_n=I_{b_n}+I_{d_n}
\]
are conserved. For the physical signed vector $(m_n,-m_n,1,-1)$ the selector satisfies $q_n\cdot\nabla P_{\rm sel}=K_n$.
For one module suppress $n$ and write
$M=u_a^m\bar u_c^m u_b\bar u_d$.
The selector and transfer fields are, respectively,
\begin{align}\label{eq:G-selector-coordinates}
 (X_{\rm sel})_a&=-\tfrac\ii2(I_b-I_d)u_a,&
 (X_{\rm sel})_c&=-\tfrac\ii2(I_b-I_d)u_c,\nonumber\\
 (X_{\rm sel})_b&=-\tfrac\ii2(I_a+I_c)u_b,&
 (X_{\rm sel})_d&= \tfrac\ii2(I_a+I_c)u_d,
\end{align}
and
\begin{align}\label{eq:G-transfer-coordinates}
 (X_{\rm tr})_a&=-\ii m\bar u_a^{m-1}u_c^m\bar u_bu_d,&
 (X_{\rm tr})_c&=-\ii m u_a^m\bar u_c^{m-1}u_b\bar u_d,\nonumber\\
 (X_{\rm tr})_b&=-\ii\bar u_a^m u_c^m u_d,&
 (X_{\rm tr})_d&=-\ii u_a^m\bar u_c^m u_b.
\end{align}
For the selector, $K_n,J_n\leq\|u\|_s^2$, and
\eqref{eq:G-selector-coordinates} gives
\begin{align*}
 \|X_{\rm sel}(u)\|_s^2
 &=\frac14\sum_n\bigl[(I_{b_n}-I_{d_n})^2
                  (W_{a_n}I_{a_n}+W_{c_n}I_{c_n})
       +K_n^2(W_{b_n}I_{b_n}+W_{d_n}I_{d_n})\bigr]\\
 &\leq\tfrac14\|u\|_s^4
                 \sum_n\sum_{j\in\mathcal E_n}W_jI_j
  \leq\tfrac14\|u\|_s^6,\\
 |P_{\rm sel}(u)|
 &\leq\tfrac12\sum_nK_nJ_n
  \leq\tfrac12\left(\sum_nK_n\right)\left(\sum_nJ_n\right)
  \leq\tfrac12\|u\|_s^4.
\end{align*}

The four transfer components have the following weight coefficients.
After replacing $u_j$ by $W_j^{-1/2}z_j$, their coefficients
apart from the integers $m$ are
\begin{align}\label{eq:G-four-weight-coefficients}
 C_a&=\frac{W_a^{1/2}}{W_a^{(m-1)/2}W_c^{m/2}W_b^{1/2}W_{N}^{1/2}},&
 C_c&=\frac{W_c^{1/2}}{W_a^{m/2}W_c^{(m-1)/2}W_b^{1/2}W_{N}^{1/2}},
 \nonumber\\
 C_b&=\frac{W_b^{1/2}}{W_a^{m/2}W_c^{m/2}W_N^{1/2}},&
 C_d&=\frac{W_N^{1/2}}{W_a^{m/2}W_c^{m/2}W_b^{1/2}}.
\end{align}
Since $m\geq2$, $W_j\geq1$ and $b>N$, we have
\[
 C_a,C_c,C_d\leq1,\qquad
 C_b\leq(W_b/W_N)^{1/2}\leq C.
\]
Consequently, for $R<1$ and every fixed integer $a\geq0$,
\begin{align*}
 \sup_{\|u\|_s\leq R}\|X_{{\rm tr},n}(u)\|_s
       &\leq Cm_nR^{2m_n+1},\\
 \sup_{\|u\|_s\leq R}\|D^aX_{{\rm tr},n}(u)\|
       &\leq C_am_n^{a+1}R^{2m_n+1-a}.
\end{align*}
Here $D_n=2m_n+2\geq c n+c'$ after the recursive choice, so
\eqref{eq:module-tail-field}--\eqref{eq:module-tail-derivative}
prove uniform convergence of the full transfer series and its
derivatives. Proposition~\ref{prop:module-analytic} supplies the
locally convergent Taylor expansion, including the selector series.

The coordinate action equations from \eqref{eq:G-transfer-coordinates}
are
\begin{equation}\label{eq:G-Cartesian-actions}
 \dot I_a=-2m\operatorname{Im}M,\quad
 \dot I_c=2m\operatorname{Im}M,\quad
 \dot I_b=-2\operatorname{Im}M,\quad
 \dot I_d=2\operatorname{Im}M.
\end{equation}
The diagonal linear and selector terms have zero contribution to
these derivatives. Thus $K_n=I_a+I_c$ and $J_n=I_b+I_d$ are
conserved for every module trajectory. Explicitly,
\[
 \|u^{(n)}(t)\|_s^2
 \leq W_{c_n}K_n+W_{b_n}J_n
 \leq C_W^2(W_{a_n}K_n+W_{N_n}J_n)
 \leq C_W^2\|u^{(n)}(0)\|_s^2,
\]
where $C_W^2=\max\{\sup_nW_{c_n}/W_{a_n},
\sup_nW_{b_n}/W_{N_n}\}<\infty$.
Proposition~\ref{prop:module-flow} now gives the exact invariant
module flow and the global bound.

\subsection{Phase locking and the first amplification time}\label{sec:Gdynamics}
We apply Lemma~\ref{lem:locking} to the four-mode relation
\eqref{eq:Gmomentum}. The calculations below specify its selector,
reconstruct the four coordinate phases, and determine the positive
transfer interval on which the prescribed norm crossing occurs.
Fix constants $\alpha>0$ sufficiently small and $\eta_A>0$ sufficiently large, as specified below. On the $n$th module choose
\begin{equation}\label{eq:Gdata}
 I_a(0)=I_c(0)=\kappa_n/2,\qquad
 I_b(0)=\alpha J_n,\quad I_d(0)=(1-\alpha)J_n,
 \qquad J_n=\eta_A\kappa_n W_{c_n}/W_{N_n}.
\end{equation}
All other modes vanish. Choose angles with
$\psi=m_n(\theta_c-\theta_a)-\theta_b+\theta_d=-\pi/2$.
Since $K_n=\kappa_n$, Lemma~\ref{lem:locking} applies. With $x=J_n y$, define
\[
 \rho_n=2m_nJ_n/\kappa_n=2\eta_A m_nW_{c_n}/W_{N_n}.
\]
The full action evolution is
\begin{align}
 I_a&=\kappa_n/2+m_nJ_ny,& I_c&=\kappa_n/2-m_nJ_ny,\label{eq:Gactions}\\
 I_b&=J_n(\alpha+y),&I_d&=J_n(1-\alpha-y),\nonumber
\end{align}
and the scalar equation is exactly
\begin{equation}\label{eq:Gode}
 \dot y=2(\kappa_n/2)^{m_n}(1-\rho_n^2y^2)^{m_n/2}
           \sqrt{(\alpha+y)(1-\alpha-y)}.
\end{equation}
The estimate $m_n\rho_n^2\to0$ follows from \eqref{eq:Gweights} and will justify the time integral.

Put
\begin{align}
 Q_{-,n}&=\tfrac12(1+W_{a_n}/W_{c_n})+
                   \eta_A[(1-\alpha)+\alpha E_n],\label{eq:GQ}\\
 \delta_n&=E_n-1-m_n(W_{c_n}-W_{a_n})/W_{N_n}.\nonumber
\end{align}
Then $Q_{-,n}\to Q_-=1+\eta_A[1+\alpha(E-1)]$ and $\delta_n\to E-1>0$. Exactly, with the initial physical norm $\varepsilon_n$,
\begin{equation}\label{eq:Gnorm}
 \varepsilon_n^2=\kappa_n W_{c_n}Q_{-,n},\qquad
 \nrm{u^{(n)}(t)}_s^2=\kappa_n W_{c_n}[Q_{-,n}+\eta_A\delta_n y(t)].
\end{equation}
The target level is consequently
\begin{equation}\label{eq:Gtarget}
 y_{A,n}=\frac{(A^2-1)Q_{-,n}}{\eta_A\delta_n}
 \longrightarrow y_A=(A^2-1)\left[\frac1{\eta_A(E-1)}+\frac1{E-1}+\alpha\right].
\end{equation}
Since $(A^2-1)/(E-1)<1$, choose $\alpha>0$ small and then $\eta_A$ large so that $0<y_A<1-2\alpha$. For large $n$ the target is uniformly inside $0<y<1-\alpha$ and $\rho_ny_{A,n}<1/2$. All four actions are strictly positive up to the target, the scalar speed is positive, and \eqref{eq:Gnorm} is strictly increasing. The scalar trajectory therefore reaches this interior target at precisely the first time in \eqref{eq:escape}.

\subsubsection{Selector cancellation and reconstruction of the phases}
In the positive-action region,
$\mathcal A=(I_aI_c)^{m/2}(I_bI_d)^{1/2}$ and
$\psi=m(\theta_c-\theta_a)-\theta_b+\theta_d$.
The individual angle equations are
\begin{align*}
 \dot\theta_a&=\omega_a+\tfrac12(I_b-I_d)
                          +m\mathcal A I_a^{-1}\cos\psi,\\
 \dot\theta_c&=\omega_c+\tfrac12(I_b-I_d)
                          +m\mathcal A I_c^{-1}\cos\psi,\\
 \dot\theta_b&=\omega_b+\tfrac12(I_a+I_c)
                             +\mathcal A I_b^{-1}\cos\psi,\\
 \dot\theta_d&=\omega_d-\tfrac12(I_a+I_c)
                             +\mathcal A I_d^{-1}\cos\psi.
\end{align*}
Combining them, and using the prescribed divisor, gives
\[
 \dot\psi=\kappa_n-K_n-\mathcal A\cos\psi
 \left\{m^2(I_a^{-1}-I_c^{-1})+I_b^{-1}-I_d^{-1}\right\}.
\]
Together with $\dot K_n=0$ and the action equations, this verifies
the phase-locking hypotheses directly for this selector.
Choose $\theta_b(0)=\pi/2$ and the other three angles zero.
Then $\psi(0)=-\pi/2$, and the invariant surface
$K_n=\kappa_n$, $\psi=-\pi/2$ gives exactly
\eqref{eq:Gactions}--\eqref{eq:Gode}.

The Cartesian solution is reconstructed by
\[u_j(t)=\sqrt{I_j(t)}\exp(-\ii\theta_j(t)).\]
Its phases are
\begin{align}\label{eq:G-phase-reconstruction}
 \theta_a(t)&=\theta_a(0)+\omega_at
             +\tfrac12J_n(2\alpha-1)t+J_n\int_0^t y(v)\,dv,\nonumber\\
 \theta_c(t)&=\theta_c(0)+\omega_ct
             +\tfrac12J_n(2\alpha-1)t+J_n\int_0^t y(v)\,dv,\\
 \theta_b(t)&=\theta_b(0)+(\omega_b+\kappa_n/2)t,\qquad
 \theta_d(t)=\theta_d(0)+(\omega_d-\kappa_n/2)t.\nonumber
\end{align}
These formulas solve every angle equation on the locked surface.
Every action remains positive until the target specified below;
therefore differentiation of the reconstruction verifies the
Cartesian Hamiltonian equations throughout that interval.

\subsubsection{Uniform positivity and the first weighted crossing}
Put $B_A=A^2-1$, $D_E=E-1$ and
$t_A=1-B_A/D_E>0$. It suffices to choose
\begin{equation}\label{eq:G-fixed-alpha-eta}
 0<\alpha<\min\{1/4,t_A/[4(B_A+2)]\},\qquad
 \eta_A>4B_A/(D_Et_A).
\end{equation}
The limiting target in \eqref{eq:Gtarget} then satisfies
$y_A+2\alpha<1-t_A/2$ and $y_A>0$.
Convergence of $Q_{-,n}$, $\delta_n$ and $y_{A,n}$ gives a fixed
$\delta>0$ such that, for all sufficiently large $n$,
\[
 y_A/2\leq y_{A,n}\leq1-\alpha-\delta,
 \qquad\delta_n\geq D_E/2,
 \qquad\rho_n y_{A,n}\leq1/2.
\]
The last inequality follows from \eqref{eq:G-superpolynomial-gap}.
In fact
$m_n\rho_n^2=4\eta_A^2m_n^3(W_{c_n}/W_{N_n})^2\to0$.
For $0\leq y\leq y_{A,n}$, the four action factors in
\eqref{eq:Gactions} are bounded below by
$\kappa_n/4$, $\kappa_n/4$, $\alpha J_n$ and $\delta J_n$,
respectively. The scalar right side is continuous and strictly
positive on this compact interval. Its reciprocal has a finite
integral, which is the time needed to reach $y_{A,n}$.
Finally the exact affine norm formula \eqref{eq:Gnorm} has positive
slope $\kappa_nW_{c_n}\eta_A\delta_n$. It is smaller than
$A^2\varepsilon_n^2$ at every $0\leq y<y_{A,n}$ and equal to
that value at the endpoint. Hence the computed arrival is the first
prescribed norm amplification.

\subsubsection{Evaluation of the time integral and scale inversion}

\begin{theorem}[Gevrey effective upper bound]\label{thm:gevreyupper}
The fixed Hamiltonian \eqref{eq:Gham} satisfies all the lower-bound hypotheses. Its global finite-support solutions \eqref{eq:Gdata} have $\varepsilon_n\to0$ and
\begin{equation}\label{eq:Gconstant}
 \log T_A^+(u_0^{(n)};H_G)=\left(C_G+O\!\left(\frac{\log\log L_n}{\log L_n}\right)\right)\Phi_{G,p}(L_n),
 \qquad C_G=2\left(\frac{p}{\tau_0 2^{p-1}}\right)^{1/p}.
\end{equation}
In particular, for every fixed $h\in(0,C_G)$ and all sufficiently large $n$,
\begin{equation}\label{eq:Gfixedconstants}
 \exp((C_G-h)\Phi_{G,p}(L_n))\leq T_A^+(u_0^{(n)};H_G)
 \leq\exp((C_G+h)\Phi_{G,p}(L_n)).
\end{equation}
\end{theorem}
\begin{proof}
Integrating \eqref{eq:Gode} up to \eqref{eq:Gtarget} gives the exact expression
\begin{equation}\label{eq:Gintegral}
 T_n=\frac12\left(\frac2{\kappa_n}\right)^{m_n}
 \int_0^{y_{A,n}}\frac{(1-\rho_n^2y^2)^{-m_n/2}}
                    {\sqrt{(\alpha+y)(1-\alpha-y)}}\,dy.
\end{equation}
Define
\[\mathcal I_G=\int_0^{y_A}\frac{dy}{\sqrt{(\alpha+y)(1-\alpha-y)}}.\]
Choose a fixed $Y_*<1-\alpha$ such that
$0<y_A,y_{A,n}\leq Y_*$ for all sufficiently large $n$. On $[0,Y_*]$, $\rho_n^2y^2\leq1/4$ and
$0\leq-\log(1-v)\leq2v$ for $0\leq v\leq1/4$.
It follows that
\[
 1\leq(1-\rho_n^2y^2)^{-m_n/2}
       \leq\exp(Cm_n\rho_n^2),\qquad
 |(1-\rho_n^2y^2)^{-m_n/2}-1|\leq Cm_n\rho_n^2,
\]
for all sufficiently large $n$, uniformly on the interval.
On the same interval,
\[
 \frac1{\sqrt{(\alpha+y)(1-\alpha-y)}}
 \leq\frac1{\sqrt{\alpha(1-\alpha-Y_*)}}.
\]
Writing $\mathcal I_{G,n}$ for the integral in
\eqref{eq:Gintegral}, we obtain
\begin{equation}\label{eq:G-integral-error}
 |\mathcal I_{G,n}-\mathcal I_G|
 \leq C\{m_n\rho_n^2+|y_{A,n}-y_A|\}\longrightarrow0.
\end{equation}
Its limiting positive value is explicitly
\[
 \mathcal I_G=2\arcsin\sqrt{\alpha+y_A}
                         -2\arcsin\sqrt\alpha.
\]
In particular $\log\mathcal I_{G,n}=O(1)$, with constants
independent of the module.

Let $B=\tau_02^{p-1}$. The exact detuning gives
\begin{align}\label{eq:G-logarithms-explicit}
 \log(1/\kappa_n)
 &=\tau_0(2m_n+2)^p\log b_n-\log\xi_n
   =2B m_n^p\log m_n+O(m_n^p),\nonumber\\
 L_n&=\tfrac12\log(1/\kappa_n)-s c_n^\theta
                                      -\tfrac12\log Q_{-,n}\\
 &=B m_n^p\log m_n+O(m_n^p).
 \nonumber
\end{align}
For the first line use
$\log b_n=\log m_n-\log2+O(N_n^{-\theta})$ and
$(2m_n+2)^p=2^pm_n^p(1+O(1/m_n))$.
For the second, $c_n^\theta=O(m_n^{\theta(1-\theta)})$
and $\theta(1-\theta)<1\leq p$, while $Q_{-,n}$ stays
between two positive constants. These formulas imply
$L_n\to\infty$, hence $\varepsilon_n\to0$.
The exact time formula now gives
\begin{align}\label{eq:G-time-errors-full}
 \log T_n
 &=m_n\log(1/\kappa_n)+m_n\log2-\log2
                                   +\log\mathcal I_{G,n}\nonumber\\
 &=2m_nL_n+2s m_nc_n^\theta+m_n\log Q_{-,n}
                                   +m_n\log2+O(1)\\
 &=2m_nL_n+o(m_nL_n).
 \nonumber
\end{align}
The inversion can be quantified from the same identities.
Put $v_n=\log m_n$. Then
\[
 \log L_n=pv_n+\log v_n+\log B+O(v_n^{-1}).
\]
First $v_n\asymp\log L_n$; substituting this bound back gives
$v_n=p^{-1}\log L_n+O(\log\log L_n)$.
Consequently
\[
 m_n^p=\frac{L_n}{Bv_n}(1+O(v_n^{-1}))
      =\frac{pL_n}{B\log L_n}
         \left(1+O\!\left(\frac{\log\log L_n}{\log L_n}\right)\right).
\]
Moreover \eqref{eq:G-time-errors-full} divided by $2m_nL_n$
has error $O(c_n^\theta/L_n+L_n^{-1})=O((\log L_n)^{-1})$,
because $c_n^\theta=O(m_n^{\theta(1-\theta)})$ and
$\theta(1-\theta)<p$.
Taking the $p$th root proves the error bound in
\eqref{eq:Gconstant}. Lemma~\ref{lem:Lambert-both} gives the
corresponding exact continuous inversion.

More explicitly,
\[
 m_n=(p/B)^{1/p}(L_n/\log L_n)^{1/p}(1+o(1)),\qquad
 \frac{\log T_n}{\Phi_{G,p}(L_n)}\longrightarrow2(p/B)^{1/p}=C_G.
\]
For any fixed $0<h<C_G$, convergence of this ratio supplies an
index beyond which it lies in $[C_G-h,C_G+h]$. Exponentiation
proves \eqref{eq:Gfixedconstants}. The same fixed Hamiltonian is
used at every index of this sequence.
\end{proof}

\section{The logarithmic example: construction and amplification time}\label{sec:Uupper}
For the logarithmic weight we use four additional signal modes. We
choose their relation together with the pump relation so that momentum
and quadratic dispersion cancel exactly. The signal amplitudes will
then determine both the first norm crossing and a leading contribution
to the time integral.
Fix $s>0$, $q>1$, $p\geq1$, $A>1$, $\tau_0>0$ and $\mu_*>0$. Use the exact weight convention \eqref{eq:logweight}.

\subsection{The two integer relations and recursive disjointness}
Introduce
\begin{equation}\label{eq:Uexponents}
 r_*=(q-1)/(p+1),\qquad \beta_*=(pq+1)/(p+1)=pr_*+1,
 \qquad r_*+\beta_*=q.
\end{equation}
These exponents determine the pump degree and the detuning scale. Let $H_n=\log N_n$ for rapidly increasing integers $N_n$. Choose odd integers $m_n=\mu_*H_n^{r_*}+O(1)$, $m_n\geq2n+5$, and put
\begin{align}
 S_n&=m_n^2,&P_n&=(m_n-1)/2,\nonumber\\
 a_n&=S_n,&c_n&=S_n+1,&b_n&=S_n+P_n+1,&d_n&=S_n-P_n.\label{eq:Upumps}
\end{align}
These four pump modes satisfy
\begin{equation}\label{eq:Upumprels}
 m_n(a_n-c_n)+b_n-d_n=0,\qquad
 m_n(a_n^2-c_n^2)+b_n^2-d_n^2=0,
\end{equation}
because $b_n-d_n=m_n$ and $b_n+d_n=2S_n+1$.

Choose a constant $\lambda>0$ below and set
\begin{equation}\label{eq:Usignals}
 G_n=\lambda N_n H_n^{-(q-1)}+O(1)\in\NN,
 \quad(h_{1,n},h_{2,n},h_{3,n},h_{4,n})
 =(N_n-2G_n,N_n-G_n,N_n+G_n,N_n+2G_n).
\end{equation}
The signal vector $(-1,2,-2,1)$ has zero mass, momentum and quadratic energy. More explicitly, its zeroth, first and second moments on $(-2,-1,1,2)$ vanish, giving
\begin{equation}\label{eq:Usignalrels}
 -h_1+2h_2-2h_3+h_4=0,\qquad
 -h_1^2+2h_2^2-2h_3^2+h_4^2=0.
\end{equation}
For each $n$ increase $N_n$ until $d_n>h_{4,n-1}$, $b_n<h_{1,n}$ and $N_n\geq2^n$. This is possible since $m_n\to\infty$ whereas $m_n^2/N_n\to0$. The eight absolute indices are distinct, and all modules are disjoint. For
\[
 q_n=(m_n,-m_n,1,-1,-1,2,-2,1)
\]
in this order, all spatial indices are positive, so $\widehat q_n=q_n$.
Its length is $D_n=2m_n+8$ and its pivot is
$\Lambda_n=h_{4,n}$. Choose $\mu>1$. Then
\[
 \sum_{j\ne\Lambda_n}|\widehat q_{n,j}|\br j^{-\mu}
 \leq Cm_n(m_n^2)^{-\mu}+CN_n^{-\mu}\longrightarrow0.
\]
For the free and pivot perturbations, respectively,
\begin{align*}
 \sum_{j\text{ free}}|V_j|
 &\leq |V_0|+C\sum_{j\geq1}\br j^{-\mu}<\infty,\\
 \sum_n|V_{\Lambda_n}|
 &\leq\sum_n\bigl(\kappa_n+Cm_n^{1-2\mu}+CN_n^{-\mu}\bigr)\\
 &\leq C\sum_n\bigl(2^{-n\tau_0(2n+5)^p}
                  +(2n+5)^{1-2\mu}+2^{-n\mu}\bigr)<\infty.
\end{align*}
Here $m_n\geq2n+5$, $N_n\geq2^n$ and $1-2\mu<-1$.
The even extension therefore satisfies $V\in\ell^1(\mathbb Z)$. The quadratic identities allow Lemma~\ref{lem:implant}, with both its transverse and kernel estimates, to provide the full non-resonant frequency sequence and
\begin{equation}\label{eq:Ukappa}
 q_n\cdot\omega=-\kappa_n,
 \qquad\kappa_n=\xi_n\Lambda_n^{-\tau_0(2m_n+8)^p},\quad1\leq\xi_n\leq2.
\end{equation}

For the pumps, $m=2P+1$ and the four positions relative to $S$
are $0,1,P+1,-P$. Thus
\begin{align}\label{eq:U-pump-calculation-full}
 m(a-c)+b-d&=-m+(2P+1)=0,\nonumber\\
 m(a^2-c^2)+b^2-d^2
 &=-m(2S+1)+(2P+1)(2S+1)=0.
\end{align}
For $m\geq3$ they are distinct and satisfy $d<a<c<b$.
The signal coefficients are $q^{\rm sig}=(-1,2,-2,1)$ at
$t=(-2,-1,1,2)$. Their three moments are
\[
 \sum q^{\rm sig}_i=-1+2-2+1=0,
 \quad\sum q^{\rm sig}_it_i=2-2-2+2=0,
 \quad\sum q^{\rm sig}_it_i^2=-4+2-2+4=0.
\]
Expansion of $h_i=N+t_iG$ and $h_i^2=N^2+2Nt_iG+t_i^2G^2$
proves both signal identities in \eqref{eq:Usignalrels}. This
signal monomial has degree six, even though it uses four distinct
coordinates. The total transfer degree is therefore
$(2m+2)+6=2m+8$.

One explicit rounding is
$m(N)=2\lfloor\mu_*(\log N)^{r_*}/2\rfloor+1$ and
$G(N)=\lfloor\lambda N(\log N)^{-(q-1)}\rfloor$.
For large integers $N$, $m(N)\to\infty$,
$m(N)^2/N\to0$ and $G(N)/N\to0$, with $G(N)\geq1$.
Given the preceding module, increase $N$ until
\[
 m(N)^2-(m(N)-1)/2>h_{4,n-1},\quad
 m(N)^2+(m(N)+1)/2<N-2G(N),
\]
as well as $m(N)\geq2n+5$ and $N\geq2^n$.
The left side in the first condition tends to infinity, and the
ratio of the left side to the right side in the second tends to
zero. Thus the recursive choice is possible. It yields the strict
ordering
\[
 h_{4,n-1}<d_n<a_n<c_n<b_n<h_{1,n}<h_{2,n}<h_{3,n}<h_{4,n}.
\]
In particular the pivot $\Lambda_n=h_{4,n}$ is the largest
absolute coordinate and has coefficient one in the compressed
integer relation. Since the pump and signal supports are disjoint,
$R(q_n)=2m_n+8$ without cancellation.

\subsection{Pump and signal weight calculations}
Write $W_j=\exp(2s[\log(|j|+\kappa_*)]^q)$ on these large indices and $W_n=W_{b_n}$. The needed weight calculations are
\begin{align}
 W_{a_n}/W_n, W_{c_n}/W_n, W_{d_n}/W_n&\longrightarrow1,\label{eq:Uweights}\\
 \log W_n&=O((\log H_n)^q)=o(H_n^{\beta_*}),\nonumber\\
 W_{N_n+tG_n}/W_{N_n}&\longrightarrow\exp(t\chi)\quad(t=-2,-1,1,2),
 &\chi&=2sq\lambda.\nonumber
\end{align}
For the pumps, the maximum distance is $O(m_n)$ and
$f'(x)=O((\log m_n)^{q-1}/m_n^2)$ there, so the oscillation is $O((\log m_n)^{q-1}/m_n)\to0$. For signals, $f'(N_n)G_n=q\lambda+o(1)$ while the second-order error is $O(H_n^{-(q-1)})\to0$. Also, for every fixed $M>0$,
\begin{equation}\label{eq:Uhugegap}
 W_n/W_{N_n}=o(m_n^{-M}),\qquad
 \log(W_{N_n}/W_n)=2sH_n^q(1+o(1)).
\end{equation}
Choose $\lambda$ sufficiently large that
\begin{equation}\label{eq:Uratio}
 \mathcal R_\chi:=\frac{2\ee^{-\chi}+\ee^{2\chi}}
                         {\ee^{-2\chi}+2\ee^\chi}>A^2.
\end{equation}
This is possible since $\mathcal R_\chi\sim\ee^\chi/2$. Define
\begin{equation}\label{eq:Uslope}
 c_{\rm sig}=-\ee^{-2\chi}+2\ee^{-\chi}-2\ee^\chi+\ee^{2\chi}
            =4\sinh\chi(\cosh\chi-1)>0.
\end{equation}

Let $f(x)=[\log(x+\kappa_*)]^q$. Its derivatives are
\begin{equation}\label{eq:U-weight-derivatives}
 f'(x)=\frac{q[\log(x+\kappa_*)]^{q-1}}{x+\kappa_*},\qquad
 f''(x)=\frac{q[\log(x+\kappa_*)]^{q-2}
                (q-1-\log(x+\kappa_*))}{(x+\kappa_*)^2}.
\end{equation}
All pump indices lie in $[m_n^2-m_n,m_n^2+m_n]$ for large $n$.
On that interval $|f'|\leq C(\log m_n)^{q-1}/m_n^2$.
Their maximum weight-logarithm difference is at most
$C_s(\log m_n)^{q-1}/m_n$, which tends to zero. This proves
the three pump ratios in \eqref{eq:Uweights}, and
\[
 \log W_n=2sf(b_n)=O((\log m_n)^q)=O((\log H_n)^q).
\]
For the signals $N_n+tG_n$ lies in $[N_n/2,2N_n]$.
Taylor's formula with \eqref{eq:U-weight-derivatives} gives
\begin{align}\label{eq:U-signal-Taylor-full}
 f(N_n+tG_n)-f(N_n)
 &=tG_nf'(N_n)+O(G_n^2H_n^{q-1}/N_n^2)\nonumber\\
 &=tq\lambda+O(H_n^{q-1}/N_n)+O(H_n^{-(q-1)}).
\end{align}
Multiplication by $2s$ and exponentiation give the four limits
$W_{N_n+tG_n}/W_{N_n}\to\ee^{t\chi}$.
Also $f(N_n)=H_n^q+O(H_n^{q-1}/N_n)$.
Since $(\log H_n)^q=o(H_n^q)$,
\[
 \log(W_{N_n}/W_n)=2sH_n^q+o(H_n^q),\quad
 \log\left(m_n^B W_n/W_{N_n}\right)
   =B\log m_n-2sH_n^q+o(H_n^q)\to-\infty,
\]
for each fixed $B>0$. This proves \eqref{eq:Uhugegap} with
the precision required for the pump corrections.

\subsection{The eight coordinate equations and analytic direct sum}
The fixed Hamiltonian is
\begin{align}
 H_U&=\sum_j\omega_jI_j+\sum_n(I_{a_n}+I_{c_n})I_{h_{4,n}}
                     +2\operatorname{Re}\sum_n M_n,\label{eq:Uham}\\
 M_n&=u_{a_n}^{m_n}\bar u_{c_n}^{m_n}u_{b_n}\bar u_{d_n}
                \bar u_{h_{1,n}}u_{h_{2,n}}^2\bar u_{h_{3,n}}^2u_{h_{4,n}}.\nonumber
\end{align}
The identities \eqref{eq:Upumprels} and \eqref{eq:Usignalrels}
give zero momentum, and Proposition~\ref{prop:module-analytic}
gives analyticity. The three masses
\[
 K_n=I_{a_n}+I_{c_n},\qquad B_n=I_{b_n}+I_{d_n},\qquad
 J_{{\rm sig},n}=I_{h_{1,n}}+I_{h_{2,n}}+I_{h_{3,n}}+I_{h_{4,n}}
\]
satisfy
\begin{align*}
 \dot K_n&=-2(m_n-m_n)\operatorname{Im}M_n=0,\\
 \dot B_n&=-2(1-1)\operatorname{Im}M_n=0,\\
 \dot J_{{\rm sig},n}&=-2(-1+2-2+1)\operatorname{Im}M_n=0,\\
 q_n\cdot\nabla P_{\rm sel}
   &=m_nI_{h_{4,n}}-m_nI_{h_{4,n}}+K_n=K_n.
\end{align*}
For the three groups, \eqref{eq:Uweights} gives
\[
 \frac{\max(W_{a_n},W_{c_n})}{\min(W_{a_n},W_{c_n})}\to1,
 \qquad
 \frac{\max(W_{b_n},W_{d_n})}{\min(W_{b_n},W_{d_n})}\to1,
 \qquad
 \frac{\max_{1\leq i\leq4}W_{h_{i,n}}}
      {\min_{1\leq i\leq4}W_{h_{i,n}}}\to\ee^{4\chi}.
\]
The supremum of these ratios over $n$ is finite; choose it as
$C_W^2$ in Proposition~\ref{prop:module-flow}. Then
$\|u(t)\|_s\leq C_W\|u(0)\|_s$ for this Hamiltonian as well.

For one module write
$\mathcal P=u_a^m\bar u_c^m u_b\bar u_d$ and
$\mathcal S=\bar u_{h_1}u_{h_2}^2\bar u_{h_3}^2u_{h_4}$,
so $M=\mathcal P\mathcal S$.
The selector equations are
\begin{equation}\label{eq:U-selector-fields-full}
 (X_{\rm sel})_a=-\ii I_{h_4}u_a,\quad
 (X_{\rm sel})_c=-\ii I_{h_4}u_c,\quad
 (X_{\rm sel})_{h_4}=-\ii(I_a+I_c)u_{h_4};
\end{equation}
every other selector component vanishes.
The transfer components are
\begin{align}\label{eq:U-eight-fields-full}
 (X_{\rm tr})_a&=-\ii m\bar u_a^{m-1}u_c^m\bar u_bu_d\bar{\mathcal S},&
 (X_{\rm tr})_c&=-\ii m u_a^m\bar u_c^{m-1}u_b\bar u_d\mathcal S,\nonumber\\
 (X_{\rm tr})_b&=-\ii\bar u_a^m u_c^m u_d\bar{\mathcal S},&
 (X_{\rm tr})_d&=-\ii u_a^m\bar u_c^m u_b\mathcal S,\nonumber\\
 (X_{\rm tr})_{h_1}&=-\ii\mathcal P u_{h_2}^2\bar u_{h_3}^2u_{h_4},&
 (X_{\rm tr})_{h_2}&=-2\ii\bar{\mathcal P}u_{h_1}\bar u_{h_2}u_{h_3}^2\bar u_{h_4},\nonumber\\
 (X_{\rm tr})_{h_3}&=-2\ii\mathcal P\bar u_{h_1}u_{h_2}^2\bar u_{h_3}u_{h_4},&
 (X_{\rm tr})_{h_4}&=-\ii\bar{\mathcal P}u_{h_1}\bar u_{h_2}^2u_{h_3}^2.
\end{align}
The factors two come from differentiating the two squared signal coordinates.
Lemma~\ref{lem:module-weight} applies separately to each of these
outputs. With $D_n=2m_n+8\geq4n+18$ it gives
\[
 \|D^aX_{{\rm tr},n}(u)\|_{\mathcal L^a}
 \leq C_aD_n^{a+1}R^{D_n-1-a}
 \quad(a\leq D_n-1,\ \|u\|_s\leq R<1).
\]
The explicit convergent series in
\eqref{eq:module-tail-field}--\eqref{eq:module-tail-derivative}
therefore control the infinite transfer and all its derivatives.

For the selector an additional direct estimate is useful. Put
$w_n=\min(W_{a_n},W_{c_n})$, $\widetilde w_n=W_{h_{4,n}}$
and $r_n=\|\Pi_{\mathcal E_n}u\|_s$. Then
$K_n\leq r_n^2/w_n$ and $I_{h_{4,n}}\leq r_n^2/\widetilde w_n$.
Squaring \eqref{eq:U-selector-fields-full} with the output weights
and summing gives
\begin{equation}\label{eq:U-selector-tail-explicit}
 \|X_{{\rm sel},n}(u)\|_s^2
 \leq Cr_n^6(w_n^{-2}+\widetilde w_n^{-2}).
\end{equation}
Since $w_n,\widetilde w_n\to\infty$ and
$\sum_nr_n^6\leq(\sum_nr_n^2)^3$, orthogonality gives
\[
 \sup_{\|u\|_s\leq R}
       \left\|\sum_{n>N_*}X_{{\rm sel},n}(u)\right\|_s^2
 \leq CR^6\sup_{n>N_*}(w_n^{-2}+\widetilde w_n^{-2})
 \longrightarrow0.
\]
To compute the derivative tails, take three independent inputs $v^{(1)},v^{(2)},
v^{(3)}$ in the cubic coordinate expressions, treating conjugate
coordinates as independent on the complexification. For $j=a,c$,
the coefficient after inserting the three weighted input coordinates
is
\[
 \frac{W_j^{1/2}}
 {\widetilde w_n^{1/2}\widetilde w_n^{1/2}W_j^{1/2}}
       =\widetilde w_n^{-1}.
\]
For the $h_4$ output and either $j=a$ or $j=c$ it is
\[
 \frac{\widetilde w_n^{1/2}}
       {W_j^{1/2}W_j^{1/2}\widetilde w_n^{1/2}}
       =W_j^{-1}\leq w_n^{-1}.
\]
Thus each ordered trilinear term is bounded by
$C(w_n^{-1}+\widetilde w_n^{-1})
 \prod_{i=1}^3\|\Pi_{\mathcal E_n}v^{(i)}\|_s$.
Symmetrizing averages these ordered terms and preserves the bound.
To estimate its direct sum, square and use
\[
 \sum_{n>N_*}\prod_{i=1}^3
       \|\Pi_{\mathcal E_n}v^{(i)}\|_s^2
 \leq\prod_{i=1}^3\sum_{n>N_*}
       \|\Pi_{\mathcal E_n}v^{(i)}\|_s^2
 \leq\prod_{i=1}^3\|v^{(i)}\|_s^2.
\]
Differentiating a cubic chooses $3!/(3-a)!$ input positions.
For $0\leq a\leq3$, inserting $3-a$ copies of $u$ and $a$
unit direction vectors consequently proves
\begin{equation}\label{eq:U-selector-derivative-tail}
 \sup_{\|u\|_s\leq R}
 \left\|D^a\sum_{n>N_*}X_{{\rm sel},n}(u)\right\|_{\mathcal L^a}
 \leq C_aR^{3-a}\sup_{n>N_*}
                    (w_n^{-1}+\widetilde w_n^{-1})\longrightarrow0.
\end{equation}
For $a>3$, $D^aX_{{\rm sel},n}=0$. For the scalar Hamiltonian,
\[
 \sup_{\|u\|_s\leq R}\sum_{n>N_*}K_nI_{h_{4,n}}
 \leq\sup_{n>N_*}(w_n\widetilde w_n)^{-1}\sum_nr_n^4
 \leq R^4\sup_{n>N_*}(w_n\widetilde w_n)^{-1}\longrightarrow0.
\]
Thus both Hamiltonian and vector-field series converge with the
regularity asserted in Proposition~\ref{prop:module-analytic}.
In particular, the $a=1$ estimates for the selector and transfer
give a finite derivative bound $L_R$ on every ball $R<1$.
For $u,v$ in that ball, integration along their connecting segment
gives
\[
 X_P(u)-X_P(v)=\int_0^1DX_P(v+t(u-v))[u-v]\,dt,\qquad
 \|X_P(u)-X_P(v)\|_s\leq L_R\|u-v\|_s.
\]
This supplies the local Lipschitz constant in the module-flow
construction directly from the coordinate estimates.

For every coordinate the transfer gives
$\dot I_j=-2q_{n,j}\operatorname{Im}M_n$. The coefficient sums
over $(a,c)$, $(b,d)$ and $(h_1,h_2,h_3,h_4)$ are zero.
Their three action sums are conserved; the linear and selector terms
have zero action derivatives. The within-group weight ratios tend
to $1$, $1$ and at most $\ee^{4\chi}$, respectively.
They are uniformly bounded after including the finite initial
modules. Proposition~\ref{prop:module-flow} proves global small
solutions and exact module invariance for this Hamiltonian.

\subsection{Phase locking and the first amplification time}\label{sec:Udynamics}
We apply Lemma~\ref{lem:locking} to the eight-mode monomial
\eqref{eq:Uham}. The signal transfer determines the weighted norm
slope and contributes the factor $J_n^{-2}$ to the arrival time.
Integrating the phase equations then reconstructs the solution
of the eight coordinate equations.
Choose fixed $\alpha>0$ small and $\eta_A>0$ large. Activate only the $n$th module with
\begin{align}
 I_a(0)=I_c(0)=I_b(0)=I_d(0)&=\kappa_n/2,
 &J_n&=\eta_A\kappa_nW_n/W_{N_n},\label{eq:Udata}\\
 (I_{h_1},I_{h_2},I_{h_3},I_{h_4})(0)&=J_n(1,2\alpha,2,\alpha).\nonumber
\end{align}
The conserved signal mass is $J_{{\rm sig},n}=3(1+\alpha)J_n$. Taking $\theta_{h_4}(0)=\pi/2$ and all other angles zero gives $\psi=-\pi/2$. Lemma~\ref{lem:locking} applies with $K_n=\kappa_n$. For $x=J_n y$ the exact actions are
\begin{align}
 I_a&=\kappa_n/2+m_nJ_ny,& I_c&=\kappa_n/2-m_nJ_ny,\label{eq:Uactions}\\
 I_b&=\kappa_n/2+J_ny,& I_d&=\kappa_n/2-J_ny,\nonumber\\
 I_{h_1}&=J_n(1-y),& I_{h_2}&=2J_n(\alpha+y),\nonumber\\
 I_{h_3}&=2J_n(1-y),& I_{h_4}&=J_n(\alpha+y).\nonumber
\end{align}
Set $\rho_n=2m_nJ_n/\kappa_n$ and $\sigma_n=2J_n/\kappa_n$. By \eqref{eq:Uhugegap}, both tend to zero and $m_n\rho_n^2\to0$.

The full weighted slope and initial normalization are
\begin{align}
 C_n={}&m_n(W_a-W_c)+W_b-W_d-W_{h_1}+2W_{h_2}-2W_{h_3}+W_{h_4},\label{eq:Unormcoeff}\\
 \Gamma_n={}&C_n/W_{N_n}\longrightarrow c_{\rm sig},\nonumber\\
 Q_{p,n}={}&(W_a+W_c+W_b+W_d)/(2W_n)\longrightarrow2,\nonumber\\
 Q_{s,n}={}&(W_{h_1}+2\alpha W_{h_2}+2W_{h_3}+\alpha W_{h_4})/W_{N_n}\nonumber\\
 &\longrightarrow Q_s(\alpha)=\ee^{-2\chi}+2\alpha\ee^{-\chi}+2\ee^\chi+\alpha\ee^{2\chi},\nonumber\\
 Q_{-,n}={}&Q_{p,n}+\eta_A Q_{s,n}\longrightarrow Q_-=2+\eta_A Q_s(\alpha).\nonumber
\end{align}
The pump contribution to $\Gamma_n$ vanishes by $m_nW_n/W_{N_n}\to0$; the strictly positive signal limit determines the sign for large $n$. Thus
\begin{equation}\label{eq:Unorm}
 \varepsilon_n^2=\kappa_nW_nQ_{-,n},\qquad
 \nrm{u^{(n)}(t)}_s^2=\kappa_nW_n[Q_{-,n}+\eta_A\Gamma_n y(t)].
\end{equation}
The target satisfies
\begin{equation}\label{eq:Utarget}
 y_{A,n}=\frac{(A^2-1)Q_{-,n}}{\eta_A\Gamma_n}
 \longrightarrow y_A=\frac{(A^2-1)[2/\eta_A+Q_s(\alpha)]}{c_{\rm sig}}.
\end{equation}
Condition \eqref{eq:Uratio} is exactly $(A^2-1)Q_s(0)<c_{\rm sig}$. Choose $\alpha>0$ small, then $\eta_A$ large, to make $0<y_A<1-2\delta$ for some fixed $\delta>0$. For large $n$ all actions in \eqref{eq:Uactions} remain positive up to $y_{A,n}$, and the norm increases strictly. This identifies the first amplification time.

\subsubsection{Phase equations and the locked scalar transfer}
The transfer amplitude is
\[
 \mathcal A=(I_aI_c)^{m/2}(I_bI_d)^{1/2}
                 I_{h_1}^{1/2}I_{h_2}I_{h_3}I_{h_4}^{1/2},
 \qquad\psi=-\sum_jq_j\theta_j.
\]
The individual angle equations in the positive-action region are
\begin{align}\label{eq:U-eight-angle-equations}
 \dot\theta_a&=\omega_a+I_{h_4}+m\mathcal A I_a^{-1}\cos\psi,&
 \dot\theta_c&=\omega_c+I_{h_4}+m\mathcal A I_c^{-1}\cos\psi,\nonumber\\
 \dot\theta_b&=\omega_b+\mathcal A I_b^{-1}\cos\psi,&
 \dot\theta_d&=\omega_d+\mathcal A I_d^{-1}\cos\psi,\nonumber\\
 \dot\theta_{h_1}&=\omega_{h_1}+\mathcal A I_{h_1}^{-1}\cos\psi,&
 \dot\theta_{h_2}&=\omega_{h_2}+2\mathcal A I_{h_2}^{-1}\cos\psi,\\
 \dot\theta_{h_3}&=\omega_{h_3}+2\mathcal A I_{h_3}^{-1}\cos\psi,&
 \dot\theta_{h_4}&=\omega_{h_4}+K+\mathcal A I_{h_4}^{-1}\cos\psi.\nonumber
\end{align}
Indeed differentiation of $2\mathcal A\cos\psi$ with respect
to $I_j$ gives $|q_j|\mathcal A I_j^{-1}\cos\psi$.
For $P_{\rm sel}=K I_{h_4}$,
$q\cdot\nabla P_{\rm sel}=mI_{h_4}-mI_{h_4}+K=K$ and
$q\cdot\nabla K=m-m=0$. Hence
\begin{align}\label{eq:U-phase-full}
 \dot\psi=\kappa_n-K-\mathcal A\cos\psi\bigl[
 &m^2(I_a^{-1}-I_c^{-1})+I_b^{-1}-I_d^{-1}\nonumber\\
 &-I_{h_1}^{-1}+4I_{h_2}^{-1}-4I_{h_3}^{-1}+I_{h_4}^{-1}\bigr].
\end{align}
The factors four are $q_j|q_j|$ for the repeated signal indices.
Together with $\dot K=0$, this proves invariance of
$K=\kappa_n$, $\psi=-\pi/2$ in the positive-action region.
Along it the eight actions satisfy $\dot I_j=2q_j\mathcal A$,
which integrates to \eqref{eq:Uactions}.

The angles of the exact solution are reconstructed by
\begin{align}\label{eq:U-phase-reconstruction}
 \theta_a(t)&=\theta_a(0)+\omega_at+J_n\alpha t
                                  +J_n\int_0^ty(v)\,dv,\nonumber\\
 \theta_c(t)&=\theta_c(0)+\omega_ct+J_n\alpha t
                                  +J_n\int_0^ty(v)\,dv,\\
 \theta_{h_4}(t)&=\theta_{h_4}(0)+(\omega_{h_4}+\kappa_n)t,\nonumber\\
 \theta_j(t)&=\theta_j(0)+\omega_jt
                     \quad(j=b,d,h_1,h_2,h_3).\nonumber
\end{align}
Indeed all transfer contributions to the angle equations contain
$\cos\psi$ and vanish on this surface, and the only selector
derivatives are those in \eqref{eq:U-selector-fields-full}.

\subsubsection{The exact norm slope and fixed choices of parameters}
The norm derivative with respect to the scalar transfer $x$ is
the coefficient $C_n$ of \eqref{eq:Unormcoeff}. For the pump
terms,
\[
 \frac{|m_n(W_a-W_c)+W_b-W_d|}{W_{N_n}}
 \leq C(m_n+1)\frac{W_n}{W_{N_n}}\longrightarrow0.
\]
The four signal limits give
$\Gamma_n\to c_{\rm sig}$ with the positive value in
\eqref{eq:Uslope}. One may verify that identity by expanding
$\ee^{2\chi}-\ee^{-2\chi}=2\sinh(2\chi)$ and
$2\ee^\chi-2\ee^{-\chi}=4\sinh\chi$.
Thus $c_{\rm sig}=4\sinh\chi(\cosh\chi-1)>0$.

Put $B_A=A^2-1$,
$Q_0=\ee^{-2\chi}+2\ee^\chi$,
$Q_1=2\ee^{-\chi}+\ee^{2\chi}$ and
$t_A=c_{\rm sig}-B_AQ_0=Q_1-A^2Q_0>0$.
The last strict inequality is exactly \eqref{eq:Uratio}.
It is sufficient to choose
\begin{equation}\label{eq:U-fixed-alpha-eta}
 0<\alpha<\min\{1,t_A/(4B_AQ_1)\},\qquad
 \eta_A>8B_A/t_A.
\end{equation}
Since $Q_s(\alpha)=Q_0+\alpha Q_1$, these conditions imply
$0<y_A\leq1-t_A/(2c_{\rm sig})$.
Convergence in \eqref{eq:Utarget} gives a fixed $\delta>0$ with
$y_A/2\leq y_{A,n}\leq1-\delta$ for all sufficiently large $n$.
Furthermore
\[
 m_n\rho_n^2=4\eta_A^2m_n^3(W_n/W_{N_n})^2\longrightarrow0,
 \qquad\sigma_n^2=4\eta_A^2(W_n/W_{N_n})^2\longrightarrow0.
\]
Thus $\rho_n y_{A,n},\sigma_n y_{A,n}\leq1/2$ eventually.
For $0\leq y\leq y_{A,n}$, \eqref{eq:Uactions} consequently gives
\begin{align*}
 \min(I_a,I_c)&\geq\tfrac{\kappa_n}{2}(1-\rho_n y_{A,n})
                         \geq\kappa_n/4,\\
 \min(I_b,I_d)&\geq\tfrac{\kappa_n}{2}(1-\sigma_n y_{A,n})
                         \geq\kappa_n/4,\\
 I_{h_1}&\geq\delta J_n,\qquad I_{h_3}\geq2\delta J_n,\\
 I_{h_2}&\geq2\alpha J_n,\qquad I_{h_4}\geq\alpha J_n.
\end{align*}
It follows that, throughout this closed interval,
\[
 \mathcal A>0,\qquad \dot y=2\mathcal A/J_n>0,\qquad
 \frac{d}{dy}\|u^{(n)}\|_s^2=\kappa_nW_n\eta_A\Gamma_n>0.
\]
Thus the target in \eqref{eq:Utarget} is the first norm crossing.

\subsubsection{First-arrival integral and asymptotic estimates}

\begin{theorem}[Logarithmic ultra-differentiable effective upper bound]\label{thm:logupper}
The fixed Hamiltonian \eqref{eq:Uham} satisfies all the lower-bound hypotheses. The data \eqref{eq:Udata} generate global finite-support solutions with $\varepsilon_n\to0$ and
\begin{align}
 \log T_A^+(u_0^{(n)};H_U)&=\left(C_U+O(L_n^{-r_*/\beta_*})\right)\Phi_{U,p,q}(L_n),\label{eq:Uconstant}\\
 C_\kappa&=\tau_0 2^p\mu_*^p,&
 C_U&=(\mu_*C_\kappa+4s)(2/C_\kappa)^{q/\beta_*}.\nonumber
\end{align}
For every fixed $h\in(0,C_U)$ and all sufficiently large $n$,
\begin{equation}\label{eq:Ufixedconstants}
 \exp((C_U-h)\Phi_{U,p,q}(L_n))\leq T_A^+(u_0^{(n)};H_U)
 \leq\exp((C_U+h)\Phi_{U,p,q}(L_n)).
\end{equation}
\end{theorem}
\begin{proof}
Substitution of all eight actions from \eqref{eq:Uactions} gives
\[
 \mathcal A_n=4(\kappa_n/2)^{m_n+1}J_n^3
 (1-\rho_n^2y^2)^{m_n/2}(1-\sigma_n^2y^2)^{1/2}
 (\alpha+y)^{3/2}(1-y)^{3/2}.
\]
Because $\dot x=2\mathcal A_n$ and $x=J_n y$, the first arrival time is exactly
\begin{align}
 T_n&=\frac18(2/\kappa_n)^{m_n+1}J_n^{-2}\mathcal I_n,\label{eq:Uintegral}\\
 \mathcal I_n&=\int_0^{y_{A,n}}
 \frac{(1-\rho_n^2y^2)^{-m_n/2}(1-\sigma_n^2y^2)^{-1/2}}
      {(\alpha+y)^{3/2}(1-y)^{3/2}}\,dy.\nonumber
\end{align}
Define
\[\mathcal I_U=\int_0^{y_A}\frac{dy}{(\alpha+y)^{3/2}(1-y)^{3/2}}.\]
For $0\leq y\leq1-\delta$,
\[
 (\alpha+y)^{3/2}(1-y)^{3/2}\geq(\alpha\delta)^{3/2}>0.
\]
Set
\[F_n(y)=(1-\rho_n^2y^2)^{-m_n/2}(1-\sigma_n^2y^2)^{-1/2}.\]
Since both squared arguments are at most $1/4$ there,
$0\leq-\log(1-v)\leq2v$ implies
\[
 0\leq\log F_n(y)\leq C(m_n\rho_n^2+\sigma_n^2),\qquad
 |F_n(y)-1|\leq C(m_n\rho_n^2+\sigma_n^2).
\]
The second bound follows from $\ee^z-1\leq2z$ for small
nonnegative $z$, uniformly for large $n$.
Integrate it on the common part of the two target intervals.
On their difference the integrand is uniformly bounded, giving
\begin{equation}\label{eq:U-integral-error}
 |\mathcal I_n-\mathcal I_U|
 \leq C\{m_n\rho_n^2+\sigma_n^2+|y_{A,n}-y_A|\}
 \longrightarrow0.
\end{equation}
Here the positive limiting integral has the primitive
\[
 \mathcal I_U=\frac{2}{(1+\alpha)^2}
 \left[\frac{2y+\alpha-1}{\sqrt{(\alpha+y)(1-y)}}\right]_{y=0}^{y=y_A}.
\]
To verify it, differentiate the numerator divided by the square
root; the resulting numerator is $(1+\alpha)^2/2$ after
putting the terms over the common denominator
$((\alpha+y)(1-y))^{3/2}$.

The rounded degree and pivot satisfy
$m_n=\mu_*H_n^{r_*}+O(1)$ and
$\log\Lambda_n=H_n+O(H_n^{-(q-1)})$.
Using $\beta_*=pr_*+1$ gives the following error estimates:
\begin{align}\label{eq:U-logarithms-explicit}
 \log(1/\kappa_n)
 &=C_\kappa H_n^{\beta_*}+O(H_n^{\beta_*-r_*}+1),\nonumber\\
 L_n&=\frac{C_\kappa}{2}H_n^{\beta_*}
          +O(H_n^{\beta_*-r_*}+(\log H_n)^q+1),\\
 \log(1/J_n)
 &=2sH_n^q+C_\kappa H_n^{\beta_*}
          +O(H_n^{\beta_*-r_*}+(\log H_n)^q+1).
 \nonumber
\end{align}
The first error comes from the relative $O(H_n^{-r_*})$ error
in $(2m_n+8)^p$; the pivot-logarithm correction contributes
$O(H_n^{\beta_*-q})=O(H_n^{-r_*})$ and is smaller.
The second line subtracts
$\frac12\log W_n+\frac12\log Q_{-,n}$.
For the third, use the exact identity
$\log(1/J_n)=\log(1/\kappa_n)+\log W_{N_n}-\log W_n-\log\eta_A$
and the weight calculations. In particular $L_n\to\infty$.

Taking logarithms in \eqref{eq:Uintegral} yields
\begin{align}\label{eq:U-time-errors-full}
 \log T_n
 &=(m_n+1)\log(1/\kappa_n)+2\log(1/J_n)
            +(m_n+1)\log2-\log8+\log\mathcal I_n\nonumber\\
 &=(\mu_*C_\kappa+4s)H_n^q+O(H_n^{\beta_*}).
\end{align}
For this simplification use $r_*+\beta_*=q$,
$\beta_*-r_*=(p-1)r_*+1>0$, and $\beta_*<q$.
The terms $O(H_n^{r_*})$ and $O((\log H_n)^q)$ are absorbed
into $O(H_n^{\beta_*})$.
Since $\beta_*-r_*>0$, the second line of
\eqref{eq:U-logarithms-explicit} is
$L_n=(C_\kappa/2)H_n^{\beta_*}(1+O(H_n^{-r_*}))$.
This first gives $H_n\asymp L_n^{1/\beta_*}$. Raising the
identity to the power $q/\beta_*$ then yields
\[
 H_n^q=(2/C_\kappa)^{q/\beta_*}L_n^{q/\beta_*}
                (1+O(L_n^{-r_*/\beta_*})).
\]
The remainder in \eqref{eq:U-time-errors-full}, divided by $H_n^q$,
is $O(H_n^{\beta_*-q})=O(H_n^{-r_*})$.
Substitution proves the quantitative error in \eqref{eq:Uconstant}.

Since $q/\beta_*=1+(q-1)/(pq+1)$, the resulting ratio is
\[
 \frac{\log T_n}{\Phi_{U,p,q}(L_n)}
 \longrightarrow(\mu_*C_\kappa+4s)(2/C_\kappa)^{q/\beta_*}=C_U.
\]
For each fixed $0<h<C_U$ the ratio eventually lies in
$[C_U-h,C_U+h]$; exponentiation gives \eqref{eq:Ufixedconstants}.
\end{proof}

\section{Nonlocal PDEs and sharp stability scales}\label{sec:upper-comparison}
\subsection{Realization as fixed nonlocal equations}\label{sec:pdebridge}
Fourier synthesis turns the two constructed Hamiltonians into
autonomous equations on the circle. We identify their nonlinearities
and the corresponding solution flows.

\begin{proposition}[Fixed analytic nonlocal equations]\label{prop:pdebridge}
For $X=G$ or $U$, let $H_X=\sum_j(j^2+V_j)|u_j|^2+P_X$ be the corresponding constructed Hamiltonian. Its parameters can be chosen so that $V\in\ell^1(\ZZ)$ is real and even. Then $V(x)=\sum_jV_j\ee^{\ii jx}$ is a continuous real convolution potential, and the Hamiltonian equation is the autonomous PDE
\begin{equation}\label{eq:upperPDE}
 \ii\partial_t u=-\partial_x^2u+V*u+\mathcal N_X(u),\qquad x\in\mathbb T,
\end{equation}
where, with the normalized circle integral in the convolution,
\begin{equation}\label{eq:nonlocalnonlinearity}
 \mathcal N_X(u)(x)=\sum_{j\in\ZZ}
 \partial_{\bar u_j}P_X((u_k)_{k\in\ZZ},(\bar u_k)_{k\in\ZZ})\ee^{\ii jx}.
\end{equation}
The map $\mathcal N_X$ is real analytic from a fixed ball of $W_s$ to $W_s$, vanishes to order three, and commutes with spatial translations. The finite-support trajectories of Theorems~\ref{thm:gevreyupper} and~\ref{thm:logupper} are global solutions of \eqref{eq:upperPDE} with precisely the first amplification times stated there.
\end{proposition}
\begin{proof}
Sections~\ref{sec:Gupper} and~\ref{sec:Uupper} give
$\sum_j|V_j|<\infty$. Thus
\begin{align*}
 \sup_x\left|\sum_{|j|>N}V_j\ee^{\ii jx}\right|
    &\leq\sum_{|j|>N}|V_j|\longrightarrow0,\\
 \|V*u\|_s^2
    &=\sum_jw_s(j)^2|V_j|^2|u_j|^2
      \leq\|V\|_{\ell^\infty}^2\|u\|_s^2.
\end{align*}
Moreover $V_{-j}=V_j\in\mathbb R$ gives
$\overline{V(x)}=\sum_jV_j\ee^{-\ii jx}=V(x)$.

For the nonlinear Hamiltonians, $\ord P_X\geq4$ and
$\sup_{m\geq4}c_m(P_X)\leq C_X$. Therefore
\[
 \coef{P_X}_R\leq C_X\sum_{m\geq4}m^3(bR)^{m-2}<\infty
                  \qquad(R<1/b).
\]
Lemmas~\ref{lem:cut} and~\ref{lem:scalar-evaluation} then give
the $W_s$-valued holomorphic vector field defining
\eqref{eq:nonlocalnonlinearity}.
Proposition~\ref{prop:module-analytic} also gives real analyticity
on the physical unit ball, including the quartic selector series. In particular, for sufficiently small $r$,
\begin{equation}\label{eq:nonlocalflowbounds}
 \sup_{\nrm u_s\leq r}\nrm{\mathcal N_X(u)}_s\leq C_Xr^3,
 \qquad
 \nrm{\mathcal N_X(u)-\mathcal N_X(v)}_s\leq C_Xr^2\nrm{u-v}_s
 \quad(\nrm u_s,\nrm v_s\leq r).
\end{equation}
Indeed, for $2r<R$, the order-four estimate and the Cauchy
inequality give
\begin{align*}
 \coef{P_X}_{2r}&\leq(2r/R)^2\coef{P_X}_R\leq C_Xr^2,\\
 \sup_{B_s(2r)}\|X_{P_X}\|_s
      &\leq2Cr\coef{P_X}_{2r}\leq C_Xr^3,\\
 \sup_{B_s(r)}\|DX_{P_X}\|
      &\leq r^{-1}\sup_{B_s(2r)}\|X_{P_X}\|_s\leq C_Xr^2.
\end{align*}
Integration of the last derivative bound along $v+t(u-v)$ proves
the second inequality in \eqref{eq:nonlocalflowbounds}. The
Duhamel construction of Subsection~\ref{subsec:flow} therefore
applies on a fixed ball.

Let $(\tau_a u)(x)=u(x+a)$, so $(\tau_a u)_j=\ee^{\ii ja}u_j$.
For every monomial of $P_X$, zero momentum gives
\[
 (\tau_a u)^{\mathcal J}
       =\ee^{\ii a\mathcal M(\mathcal J)}u^{\mathcal J}=u^{\mathcal J},
 \qquad P_X(\tau_a u)=P_X(u).
\]
Differentiating this identity yields
\[
 (\partial_{\bar u_j}P_X)(\tau_a u)
   =\ee^{\ii ja}(\partial_{\bar u_j}P_X)(u),\qquad
 \mathcal N_X(\tau_a u)=\tau_a\mathcal N_X(u).
\]
Thus the nonlocal nonlinearity in \eqref{eq:nonlocalnonlinearity}
commutes with translations.

Finally, the sequence of modules and the frequency parameters are fixed once. Each inactive module has zero vector field when all its coordinates vanish. By uniqueness, activating one module gives exactly its finite-dimensional canonical equations, while every other coordinate stays zero. The conserved group masses bound that module for all times, so these polynomial trajectories are global; finite Fourier support also makes them smooth functions of $x$. The action computations and the physical norm \eqref{eq:physicalnorm} therefore identify these times with the first norm amplification times of the PDE.
\end{proof}

For the two constructed equations, the estimates entering
Theorem~\ref{thm:joint} are collected below.
\begin{center}\small
\begin{tabular}{>{\raggedright\arraybackslash}p{.28\textwidth}>{\raggedright\arraybackslash}p{.62\textwidth}}
\toprule
Property & Construction and estimates\\\midrule
Phase space and norm & The Fourier series in \eqref{eq:upperPDE} uses $j\in\ZZ$ and exactly \eqref{eq:physicalnorm}.\\[3pt]
Analytic perturbation & The selectors are quartic, transfer coefficients are bounded, and \eqref{eq:nonlocalflowbounds} holds on a fixed ball.\\[3pt]
Even frequencies & $\omega_j=j^2+V_{|j|}$, distinct absolute frequencies, and Lemma~\ref{lem:1dgap} give $\beta=2$, $\delta=1$.\\[3pt]
Uniform divisor bounds & Lemma~\ref{lem:implant} gives the bounds for transverse relations and all integer combinations of the implanted relations.\\[3pt]
First norm amplification & $A$ and the Hamiltonian are fixed before the sequence of initial data; the target is the first crossing of $A\varepsilon_n$.\\\bottomrule
\end{tabular}
\end{center}
The lower theorem therefore applies to each fixed equation in
\eqref{eq:upperPDE}.

\subsection{Comparison of the lower and upper bounds}\label{sec:sharpness}
The two time computations and Proposition~\ref{prop:pdebridge} prove the upper half of Theorem~\ref{thm:joint}. Together with the lower bound, they establish the following sharpness statement.

\begin{corollary}[Sharpness for effective stability in the stated class]\label{cor:optimal}
Fix one of the weights, $p\geq1$ and $A>1$. Every Hamiltonian satisfying the hypotheses has a lower effective stability time $\exp(c_H\Phi(L))$ for sufficiently small initial data. A fixed Hamiltonian in the class also admits a sequence $L_n\to\infty$ with first amplification time at most $\exp(C_H\Phi(L_n))$. Consequently, if $\Psi(L)/\Phi(L)\to\infty$, a bound $T_A^+\geq\exp(c_H\Psi(L))$ with $c_H>0$ cannot hold for every Hamiltonian in this class and all sufficiently small data. Here $c_H$ may depend on the individual Hamiltonian.
\end{corollary}
\begin{proof}
Fix the constructed Hamiltonian $H$. Since
$\Psi(L_n)/\Phi(L_n)\to\infty$, for every $c_H>0$ and all
sufficiently large $n$,
\[
 \log T_A^+(u_0^{(n)};H)
       \leq C_H\Phi(L_n)<c_H\Psi(L_n).
\]
This contradicts $T_A^+(u_0;H)\geq\exp(c_H\Psi(L))$ for all
sufficiently small data. The lower theorem applies to this same
$H$, with the norm \eqref{eq:physicalnorm} and the criterion
\eqref{eq:escape}.
\end{proof}

The asymptotic relations proved above have the common form
\[
 \lim_{n\to\infty}\frac{\log T_n}{\Phi(L_n)}=C_{\rm sh}>0.
\]
For each fixed tolerance $0<h<C_{\rm sh}$ this is precisely the pair of
bounds $\exp((C_{\rm sh}-h)\Phi(L_n))\leq T_n\leq
\exp((C_{\rm sh}+h)\Phi(L_n))$ for all sufficiently large $n$.
The constant multiples occur in the logarithm of the effective
stability time. These bounds use the fixed-amplification criterion
of \eqref{eq:escape}, as in the discussion of \cite{FF22}.

The two constructions display the mechanisms producing the two scales.
The Gevrey example uses a high-power pump and a four-mode transfer.
In the logarithmic example the signal contribution $J_n^{-2}$ is of
the same logarithmic order as the pump term; both enter its leading constant.

\appendix
\section{Constants and order bookkeeping}\label{app:bookkeeping}
\subsection{A fixed order of choosing constants}
The choices can be made before $d,N,r$ are varied. Start with
the weight parameters $c_0,C_f,s_0,s$, hence $S_0,b,\nu$,
and the fixed analytic data $R,M$.
Let $C_L\geq1$ dominate the absolute evaluation, Poisson, Cauchy,
flow and transport constants in Lemmas~\ref{lem:cut}--\ref{lem:integral-lie}.
Include in $C_L$ the two fixed bounds in
\eqref{eq:complex-Picard-contraction}; the complex-time
construction then uses the same iteration smallness condition.
Set $C_P=1+4M/R$.
Independently choose the buffer constants from the frequency data
so that \eqref{eq:buffer} and its two following inequalities hold,
and then choose $C_{\rm div}$ from Lemma~\ref{lem:div}.
One admissible overestimate for the iteration constants is
\begin{equation}\label{eq:explicit-iteration-constants}
 C_*=2^{12}C_L^2(1+C_P)^2,\qquad
 K=2^{20}(1+R^{-1})(1+C_L+C_P)^6.
\end{equation}
Increasing $C_L$ once to cover the elementary fixed factors in the
proof gives $\mu_a=C_*/K\leq1/32$,
$\eta_k\leq1/16$ and a displacement at most $\ell_a/8$ at
every stage. Indeed \eqref{eq:eta-step-full} gives
$\eta_k\leq4C_LC_Pd\Delta a\mu_a^{k-3}$, while the
one-step polynomial multiplier is at most
$32C_LC_Pd\Delta a\leq C_*d^2\Delta a$.
Substitution of $\Delta a=(Kd^2)^{-1}$ into these two bounds
proves the required smallness uniformly for $d\geq4$.

The solved recurrences \eqref{eq:Trec-solved},
\eqref{eq:Ecoef-geometric} and \eqref{eq:Erec-solved} give
\[
 \max_{3\leq k\leq d+1}
       \{T_k,e_k^\#,\ee^{\nu f(N)}e_k\}
 \leq\frac{4Ca}{3}\exp\!\left(
       \frac{C\,4C_LC_P}{Kd(1-\mu_a)}\right)
       \left(C_P+\frac{CC_P\mu_a}{1-\mu_a}\right).
\]
Since $d\geq4$, replacing $d$ by $4$ in the exponential makes the
bound independent of the degree. For the displacement,
\[
 \sup_{B_s(3a)}\|\mathcal T_d-\Id\|_s
 \leq\sum_{k=3}^d4C_LC_P\Delta a^2\mu_a^{k-3}
 \leq\frac{4C_LC_P}{1-\mu_a}\Delta a^2.
\]
The inverse estimate follows on $B_s(2a)$ from the construction in
Section~\ref{sec:normal}. The value of $K$ also ensures $4a<R/4$ and the
inverse-map derivative bound.

For the parameter optimization choose $b_0>0$ first and then
choose $a_0>0$ so that
$C a_0^p<1/16$ in the Gevrey case or
$Cb_0^{1/q}a_0^{p+1/q}<1/16$ in the logarithmic case.
The remaining terms in \eqref{eq:smallbalance} are $o(L)$;
there is therefore one $L_0$ beyond which they sum to at most
$L/8$ and all integer cutoff and degree thresholds hold.
Finally the chosen amplification $A$ fixes $\delta_A$ and the
further initial-data threshold in the bootstrap.
This order of choices separates fixed constants, normalization
parameters and initial-data size.

\subsection{Terminal estimates and inversion formulas}
The iteration and restriction estimates are
\begin{align*}
 \Delta&=\gamma^{-d}(CdN)^{Cd^p},&
 a^{-1}&=Kd^2\Delta,&\mu_a&=C_*d^2a\Delta,\\
 \coef{P_k}_{\rho_k}&\leq C a\mu_a^{k-3},&
 \coef{G_k}_{\rho_k}&\leq C\Delta a\mu_a^{k-3},&
 \sum\eta_k&\leq Cd\Delta a,\\
 \coef{R_T}_{3a}&\leq Ca,&\ord R_T&\geq d+1,&
 \vf{E_d}_{3a,s}&\leq Ca\ee^{-\nu f(N)},\\
 \coef{P_k}_{\widehat\rho_k}&\leq Cr(Cd^2\Delta r)^{k-3},&&&
 \vf{R_T}_{2r,s}&\leq Cr(Cd^2\Delta r)^{d-2}.
\end{align*}
The product in the transported high remainder satisfies
\[
 \prod_{k=3}^{d}\frac{\rho_k}{\rho_{k+1}}(1+C\eta_k)
 \leq\frac43\exp\!\left(C\sum_{k=3}^d\eta_k\right)
 \leq\frac43\exp\!\left(\frac{C}{Kd}\right)\leq C.
\] Equation~\eqref{eq:physicalPk} displays the remaining degree factors explicitly.

For the Gevrey inversion, $L=B m^p\log m$ is solved exactly by
\[
 m=\exp\left(\frac1p W\left(\frac{pL}{B}\right)\right)
   =\left(\frac{pL/B}{W(pL/B)}\right)^{1/p},
\]
where $W$ is the positive real solution of $W(x)\exp(W(x))=x$. Taking logarithms shows $W(x)=\log x-\log\log x+o(1)$, and hence $m=(p/B)^{1/p}(L/\log L)^{1/p}(1+o(1))$. Lemma~\ref{lem:Lambert-both} proves the same leading formula
under a relative $o(1)$ perturbation of $L$. This verifies both the denominator $(\log L)^{1/p}$ and the constant in \eqref{eq:Gconstant}.

For the logarithmic construction, the relation $L=(C_\kappa/2)H^{\beta_*}(1+o(1))$ gives
\[H^q=(2/C_\kappa)^{q/\beta_*}L^{q/\beta_*}(1+o(1)).\] Both contributions $m\log(1/\kappa)$ and $2\log(1/J)$ are of order $H^q$. Their respective leading coefficients are $\mu_*C_\kappa$ and $4s$,
whose sum determines $C_U$.

\section{Technical lemmas}\label{app:technical}
\subsection{Strict subadditivity and summability of the weights}
\begin{lemma}[The two admissible weights]\label{lem:weights-full}
For $0<\theta<1$, $f(x)=x^\theta$ satisfies
\eqref{eq:strictsub} with $C_f=2^\theta-1$.
For $q>1$ and $\kappa_*>\ee^q$,
$f(x)=[\log(x+\kappa_*)]^q$ satisfies it with
$C_f=q/\log\kappa_*$. Both satisfy
\eqref{eq:sumweight} for every $s_0>0$.
\end{lemma}
\begin{proof}
For the power weight take $x\geq y>0$ and $t=y/x\leq1$.
The quotient $g(t)=((1+t)^\theta-1)/t^\theta$ has derivative
\[
 g'(t)=\theta t^{-\theta-1}\{1-(1+t)^{\theta-1}\}\geq0.
\]
Hence $g(t)\leq g(1)=2^\theta-1$ and
$(x+y)^\theta\leq x^\theta+(2^\theta-1)y^\theta$.
For the logarithmic weight, the exact derivative formulas are
\eqref{eq:U-weight-derivatives}; $f''<0$ because
$\log\kappa_*>q$. Therefore, for $x\geq y\geq0$,
\[
 f(x+y)-f(x)=\int_0^yf'(x+t)\,dt
 \leq yf'(y)
 =f(y)\frac{qy}{(y+\kappa_*)\log(y+\kappa_*)}
 \leq\frac q{\log\kappa_*}f(y).
\]
Relabel the inputs so that $x_1=x_*=\max_i x_i$. Since
$\sum_{i=1}^{r-1}x_i\geq x_r$, repeated use of the two-variable
inequality gives
\[
 f\!\left(\sum_{i=1}^r x_i\right)
 \leq f\!\left(\sum_{i=1}^{r-1}x_i\right)+C_ff(x_r)
 \leq f(x_1)+C_f\sum_{i=2}^rf(x_i).
\]
In both cases $0<C_f<1$.

For every $a>0$, both $ax^\theta/\log x$ and
$a[\log(x+\kappa_*)]^q/\log x$ tend to infinity.
Thus for all sufficiently large integers $j$,
$\exp(-af(\br j))\leq |j|^{-2}$.
For a sufficiently large integer $J_0$,
\[
 \sum_{J\in\mathcal Z}\ee^{-af(\br j)}
 \leq2\sum_{|j|\leq J_0}\ee^{-af(\br j)}
       +4\sum_{j>J_0}j^{-2}<\infty.
\]
Taking $a=2(1-C_f)s_0$ proves \eqref{eq:sumweight}.
\end{proof}

\subsection{Scaling of holomorphic maps with a prescribed zero}
\begin{lemma}[Radial scaling]\label{lem:radial-full}
Let $F:B_X(R)\to Y$ be a holomorphic map between complex Banach
spaces, with $\sup\|F\|_Y\leq M$ and Taylor coefficients of
degrees smaller than $h$ equal to zero. Then
\[
 \sup_{B_X(r)}\|F\|_Y\leq M(r/R)^h\qquad(0<r<R).
\]
Composition of a Hamiltonian of order at least $h$ with a map
$\Phi(u)=u+O(u^2)$ preserves that order.
\end{lemma}
\begin{proof}
Fix a unit vector $v$ and a continuous linear functional $\lambda$
on $Y$ of norm one. The function
$g(z)=\lambda(F(Rzv))$ satisfies $|g(z)|\leq M$ for $|z|<1$
and $g^{(a)}(0)=0$ for $a<h$. Thus $g(z)/z^h$ extends
holomorphically across zero. For $|z|<t<1$, the maximum modulus
inequality gives
\[
 \left|\frac{g(z)}{z^h}\right|
 \leq\max_{|\zeta|=t}\left|\frac{g(\zeta)}{\zeta^h}\right|
 \leq\frac M{t^h}.
\]
Letting $t\uparrow1$ and using the dual norm, we obtain
\[
 \|F(Rzv)\|_Y
   =\sup_{\|\lambda\|_{Y^*}\leq1}|\lambda(F(Rzv))|
   \leq M|z|^h,
 \qquad
 \sup_{B_X(r)}\|F\|_Y\leq M(r/R)^h.
\]
For the composition assertion, $\Phi(tu)=tu+O(t^2)$ gives
\[
 F_m(\Phi(tu))=t^mF_m(u)+O(t^{m+1})\quad(m\geq h),\qquad
 \Pi_{<h}(F\circ\Phi)=0.
\]
The substitution is legitimate on every smaller ball by the local
absolute convergence of the homogeneous series.
\end{proof}

\subsection{Determinants and a derivative of a scalar divisor}
\begin{lemma}[A determinant detects one large pairing]\label{lem:detector-full}
Let $v_1,\ldots,v_k\in\mathbb R^k$ be independent row vectors
with $\|v_i\|_1\leq1$, and let $A$ be their row matrix.
For every $w\in\mathbb R^k$,
\[
 \max_i|v_i\cdot w|\geq
       \frac{|\det A|}{k^{3/2}}\|w\|_1.
\]
\end{lemma}
\begin{proof}
For $k=1$ the asserted inequality is an equality. For $k\geq2$, let $c_i$
be the $i$th column of $\operatorname{adj}A$, and let
$A_{\widehat i}$ denote the matrix with row $i$ deleted.
Cauchy--Binet and the Gram determinant inequality give
\[
 \|c_i\|_2^2
 =\sum_{j=1}^k\det(A_{\widehat i,\widehat j})^2
 =\det(A_{\widehat i}A_{\widehat i}^{\mathsf T})
 \leq\prod_{j\ne i}\|v_j\|_2^2
 \leq\prod_{j\ne i}\|v_j\|_1^2\leq1.
\]
For $y=Aw$,
\[
 \|w\|_1\leq\sqrt k\|w\|_2
 =\frac{\sqrt k}{|\det A|}\Big\|\sum_i y_ic_i\Big\|_2
 \leq\frac{k^{3/2}}{|\det A|}\|y\|_\infty,
\]
which is the required bound.
\end{proof}

\subsection{Zero counts and the scalar sublevel inequality}
\begin{lemma}[Wronskians and zeros]\label{lem:ECT-full}
Let $f_0,\ldots,f_k$ be real analytic functions on an interval.
Suppose every initial Wronskian $W(f_0,\ldots,f_j)$, $0\leq j\leq k$,
is nonzero throughout the interval. A nontrivial linear combination
of the first $j+1$ functions has at most $j$ zeros, counted with
multiplicity.
\end{lemma}
\begin{proof}
We induct on the number of functions. A nonzero multiple of $f_0$
has no zeros. Since $f_0$ is nonvanishing, division by it preserves
zero multiplicities. Set $g_i=(f_i/f_0)'$, $1\leq i\leq k$.
Elementary column operations and the product rule give
\begin{equation}\label{eq:Wronskian-quotient-full}
 W(g_1,\ldots,g_j)=\frac{W(f_0,\ldots,f_j)}{f_0^{j+1}}
 \qquad(1\leq j\leq k).
\end{equation}
Indeed the product rule gives
\[
 f_i^{(r)}
 =\sum_{\ell=0}^r\binom r\ell f_0^{(r-\ell)}(f_i/f_0)^{(\ell)}
                 \qquad(0\leq i,r\leq j).
\]
This is a lower-triangular row transformation with every diagonal
entry equal to $f_0$. Expanding the resulting Wronskian along its
constant column therefore gives
\[
 W(f_0,\ldots,f_j)
 =f_0^{j+1}W(1,f_1/f_0,\ldots,f_j/f_0)
 =f_0^{j+1}W(g_1,\ldots,g_j).
\]
In particular $W(g_1,\ldots,g_j)\ne0$ throughout the interval.

If $F=\sum_{i=0}^jc_if_i$ had at least $j+1$ zeros, the
derivative of $F/f_0$ would have at least $j$ zeros by Rolle's
theorem, including the inherited multiplicities at repeated zeros.
This derivative is $\sum_{i=1}^jc_ig_i$. If all these
coefficients vanish, $F/f_0$ is a nonzero constant and has no
zeros. Otherwise the induction hypothesis bounds the zeros of the
derivative by $j-1$, a contradiction. This proves the result.
\end{proof}

\begin{lemma}[A divided-difference sublevel bound]\label{lem:sublevel-full}
Let $g\in C^r(I)$, $r\geq1$, on a bounded interval, and suppose
$|g^{(r)}|\geq a>0$. Then for every $h>0$,
\begin{equation}\label{eq:sublevel-constant-full}
 |\{x\in I:|g(x)|<h\}|\leq2(r+1)(h/a)^{1/r}.
\end{equation}
\end{lemma}
\begin{proof}
Write $E=\{|g|<h\}$ and $S=|E|$. If $S=0$ there is nothing
to prove. Put $F_E(x)=|E\cap(-\infty,x]|$. For $x<y$,
\[
 0\leq F_E(y)-F_E(x)=|E\cap(x,y]|\leq y-x.
\]
Choose $x_i\in\overline E$ with
$F_E(x_i)=(i+1/2)S/(r+1)$, $0\leq i\leq r$. Then
\[
 x_j-x_i\geq F_E(x_j)-F_E(x_i)=\frac{(j-i)S}{r+1},
 \qquad |g(x_i)|\leq h.
\]
Let $\delta=S/(r+1)$. The divided difference satisfies
\begin{align*}
 |g[x_0,\ldots,x_r]|
 &=\left|\sum_{i=0}^r\frac{g(x_i)}{\prod_{j\ne i}(x_i-x_j)}\right|\\
 &\leq\frac h{\delta^r}\sum_{i=0}^r\frac1{i!(r-i)!}
 =\frac{2^rh}{r!\delta^r}.
\end{align*}
For the interpolation polynomial
\[
 P(x)=\sum_{i=0}^rg(x_i)\prod_{j\ne i}\frac{x-x_j}{x_i-x_j},
 \qquad P^{(r)}=r!g[x_0,\ldots,x_r],
\]
we have $(g-P)(x_i)=0$. Applying Rolle's theorem $r$ times gives
an interior $\xi$ such that
\[
 0=(g-P)^{(r)}(\xi),\qquad
 a\leq|g^{(r)}(\xi)|=r!|g[x_0,\ldots,x_r]|
   \leq\frac{2^rh}{\delta^r},\qquad
 S\leq2(r+1)(h/a)^{1/r}.
\]
For open endpoints the points may first be chosen inside a compact
subinterval with bad-set measure tending to $S$; passing to the
limit gives the same inequality.
\end{proof}

\subsection{The positive and negative Lambert branches}
\begin{lemma}[Explicit logarithmic inversion]\label{lem:Lambert-both}
Let $X\to\infty$. The positive solution $W(X)$ of
$w\ee^w=X$ satisfies
$W(X)=\log X-\log\log X+O(\log\log X/\log X)$.
If $0<z\downarrow0$ and $Y=-W_{-1}(-z)>1$, then, with
$Q=\log(1/z)$,
\[
 Q<Y<2Q,\qquad
 Y=Q+\log Q+O(\log Q/Q).
\]
If $L=B m^p\log m(1+o(1))$, $B,p>0$, then
$m=(p/B)^{1/p}(L/\log L)^{1/p}(1+o(1))$.
\end{lemma}
\begin{proof}
For the positive branch, $w+\log w=\log X=:Q$.
The left side is strictly increasing for $w>0$. For large $Q$,
evaluation at $Q-\log Q$ gives
$Q+\log(1-\log Q/Q)=Q+O(\log Q/Q)$.
Its derivative is $1+1/w\geq1$, so the root differs from
$Q-\log Q$ by $O(\log Q/Q)$.
For the negative branch, $Y\ee^{-Y}=z$ means
$Y-\log Y=Q$. The function is strictly increasing for $Y>1$.
Its values at $Q$ and $2Q$ bracket $Q$ for large $Q$, proving
the two elementary bounds. At $Q+\log Q$ its value is
$Q-\log(1+\log Q/Q)=Q+O(\log Q/Q)$.
The derivative $1-1/Y$ is at least $1/2$ there, so the stated
error follows from the mean value theorem.

For the last assertion first solve the exact equation
$L=Bm^p\log m$. Setting $x=p\log m$ gives
$x\ee^x=pL/B$, and hence
\[
 m=\exp\{W(pL/B)/p\}
   =\left(\frac{pL/B}{W(pL/B)}\right)^{1/p}.
\]
The first asymptotic gives the asserted expression. With a relative
$o(1)$ error in $L$, for any fixed $\epsilon>0$ and all large
$m$, the exact values at $(1-\epsilon)m$ and $(1+\epsilon)m$
bracket the perturbed relation, since
$[(1\pm\epsilon)m]^p\log((1\pm\epsilon)m)/(m^p\log m)
\to(1\pm\epsilon)^p$. Monotonicity and then
$\epsilon\downarrow0$ prove the same inversion.
\end{proof}

\section{Analytic Hamiltonian calculus}\label{app:calculus}
We prove the estimates stated in Section~\ref{sec:setting} for the
coefficient norm \eqref{eq:coefnorm}. The calculation proceeds from
weighted Fourier convolution to homogeneous-degree summation, and then
to Poisson brackets and canonical flows.

\subsection{Weighted evaluation and the high-mode gain}\label{app:cut-calculus}
\begin{proof}[Proof of Lemma~\ref{lem:cut}]
Write $z_J=w_s(j)|u_J|$ and
$a_s(J)=\exp(-(1-C_f)sf(\br j))$. Differentiating a symmetric
$m$-homogeneous polynomial gives
\[
 (X_{F_m})_{(j,\sigma)}
 =-\sigma\ii m\!
 \sum_{\mathcal M(J_1,\ldots,J_{m-1})=\sigma j}
 F_{(J_1,\ldots,J_{m-1},(j,-\sigma))}
             u_{J_1}\cdots u_{J_{m-1}}.
\]
Choose $i_*$ with $\br{j_{i_*}}=\max_i\br{j_i}$. Momentum and
\eqref{eq:strictsub} imply
\[
 \br j\leq\sum_{i=1}^{m-1}\br{j_i},\qquad
 sf(\br j)-s\sum_{i=1}^{m-1}f(\br{j_i})
       \leq-(1-C_f)s\sum_{i\ne i_*}f(\br{j_i}).
\]
Consequently
\begin{align}\label{eq:weight-ratio-full}
 \frac{w_s(j)}{\prod_{i=1}^{m-1}w_s(j_i)}
 &\leq\exp\!\left(-(1-C_f)s\sum_{i\ne i_*}f(\br{j_i})\right)
  =\prod_{i\ne i_*}a_s(J_i).
\end{align}
We may sum over all choices of a maximal position: for every tuple,
$1\leq\sum_{i=1}^{m-1}\mathbf1_{\{\br{j_i}=\max_\ell\br{j_\ell}\}}$.
For $v\geq0$ on $\mathcal Z$, define
$\widehat v(k)=v_{(k,+)}+v_{(-k,-)}$. Then
\[
 \|\widehat v\|_1=\|v\|_1,\qquad
 \|\widehat v\|_2^2
 \leq2\sum_k\bigl(v_{(k,+)}^2+v_{(-k,-)}^2\bigr)
 =2\|v\|_2^2,
\]
and the momentum sum has the convolution representation
\[
 \sum_{\mathcal M(J_1,\ldots,J_{m-1})=k}
             \prod_{i=1}^{m-1}v_i(J_i)
     =(\widehat v_1*\cdots*\widehat v_{m-1})(k).
\]
Thus \eqref{eq:weight-ratio-full} and Young's inequality give
\begin{align*}
 w_s(j)|(\underline X_{F_m}(u))_{(j,\sigma)}|
 &\leq m(m-1)c_m(F)
       (\widehat z*(\widehat{a_sz})^{*(m-2)})(\sigma j),\\
 \|\widehat z*(\widehat{a_sz})^{*(m-2)}\|_2
 &\leq\|\widehat z\|_2\|\widehat{a_sz}\|_1^{m-2}
 \leq\sqrt2\|z\|_2\|a_sz\|_1^{m-2},\\
 \|z\|_2&=\sqrt2\|u\|_s,\qquad
 \|a_sz\|_1
 \leq\left(\sum_Ja_{s_0}(J)^2\right)^{1/2}\|z\|_2
 =S_0\|u\|_s.
\end{align*}
Here the underline denotes the componentwise absolute majorant,
\[
 (\underline X_{F_m}(u))_{(j,\sigma)}
 :=m\sum_{\mathcal M(J_1,\ldots,J_{m-1})=\sigma j}
       |F_{(J_1,\ldots,J_{m-1},(j,-\sigma))}|
                                  \prod_i|u_{J_i}|.
\]
For any sequence $c(k)$,
$\frac12\sum_{j,\sigma}|c(\sigma j)|^2=\sum_k|c(k)|^2$.
Taking the doubled output norm therefore yields
\begin{equation}\label{eq:homogeneous-evaluation}
 \|\underline X_{F_m}(u)\|_s
 \leq C m(m-1)S_0^{m-2}c_m(F)\|u\|_s^{m-1}.
\end{equation}
For every coefficient subseries,
\[
 |(X_{\Pi_{\mathscr A}F_m}(u))_{(j,\sigma)}|
       \leq(\underline X_{F_m}(u))_{(j,\sigma)},\qquad
 \|X_{\Pi_{\mathscr A}F_m}(u)\|_s
       \leq\|\underline X_{F_m}(u)\|_s.
\]

For finitely supported inputs, set
\[
 \bigl(\mathscr B_m(v_1,\ldots,v_{m-1})\bigr)_{(j,\sigma)}
 =-\sigma\ii m
 \sum_{\mathcal M(J_1,\ldots,J_{m-1})=\sigma j}
 F_{(J_1,\ldots,J_{m-1},(j,-\sigma))}
                    \prod_{i=1}^{m-1}(v_i)_{J_i}.
\]
Momentum makes its output finitely supported. Applying the same
convolution calculation with $z_{i,J}=w_s(j)|(v_i)_J|$ gives
\begin{align}\label{eq:mixed-evaluation}
 \|\mathscr B_m(v_1,\ldots,v_{m-1})\|_s
 &\leq C m c_m(F)\sum_{i=1}^{m-1}\|z_i\|_2
                       \prod_{\ell\ne i}\|a_sz_\ell\|_1\nonumber\\
 &\leq C m(m-1)S_0^{m-2}c_m(F)
                       \prod_{i=1}^{m-1}\|v_i\|_s.
\end{align}
To extend the map, put $v_i^{[N]}=\Pi_{\leq N}v_i$. Multilinearity
gives the exact telescoping identity
\begin{align*}
 &\mathscr B_m(v_1^{[M]},\ldots,v_{m-1}^{[M]})
       -\mathscr B_m(v_1^{[N]},\ldots,v_{m-1}^{[N]})\\
 &\quad=\sum_{i=1}^{m-1}\mathscr B_m
  (v_1^{[M]},\ldots,v_{i-1}^{[M]},v_i^{[M]}-v_i^{[N]},
                         v_{i+1}^{[N]},\ldots,v_{m-1}^{[N]}).
\end{align*}
For $M\geq N$, \eqref{eq:mixed-evaluation} consequently implies
\begin{align*}
 &\|\mathscr B_m(v_1^{[M]},\ldots,v_{m-1}^{[M]})
       -\mathscr B_m(v_1^{[N]},\ldots,v_{m-1}^{[N]})\|_s\\
 &\quad\leq C m(m-1)S_0^{m-2}c_m(F)
       \sum_{i=1}^{m-1}\|\Pi_{>N}v_i\|_s
                            \prod_{\ell\ne i}\|v_\ell\|_s
       \longrightarrow0.
\end{align*}
The same identity with arbitrary approximating inputs proves
independence of the approximation. Completeness of $W_s$ now defines
$\mathscr B_m:W_s^{m-1}\to W_s$, with the bound
\eqref{eq:mixed-evaluation}, and
$X_{F_m}(u)=\mathscr B_m(u,\ldots,u)$.
For each fixed output, \eqref{eq:homogeneous-evaluation} gives
\[
 \sum_{\mathcal M(J_1,\ldots,J_{m-1})=\sigma j}
  |F_{(J_1,\ldots,J_{m-1},(j,-\sigma))}|
                                  \prod_i|u_{J_i}|
 \leq\frac{\sqrt2}{mw_s(j)}\|\underline X_{F_m}(u)\|_s<\infty.
\]
Moreover, $|v_{(j,\sigma)}|\leq\sqrt2w_s(j)^{-1}\|v\|_s$.
Thus the coordinatewise limit of the truncated sums agrees with this
$W_s$ limit.

For the cutting estimate, differentiation removes one factor, so
\[
 h_N(J_1,\ldots,J_{m-1})
   =h_N(\mathcal J)-\mathbf1_{\{|j|>N\}}\geq2
       \quad\text{if }h_N(\mathcal J)\geq3.
\]
After choosing $i_*$, we have
$\#\{i\ne i_*:|j_i|>N\}\geq h_N(J_1,\ldots,J_{m-1})-1\geq1$.
Choose such a position $i_1$. Its weighted $\ell^1$ norm is
\begin{align}\label{eq:high-l1-full}
 \|\mathbf1_{|j|>N}a_sz\|_1
 &\leq\ee^{-(1-C_f)(s-s_0)f(N)}\|a_{s_0}z\|_1
  \leq S_0\ee^{-\nu f(N)}\|u\|_s.
\end{align}
The $m$ choices of the differentiated factor and the at most
$(m-1)(m-2)$ choices of $(i_*,i_1)$ give
\begin{align*}
 \|\underline X_{Q_m}(u)\|_s
 &\leq C m(m-1)(m-2)c_m(Q)\|z\|_2
           \|\mathbf1_{|j|>N}a_sz\|_1\|a_sz\|_1^{m-3}\\
 &\leq C m^3S_0^{m-2}c_m(Q)
                    \ee^{-\nu f(N)}\|u\|_s^{m-1}.
\end{align*}
If $h_N(\mathcal J)=2$ and $|j|\leq N$, then
$h_N(J_1,\ldots,J_{m-1})=2$, and the identical calculation gives
\[
 \|\Pi_{\leq N}\underline X_{Q_m}(u)\|_s
 \leq C m^3S_0^{m-2}c_m(Q)
                    \ee^{-\nu f(N)}\|u\|_s^{m-1}.
\]
Finally, $S_0\leq b$ and the definition of $\coef{\cdot}_r$ give
\begin{align*}
 \vf F_{r,s}
 &\leq C\sum_{m\geq3}m(m-1)S_0^{m-2}c_m(F)r^{m-2}
  \leq C\coef F_r,\\
 \vf Q_{r,s}
 &\leq C\ee^{-\nu f(N)}
                  \sum_{m\geq3}m^3S_0^{m-2}c_m(Q)r^{m-2}
  \leq C\ee^{-\nu f(N)}\coef Q_r,\\
 r^{-1}\sup_{B_s(r)}\|\Pi_{\leq N}X_Q\|_s
 &\leq C\ee^{-\nu f(N)}
                  \sum_{m\geq3}m^3S_0^{m-2}c_m(Q)r^{m-2}
  \leq C\ee^{-\nu f(N)}\coef Q_r.
\end{align*}
The second and third lines apply to their respective high-factor
conditions. The next proof establishes convergence of the degree sums
and of their derivatives.
\end{proof}

\subsection{Analytic evaluation and convergence}\label{app:evaluation-calculus}
\begin{proof}[Proof of Lemma~\ref{lem:scalar-evaluation}]
For $\ord F\geq h$ and $0<t<R$, termwise comparison gives
\begin{align*}
 \coef F_t
 &=\sum_{m\geq h}m^3b^{m-2}c_m(F)t^{m-2}\\
 &\leq(t/R)^{h-2}
          \sum_{m\geq h}m^3b^{m-2}c_m(F)R^{m-2}
  =(t/R)^{h-2}\coef F_R.
\end{align*}
For $s\geq s_0$, Cauchy--Schwarz gives
\[
 \sum_J|u_J|
 \leq\left(\sum_Je^{-2sf(\br j)}\right)^{1/2}\|z\|_2
 \leq\left(\sum_Ja_{s_0}(J)^2\right)^{1/2}\|z\|_2
 =S_0\|u\|_s.
\]
Dropping the momentum restriction in the scalar sum, we obtain
\begin{align*}
 |F_m(u)|&\leq c_m(F)\left(\sum_J|u_J|\right)^m
             \leq c_m(F)(S_0\|u\|_s)^m,\\
 \sup_{B_s(t)}|F|
 &\leq\sum_{m\geq3}c_m(F)(S_0t)^m
  \leq S_0^2t^2\sum_{m\geq3}b^{m-2}c_m(F)t^{m-2}
  \leq S_0^2t^2\coef F_t.
\end{align*}
The scalar multilinear form and its vector-valued derivative satisfy
\begin{align*}
 \mathscr F_m(v_1,\ldots,v_m)
   &:=\sum_{\mathcal J\in\mathcal I_m}F_{\mathcal J}
                                  \prod_{i=1}^m(v_i)_{J_i},\\
 |\mathscr F_m(v_1,\ldots,v_m)|
   &\leq c_m(F)\prod_{i=1}^m\sum_J|(v_i)_J|
    \leq c_m(F)S_0^m\prod_{i=1}^m\|v_i\|_s,\\
 -\sigma\ii\partial_{u_{(j,-\sigma)}}F_m(u)
   &=\bigl(\mathscr B_m(u,\ldots,u)\bigr)_{(j,\sigma)}.
\end{align*}

Fix an integer $a\geq0$ and $0<t<t'<R$. Differentiation of the
symmetric multilinear expression gives, for $a\leq m-1$,
\[
 D^aX_{F_m}(u)[v_1,\ldots,v_a]
 =\frac{(m-1)!}{(m-1-a)!}
       \mathscr B_m(v_1,\ldots,v_a,u,\ldots,u).
\]
Since $(m-1)!/(m-1-a)!\leq m^a$, \eqref{eq:mixed-evaluation}
implies
\begin{equation}\label{eq:degree-derivative-bound}
 \sup_{B_s(t)}\|D^aX_{F_m}(u)\|_{\mathcal L^a}
 \leq C m^{a+2}S_0^{m-2}c_m(F)t^{m-1-a};
\end{equation}
for $a>m-1$, $D^aX_{F_m}=0$. Here $\mathcal L^a$ is the
multilinear operator norm from $W_s^a$ to $W_s$, with
$\mathcal L^0=W_s$. Put
\[
 A_m(t')=m^3b^{m-2}c_m(F)(t')^{m-2},\qquad q_t=t/t',\qquad
 \sigma_a(m_*,q_t)=\sup_{m>\max(m_*,2)}m^{a-1}q_t^{m-2}.
\]
The ratio and the tail satisfy
\[
 \frac{(m+1)^{a-1}q_t^{m-1}}{m^{a-1}q_t^{m-2}}
      =(1+1/m)^{a-1}q_t\longrightarrow q_t<1,
 \qquad \sigma_a(m_*,q_t)\longrightarrow0.
\]
Consequently, for each degree cutoff $m_*\geq2$,
\begin{align}\label{eq:degree-tail-operator}
 \sum_{m>m_*}\sup_{B_s(t)}\|D^aX_{F_m}\|_{\mathcal L^a}
 &\leq Ct^{1-a}\sum_{m>m_*}m^{a-1}q_t^{m-2}A_m(t')\nonumber\\
 &\leq Ct^{1-a}\sigma_a(m_*,q_t)\coef F_{t'}\longrightarrow0.
\end{align}
For the scalar derivatives, the corresponding calculation is
\begin{align*}
 \sup_{B_s(t)}\|D^aF_m\|_{\mathcal L^a(W_s;\mathbb C)}
 &\leq m^a c_m(F)S_0^m t^{m-a}\quad(a\leq m),\\
 D^aF_m&=0\quad(a>m),\\
 \sum_{m>m_*}\sup_{B_s(t)}\|D^aF_m\|_{\mathcal L^a(W_s;\mathbb C)}
 &\leq S_0^2t^{2-a}
    \left(\sup_{m>m_*}m^{a-3}q_t^{m-2}\right)\coef F_{t'}
    \longrightarrow0.
\end{align*}

To check the Taylor expansion at $u$, take $\|u\|_s\leq t$ and
$\|h\|_s<t'-t$. The binomial formula gives the absolute bounds
\begin{align*}
 \sum_{m\geq3}\sum_{a=0}^{m}
       \frac{|D^aF_m(u)[h,\ldots,h]|}{a!}
 &\leq\sum_{m\geq3}c_m(F)S_0^m(\|u\|_s+\|h\|_s)^m
  \leq S_0^2(t')^2\coef F_{t'},\\
 \sum_{m\geq3}\sum_{a=0}^{m-1}
       \frac{\|D^aX_{F_m}(u)[h,\ldots,h]\|_s}{a!}
 &\leq C\sum_{m\geq3}m(m-1)S_0^{m-2}c_m(F)
                               (\|u\|_s+\|h\|_s)^{m-1}\\
 &\leq Ct'\coef F_{t'}.
\end{align*}
We can therefore rearrange both double sums:
\begin{align*}
 F(u+h)&=\sum_{a\geq0}\frac1{a!}
                 \left(\sum_{m\geq3}D^aF_m(u)\right)[h,\ldots,h],\\
 X_F(u+h)&=\sum_{a\geq0}\frac1{a!}
                 \left(\sum_{m\geq3}D^aX_{F_m}(u)\right)[h,\ldots,h].
\end{align*}
These convergent local power series prove holomorphy, and the uniform
operator-norm tails give $D^aF=\sum_mD^aF_m$ and
$D^aX_F=\sum_mD^aX_{F_m}$. Finally,
\[
 \coef P_t\leq(t/R)\coef P_R\leq Mt/R,\qquad
 \coef{\Pi_{\geq d+1}P}_t
   \leq(t/R)^{d-1}\coef{\Pi_{\geq d+1}P}_R
   \leq M(t/R)^{d-1}.
\]
\end{proof}

\begin{remark}[Degree summation and Fourier truncation]\label{rem:Fourier-convergence}
Write $F_m^{[N]}(u)=F_m(\Pi_{\leq N}u)$. Then
\[
 X_{F_m^{[N]}}(u)=\Pi_{\leq N}
          \mathscr B_m(\Pi_{\leq N}u,\ldots,\Pi_{\leq N}u).
\]
Telescoping the inputs and keeping the output projection gives
\begin{align}\label{eq:Fourier-strong-convergence}
 \|X_{F_m}(u)-X_{F_m^{[N]}}(u)\|_s
 &\leq\|\Pi_{>N}X_{F_m}(u)\|_s\nonumber\\
 &\quad+Cm(m-1)^2S_0^{m-2}c_m(F)
                \|u\|_s^{m-2}\|\Pi_{>N}u\|_s\longrightarrow0.
\end{align}
For a compact $K\subset W_s$ and an $\epsilon$-net
$\{v_1,\ldots,v_r\}\subset K$,
\[
 \sup_{u\in K}\|\Pi_{>N}u\|_s
 \leq\epsilon+\max_{1\leq i\leq r}\|\Pi_{>N}v_i\|_s,
 \qquad
 \limsup_{N\to\infty}\sup_{u\in K}\|\Pi_{>N}u\|_s\leq\epsilon.
\]
Letting $\epsilon\downarrow0$ and applying the same formula to
$X_{F_m}(K)$, we obtain
\[
 \sup_{u\in K}\|X_{F_m}(u)-X_{F_m^{[N]}}(u)\|_s\longrightarrow0.
\]
For $K\Subset B_s(R)$ choose $t<R$ with $K\subset B_s(t)$.
Since $c_m(F^{[N]})\leq c_m(F)$ and
$X_{F_m^{[N]}}=\Pi_{\leq N}X_{F_m}\circ\Pi_{\leq N}$,
\begin{align*}
 \sup_{u\in K}\|X_F(u)-X_{F^{[N]}}(u)\|_s
 &\leq\sum_{m=3}^{m_*}\sup_{u\in K}
                       \|X_{F_m}(u)-X_{F_m^{[N]}}(u)\|_s\\
 &\quad+2\sum_{m>m_*}\sup_{B_s(t)}\|X_{F_m}\|_s.
\end{align*}
First let $N\to\infty$ and then $m_*\to\infty$. Equation
\eqref{eq:degree-tail-operator} gives
$\sup_K\|X_F-X_{F^{[N]}}\|_s\to0$.

Uniform convergence on a whole ball fails already for
\[
 F(u)=|u_0|^2\sum_{j\geq1}|u_j|^2,\qquad
 F^{[N]}(u)=|u_0|^2\sum_{1\leq j\leq N}|u_j|^2.
\]
Here $c_4(F)\leq1$, $c_m(F)=0$ for $m\ne4$, and
$(X_F)_j=-\ii|u_0|^2u_j$ for $j\geq1$. For $J>N$ take
$u_0=r/(2w_s(0))$, $u_J=r/(2w_s(J))$, and all other coordinates
zero. Then $\|u\|_s=r/\sqrt2<r$, whereas
\begin{equation}\label{eq:Fourier-ball-example}
 \|X_F(u)-X_{F^{[N]}}(u)\|_s
 \geq w_s(J)|u_0|^2|u_J|=\frac{r^3}{8w_s(0)^2}.
\end{equation}
The uniform convergence in Lemma~\ref{lem:scalar-evaluation} is thus
convergence of the homogeneous-degree sums. For the particular module
sums of Section~\ref{sec:implant}, uniform tails are quantified in
\eqref{eq:module-tail-field}--\eqref{eq:module-tail-derivative}
and in the selector estimates for the two constructions.
\end{remark}
\subsection{Poisson brackets and canonical flows}\label{app:lie-calculus}
\begin{proof}[Proof of Lemma~\ref{lem:lie}]
Let $h=m+n-2$ and $\mathcal J=(J_1,\ldots,J_h)\in\mathcal I_h$.
For a permutation $\pi\in\mathfrak S_h$ and $\sigma\in\{+1,-1\}$,
put
\[
 S_\pi=\sum_{i=1}^{m-1}\sigma_{\pi(i)}j_{\pi(i)},\qquad
 K_{\pi,\sigma}=(-\sigma S_\pi,\sigma).
\]
The two contracted tuples have momenta
\[
 \mathcal M(J_{\pi(1)},\ldots,J_{\pi(m-1)},K_{\pi,\sigma})
        =S_\pi+\sigma(-\sigma S_\pi)=0,
\]
and
\[
 \mathcal M(J_{\pi(m)},\ldots,J_{\pi(h)},\bar K_{\pi,\sigma})
        =\mathcal M(\mathcal J)-S_\pi-\sigma(-\sigma S_\pi)=0.
\]
With symmetric ordered coefficients, the contraction formula is
\begin{align*}
 (\{F_m,G_n\})_{\mathcal J}
 =-\frac{\ii mn}{h!}
   \sum_{\pi\in\mathfrak S_h}\sum_{\sigma=\pm1}\sigma\,
   F_{(J_{\pi(1)},\ldots,J_{\pi(m-1)},K_{\pi,\sigma})}
   G_{(J_{\pi(m)},\ldots,J_{\pi(h)},\bar K_{\pi,\sigma})}.
\end{align*}
Thus, including tuples with repeated factors,
\[
 |(\{F_m,G_n\})_{\mathcal J}|
 \leq\frac{mn}{h!}\sum_{\pi\in\mathfrak S_h}
                      \sum_{\sigma=\pm1}c_m(F)c_n(G)
 =2mn c_m(F)c_n(G).
\]
Conjugation reverses $\sigma$, so
$(\{F_m,G_n\})_{\bar{\mathcal J}}
=\overline{(\{F_m,G_n\})_{\mathcal J}}$.
The momentum identities and the degree of each contraction give
\begin{equation}\label{eq:bracket-order}
 \{F_m,G_n\}\in\mathcal P_{m+n-2},\qquad
 \ord\{F,G\}\geq\ord F+\ord G-2.
\end{equation}

Set $A_m=m^3b^{m-2}c_m(F)$, $B_n=n^3b^{n-2}c_n(G)$ and
$t=1-\rho/a$. Since $h-2=(m-2)+(n-2)$ and $h\leq mn$,
\[
 \frac{h^3}{m^2n^2}\leq h,\qquad
 1-t^r=(1-t)\sum_{i=0}^{r-1}t^i\geq(1-t)rt^{r-1},\qquad
 (r+2)t^r\leq\frac{3}{1-t}\quad(r\geq1).
\]
It follows that
\begin{align*}
 \coef{\{F,G\}}_{a-\rho}
 &\leq2\sum_{m,n\geq3}\frac{h^3}{m^2n^2}
                     t^{h-2}A_ma^{m-2}B_na^{n-2}\\
 &\leq\frac{Ca}{\rho}
        \left(\sum_{m\geq3}A_ma^{m-2}\right)
        \left(\sum_{n\geq3}B_na^{n-2}\right)
 =\frac{Ca}{\rho}\coef F_a\coef G_a.
\end{align*}
For $\ell\geq1$, use $a_j=a-j\rho/\ell$, $0\leq j\leq\ell$.
At step $j$,
\[
 a_{j-1}-a_j=\rho/\ell,\qquad
 \coef G_{a_{j-1}}\leq\coef G_a,
 \qquad
 \coef{\ad_G^jF}_{a_j}
 \leq\frac{Ca\ell}{\rho}\coef G_a
                      \coef{\ad_G^{j-1}F}_{a_{j-1}}.
\]
Multiplication of these inequalities and
$\ell!\geq(\ell/\ee)^\ell$ yield
\begin{align*}
 \coef{\ad_G^\ell F}_{a-\rho}
  &\leq(Ca\ell/\rho)^\ell(\coef G_a)^\ell\coef F_a,\\
 \frac1{\ell!}\coef{\ad_G^\ell F}_{a-\rho}
  &\leq\left(\frac{C\ee a}{\rho}\coef G_a\right)^\ell\coef F_a
   \leq\eta^\ell\coef F_a.
\end{align*}
The last inequality fixes the constant in
$\eta=Ca\coef G_a/\rho$ large enough to include $\ee$.

We next construct the flow and its differential on the stated domains.
Lemma~\ref{lem:cut} gives
$\sup_{B_s(a)}\|X_G\|_s\leq Ca\coef G_a$. If
$u\in B_s(a-\rho/2)$ and $\|h\|_s=1$, the Cauchy formula is
\[
 DX_G(u)h=\frac1{2\pi\ii}
       \int_{|\zeta|=\rho/2}\frac{X_G(u+\zeta h)}{\zeta^2}\,d\zeta,
 \qquad
 \sup_{B_s(a-\rho/2)}\|DX_G\|
 \leq\frac{Ca}{\rho}\coef G_a.
\]
We choose the fixed smallness threshold for $\eta$ so that
\begin{equation}\label{eq:calculus-flow-margins}
 \sup_{B_s(a)}\|X_G\|_s\leq\rho/8,
 \qquad
 \sup_{B_s(a-\rho/2)}\|DX_G\|\leq1/8.
\end{equation}
Both follow by making
$Ca\coef G_a/\rho=\eta$ smaller than a fixed constant.

Fix $u\in B_s(a-\rho)$ and put $\mathbb D_2=\{t:|t|<2\}$.
Use the Banach space $H^\infty(\mathbb D_2;W_s)$ with its supremum
norm. Completeness follows from completeness of $W_s$ and the Cauchy
formula: if $v_n\to v$ uniformly, then, on every closed disc
$\overline{D(t_0,r)}\subset\mathbb D_2$,
\[
 v(t)=\lim_n\frac1{2\pi\ii}
       \int_{|\zeta-t_0|=r}\frac{v_n(\zeta)}{\zeta-t}\,d\zeta
     =\frac1{2\pi\ii}
       \int_{|\zeta-t_0|=r}\frac{v(\zeta)}{\zeta-t}\,d\zeta.
\]
In this space let
\[
 \mathscr A_u=\{v:\|v-u\|_\infty\leq\rho/4\},\qquad
 v\in\mathscr A_u\ \Longrightarrow\
 v(\mathbb D_2)\subset B_s(a-3\rho/4).
\]
Define
\begin{equation}\label{eq:complex-Picard}
 (\mathscr P_uv)(t)=u+t\int_0^1X_G(v(\xi t))\,d\xi.
\end{equation}
For $v\in\mathscr A_u$, \eqref{eq:calculus-flow-margins} gives
\[
 \sup_{0\leq\xi\leq1,\,t\in\mathbb D_2}
                 \|X_G(v(\xi t))\|_s\leq\rho/8.
\]
The Cauchy integral for the integrand therefore commutes with the
$\xi$ integral.
Thus $\mathscr P_uv$ is holomorphic. The estimates are
\begin{align}\label{eq:complex-Picard-contraction}
 \|\mathscr P_uv-u\|_\infty
   &\leq2\sup_{B_s(a)}\|X_G\|_s\leq\rho/4,\nonumber\\
 \|\mathscr P_uv-\mathscr P_u\widetilde v\|_\infty
   &\leq2\sup_{B_s(a-\rho/2)}\|DX_G\|
                     \|v-\widetilde v\|_\infty
    \leq\tfrac14\|v-\widetilde v\|_\infty.
\end{align}
Starting from $v_0(t)=u$, put $v_{n+1}=\mathscr P_uv_n$. Then
\[
 \|v_{n+1}-v_n\|_\infty\leq4^{-n}\rho/4,
 \qquad
 \|v_{n+r}-v_n\|_\infty
    \leq\frac{\rho}{4}\sum_{j=n}^{n+r-1}4^{-j}
    \leq\frac{\rho}{3}4^{-n}.
\]
The limit $v_u\in\mathscr A_u$ satisfies
\[
 v_u=\mathscr P_uv_u,\qquad v_u'=X_G(v_u),\qquad v_u(0)=u.
\]
For two fixed points, \eqref{eq:complex-Picard-contraction} gives
\[
 \|v_u-\widetilde v_u\|_\infty
      \leq\tfrac14\|v_u-\widetilde v_u\|_\infty
 \quad\Longrightarrow\quad v_u=\widetilde v_u.
\]

The iterates $v_n(u,t)$ are also holomorphic in $u$, and the last
geometric bound is uniform in $u\in B_s(a-\rho)$. For a ball
$B_s(u,\sigma)\subset B_s(a-\rho)$, the Cauchy estimate gives
\[
 \sup_{B_s(u,\sigma/2)}\|D_uv_n-D_uv_m\|
 \leq\frac{2}{\sigma}
       \sup_{B_s(u,\sigma)}\|v_n-v_m\|_\infty\longrightarrow0.
\]
For $y\in B_s(u,\sigma/2)$ and $\|h\|_s<\sigma/4$, the same
Cauchy formula gives
\[
 \sup_{|t|<2}\|v_n(y+h,t)-v_n(y,t)-D_uv_n(y,t)h\|_s
          \leq Ca\sigma^{-2}\|h\|_s^2.
\]
Passing to the limit in the Taylor formula proves Fr\'echet
holomorphy of $u\mapsto v_u(t)$. We denote the flow by
$\Phi_G^t(u)=v_u(t)$. For real $|t|\leq1$, integration of its
equation gives
\begin{align*}
 \|\Phi_G^t(u)-u\|_s
 &\leq |t|\sup_{B_s(a)}\|X_G\|_s
  \leq Ca\coef G_a\leq\rho/8,\\
 \Phi_G^t(B_s(a-\rho))&\subset B_s(a-7\rho/8).
\end{align*}

Put $V(t)=D\Phi_G^t(u)$ and $A(t)=DX_G(\Phi_G^t(u))$.
The variational equation and its iterated integral imply
\begin{align*}
 V(t)&=\Id+\int_0^t A(\tau)V(\tau)\,d\tau,\qquad
          \|A(t)\|\leq C\eta,\\
 \|V(t)\|&\leq\sum_{j\geq0}\frac{(C\eta|t|)^j}{j!}
       =\ee^{C\eta|t|},\qquad
 \|V(t)-\Id\|\leq\ee^{C\eta|t|}-1\leq C\eta.
\end{align*}
For the inverse differential, start with $W_0=\Id$ and iterate
\[
 W_{n+1}(t)=\Id-\int_0^tW_n(\tau)A(\tau)\,d\tau,
 \qquad
 \|W_{n+1}(t)-W_n(t)\|
       \leq\frac{(C\eta|t|)^{n+1}}{(n+1)!}.
\]
The limit solves $W'=-WA$, $W(0)=\Id$, and
$\|W(t)\|\leq\sum_{n\geq0}(C\eta|t|)^n/n!=\ee^{C\eta|t|}$.
Moreover $(WV)'=0$ and $(VW-\Id)'=A(VW-\Id)-(VW-\Id)A$,
with zero initial value in the latter equation. Iteration gives
\[
 \|V(t)W(t)-\Id\|
 \leq\sup_{|\tau|\leq1}\|V(\tau)W(\tau)-\Id\|
                  \frac{(2C\eta|t|)^n}{n!}\longrightarrow0.
\]
Consequently
\[
 W(t)V(t)=V(t)W(t)=\Id,\qquad
 \max\{\|D\Phi_G^t(u)\|,\|(D\Phi_G^t(u))^{-1}\|\}
 \leq\ee^{C\eta|t|}\leq1+C\eta.
\]
The flow of $-G$ supplies the inverse map along these trajectories.
Reality of the coefficients gives
$X_G(\mathfrak c u)=\mathfrak c X_G(u)$, hence, by uniqueness,
$\Phi_G^t(\mathfrak c u)=\mathfrak c\Phi_G^t(u)$ for real $t$.

For finite Fourier vectors $v,h$, put
$v_t=D\Phi_G^t(u)v$ and $h_t=D\Phi_G^t(u)h$. Then
\begin{align*}
 \frac d{dt}\Omega_{\rm symp}(v_t,h_t)
 &=\Omega_{\rm symp}(DX_G(\Phi_G^t(u))v_t,h_t)
   +\Omega_{\rm symp}(v_t,DX_G(\Phi_G^t(u))h_t)\\
 &=D^2G(\Phi_G^t(u))[v_t,h_t]
   -D^2G(\Phi_G^t(u))[h_t,v_t]=0.
\end{align*}
Integration gives
$\Omega_{\rm symp}(D\Phi_G^t(u)v,D\Phi_G^t(u)h)
=\Omega_{\rm symp}(v,h)$. Approximation by finite Fourier vectors,
together with the differential bounds, extends this equality to all
$v,h\in W_s$. Thus $(\Phi_G^t)^*\Omega_{\rm symp}=\Omega_{\rm symp}$.

It remains to identify the Lie series with the pullback. For each
fixed $\ell$ and $r<a$, \eqref{eq:liecauchy} gives
\[
 \coef{\ad_G^\ell F}_r
 \leq\ell!\left(\frac{Ca}{a-r}\coef G_a\right)^\ell\coef F_a<\infty.
\]
Lemma~\ref{lem:scalar-evaluation}
and the chain rule therefore give, on $\mathbb D_2$,
\[
 \frac{d^\ell}{dt^\ell}F(v_u(t))
     =(\ad_G^\ell F)(v_u(t)),\qquad
 \left.\frac{d^\ell}{dt^\ell}F(v_u(t))\right|_{t=0}
     =(\ad_G^\ell F)(u).
\]
The scalar evaluation bound gives
$\sup_{|t|<2}|F(v_u(t))|\leq Ca^2\coef F_a$.
Cauchy's formula on $|t|=3/2$ therefore yields
\[
 \sup_{u\in B_s(a-\rho)}
 \left|F(v_u(1))-\sum_{\ell=0}^{L}\frac1{\ell!}\ad_G^\ell F(u)\right|
 \leq Ca^2\coef F_a\sum_{\ell>L}(2/3)^\ell\longrightarrow0.
\]
The coefficient tails satisfy
\[
 \sum_{\ell>L}\frac1{\ell!}\coef{\ad_G^\ell F}_{a-\rho}
 \leq\sum_{\ell>L}\eta^\ell\coef F_a
 =\frac{\eta^{L+1}}{1-\eta}\coef F_a\longrightarrow0.
\]
For the scalar tails, \eqref{eq:tail-complete} gives
\[
 \sup_{B_s(a-\rho)}
     \left|\sum_{\ell>L}\frac1{\ell!}\ad_G^\ell F\right|
 \leq C(a-\rho)^2\frac{\eta^{L+1}}{1-\eta}\coef F_a\longrightarrow0.
\]
The scalar Taylor limit and the coefficient limit consequently agree:
\begin{align*}
 F\circ\Phi_G&=\sum_{\ell\geq0}\frac1{\ell!}\ad_G^\ell F,\\
 \coef{F\circ\Phi_G}_{a-\rho}
 &\leq\sum_{\ell\geq0}\eta^\ell\coef F_a
   =\frac{1}{1-\eta}\coef F_a
   \leq(1+2\eta)\coef F_a\qquad(\eta\leq1/2).
\end{align*}
Finally, for $u\in B_s(a-\rho)$ and $h\in W_s$,
\begin{align*}
 d(E\circ\Phi_G)(u)[h]
 &=\Omega_{\rm symp}(X_E(\Phi_G(u)),D\Phi_G(u)h)\\
 &=\Omega_{\rm symp}((D\Phi_G(u))^{-1}X_E(\Phi_G(u)),h).
\end{align*}
Nondegeneracy of the form and the domain inclusion give
\begin{align*}
 X_{E\circ\Phi_G}(u)&=(D\Phi_G(u))^{-1}X_E(\Phi_G(u)),\\
 \vf{E\circ\Phi_G}_{a-\rho,s}
 &\leq\frac{1+C\eta}{a-\rho}
                   \sup_{B_s(a)}\|X_E\|_s
  =\frac a{a-\rho}(1+C\eta)\vf E_{a,s}.
\end{align*}
\end{proof}

\subsection{Finite integral remainders}\label{app:integral-calculus}
\begin{proof}[Proof of Lemma~\ref{lem:integral-lie}]
For $j\geq0$, the chain rule gives
$\partial_t^j(F\circ\Phi_G^t)=(\ad_G^jF)\circ\Phi_G^t$.
Integrating $\ell_*+1$ times, for $0\leq t\leq1$, gives
\begin{align*}
 F\circ\Phi_G^t-\sum_{j=0}^{\ell_*}\frac{t^j}{j!}\ad_G^jF
 &=\int_{0<v_{\ell_*+1}<\cdots<v_1<t}
   (\ad_G^{\ell_*+1}F)\circ\Phi_G^{v_{\ell_*+1}}
                           \,dv_{\ell_*+1}\cdots dv_1\\
 &=\frac1{\ell_*!}\int_0^t(t-v)^{\ell_*}
             (\ad_G^{\ell_*+1}F)\circ\Phi_G^v\,dv.
\end{align*}
This is \eqref{eq:integral-lie-full}. Substitution of the convergent
Lie series into the last integral is justified by the coefficient
bound, and its $r$th coefficient is
\[
 \frac1{\ell_*!r!}\int_0^t(t-v)^{\ell_*}v^r\,dv
       =\frac{t^{\ell_*+r+1}}{(\ell_*+r+1)!}.
\]
Thus, denoting the last integral by $\mathcal R_{\ell_*}(t)$,
\begin{align*}
 \mathcal R_{\ell_*}(t)
   &=\sum_{j\geq\ell_*+1}\frac{t^j}{j!}\ad_G^jF,\\
 \coef{\mathcal R_{\ell_*}(t)}_{a-\rho}
   &\leq\sum_{j\geq\ell_*+1}(|t|\eta)^j\coef F_a
    \leq\frac{\eta^{\ell_*+1}}{1-\eta}\coef F_a.
\end{align*}
If $\ord F\geq h$ and $\ord G\geq k$, repeated use of
\eqref{eq:bracket-order} gives
\[
 \ord\ad_G^jF\geq h+j(k-2),\qquad
 \Pi_{<h+(\ell_*+1)(k-2)}\mathcal R_{\ell_*}(t)=0.
\]
Equivalently, the integral remainder has order at least
$h+(\ell_*+1)(k-2)$. For negative time,
$\Phi_G^{-t}=\Phi_{-G}^t$ and
$\ad_{-G}^jF=(-1)^j\ad_G^jF$; the same estimates apply with $|t|$.
\end{proof}
\section{Supplementary normalization, dynamics and norm comparisons}\label{app:supplementary}
This appendix gives the finite integral realization of the normal form,
the separate evolution of the low and high components, and the extension
to fixed changes of the Fourier weight. The main recurrence and the
full-energy proof are in Sections~\ref{sec:iteration}
and~\ref{sec:stability}.

\subsection{The finite form of one normalization step}\label{app:finite-step}
Fix a stage $3\leq k\leq d$ of Lemma~\ref{lem:iteration-full} and use
$\ad_kF=\{F,G_k\}$, $\Phi_k=\Phi_{G_k}^1$ and the polynomials
of \eqref{eq:Sk}. The order identities in that proof give
$\ord\ad_k^jZ^{(k)}\geq3+j(k-2)$ and
$\ord\ad_k^jP_k,\ord\ad_k^jZ_k^*\geq k+j(k-2)$.

For an explicit finite expansion through degree $d$, use
\[
 L_Z=\left\lfloor\frac{d-3}{k-2}\right\rfloor,
 \qquad L_P=\left\lfloor\frac{d-k}{k-2}\right\rfloor.
\]
Put $B_k=Z_k^*-P_k$. The three integral remainders are
\begin{align}\label{eq:finite-iteration-tails}
 \mathcal R_{k,Z}
 &=\frac1{L_Z!}\int_0^1(1-v)^{L_Z}
       (\ad_k^{L_Z+1}Z^{(k)})\circ\Phi_k^v\,dv,\nonumber\\
 \mathcal R_{k,P}
 &=\frac1{L_P!}\int_0^1(1-v)^{L_P}
       (\ad_k^{L_P+1}P_k)\circ\Phi_k^v\,dv,\nonumber\\
 \mathcal R_{k,H}
 &=\frac1{L_P!}\int_0^1\int_0^t(t-v)^{L_P}
       (\ad_k^{L_P+1}B_k)\circ\Phi_k^v\,dv\,dt.
\end{align}
Lemma~\ref{lem:integral-lie}, applied to $Z^{(k)}$, $P_k$ and
$B_k$, yields the identities
\begin{align}\label{eq:finite-iteration-expansions}
 Z^{(k)}\circ\Phi_k
  &=\sum_{j=0}^{L_Z}\frac{\ad_k^jZ^{(k)}}{j!}
       +\mathcal R_{k,Z},\nonumber\\
 P_k\circ\Phi_k
  &=\sum_{j=0}^{L_P}\frac{\ad_k^jP_k}{j!}
       +\mathcal R_{k,P},\nonumber\\
 H_0\circ\Phi_k
  &=H_0+\sum_{j=0}^{L_P}\frac{\ad_k^jB_k}{(j+1)!}
       +\mathcal R_{k,H}.
\end{align}
For the last identity we first expand $B_k\circ\Phi_k^t$ and
then integrate from $t=0$ to $t=1$. In particular, if
$j=L_P+1+r$ and $r\geq0$, the coefficient in its integral tail is
\[
 \frac1{L_P!r!}\int_0^1\int_0^t(t-v)^{L_P}v^r\,dv\,dt
 =\frac1{(L_P+r+2)!}=\frac1{(j+1)!}.
\]
It follows that
\begin{align}\label{eq:finite-tail-bounds}
 \coef{\mathcal R_{k,Z}}_{\rho_{k+1}}
   &\leq\frac{\eta_k^{L_Z+1}}{1-\eta_k}
                         \coef{Z^{(k)}}_{\rho_k},\nonumber\\
 \coef{\mathcal R_{k,P}}_{\rho_{k+1}}
   &\leq\frac{\eta_k^{L_P+1}}{1-\eta_k}\coef{P_k}_{\rho_k},\nonumber\\
 \coef{\mathcal R_{k,H}}_{\rho_{k+1}}
   &\leq\frac{\eta_k^{L_P+1}}{1-\eta_k}\coef{B_k}_{\rho_k}
    \leq\frac{\eta_k^{L_P+1}}{1-\eta_k}\coef{P_k}_{\rho_k}.
\end{align}
The final inequality uses $B_k=-\Pi_{\mathscr N}P_k$ from
\eqref{eq:homo-coefficients}. The orders of the tails are
\begin{align}\label{eq:finite-tail-orders}
 \ord\mathcal R_{k,Z}&\geq3+(L_Z+1)(k-2)>d,\nonumber\\
 \ord\mathcal R_{k,P},\ \ord\mathcal R_{k,H}
                  &\geq k+(L_P+1)(k-2)>d.
\end{align}
The strict inequalities follow directly from the definitions of
$L_Z,L_P$, including the terminal case $L_P=0$.

Combining \eqref{eq:finite-iteration-expansions}, we obtain the
finite polynomial and the analytic tail of $S_k$:
\begin{align}\label{eq:Sk-finite}
 \mathscr S_k
  &=\sum_{j=1}^{L_Z}\frac{\ad_k^jZ^{(k)}}{j!}
    +\sum_{j=1}^{L_P}\frac{\ad_k^j(jP_k+Z_k^*)}{(j+1)!},
       \nonumber\\
 \mathscr R_k&=\mathcal R_{k,Z}+\mathcal R_{k,P}
                                      +\mathcal R_{k,H},\qquad
 S_k=\mathscr S_k+\mathscr R_k.
\end{align}
Empty sums equal zero. Although $\mathscr S_k$ is finite, some of
its terms may have degree greater than $d$, since the input
polynomials need not be homogeneous. Equations
\eqref{eq:finite-tail-orders} therefore give the precise projections
\begin{equation}\label{eq:Sk-finite-projections}
 \Pi_{[k+1,d]}S_k=\Pi_{[k+1,d]}\mathscr S_k,
 \qquad
 \Pi_{\geq d+1}S_k=\Pi_{\geq d+1}\mathscr S_k+\mathscr R_k.
\end{equation}
These formulas account for every term discarded from the finite
normal polynomial.

\subsection{Separate low- and high-mode estimates}\label{app:component-dynamics}
Let $w$ solve the transformed equation and write
$w=w^<+w^>$ with cutoff $N$. Put $\mathcal R_d=R_T+E_d$ and
$Z_2=Z_{2,\mathrm{ex}}+Z_{\rm blk}$. Since $Z_0$ depends only
on $w^<$ and $Z_2$ is quadratic in the high variables, the equations
split exactly as
\begin{align}\label{eq:low-high-equations-full}
 \dot w^<&=\mathscr L w^<+X_{Z_0}(w^<)
       +\Pi_{\leq N}X_{Z_2}(w^<,w^>)+\Pi_{\leq N}X_{\mathcal R_d}(w),\nonumber\\
 \dot w^>&=\mathscr L w^>+\Pi_{>N}X_{Z_2}(w^<,w^>)
                              +\Pi_{>N}X_{\mathcal R_d}(w).
\end{align}
The next two propositions keep the two components of this system
separate. Their combination gives another proof of the integrated estimate
used in Section~\ref{sec:stability}.

\begin{proposition}[Low-mode estimate]\label{prop:low-full}
On any interval on which $\|w(t)\|_s\leq r$, with the normal
form defined on $B_s(2r)$,
\begin{equation}\label{eq:low-estimate-full}
 \big|\|w^<(t)\|_s-\|w^<(t_0)\|_s\big|
 \leq Cr^2\mathfrak e_{d,N}(r)|t-t_0|.
\end{equation}
\end{proposition}
\begin{proof}
For a real vector put
$E_<(w)=\sum_{|j|\leq N}w_s(j)^2|w_j|^2$.
The diagonal linear part commutes with $E_<$. In an exact
resonant monomial the opposite-sign factors are paired at equal
absolute indices. Each pair is either entirely below $N$ or
entirely above it. Consequently
$\{E_<,Z_0+Z_{2,\mathrm{ex}}\}=0$.
For the remaining terms, the Hamiltonian convention gives
\begin{align*}
 \left|\frac{d}{dt}E_<(w(t))\right|
 &\leq2\|w^<\|_s
   \left(\|\Pi_{\leq N}X_{Z_{\rm blk}}\|_s
                         +\|\Pi_{\leq N}X_{\mathcal R_d}\|_s\right)\\
 &\leq Cr^2\mathfrak e_{d,N}(r)\sqrt{E_<(w(t))}.
\end{align*}
Here \eqref{eq:blocklow}, \eqref{eq:taylorfinal} and
\eqref{eq:highfinal} were used. Replace $\sqrt{E_<}$ by
$\sqrt{E_<+h}$ in the differentiation, integrate and let
$h\downarrow0$. This proves \eqref{eq:low-estimate-full} even
when the low component vanishes at a time in the interval.
The sum defining $E_<$ is finite, so the same computation applies
coordinatewise to mild solutions.
\end{proof}

\begin{proposition}[High-mode isometric propagator]\label{prop:high-full}
Under the same assumptions,
\begin{equation}\label{eq:high-estimate-full}
 \big|\|w^>(t)\|_s-\|w^>(t_0)\|_s\big|
 \leq Cr^2\mathfrak e_{d,N}(r)|t-t_0|.
\end{equation}
The homogeneous high equation in \eqref{eq:low-high-equations-full}
has an isometric propagator on $\Pi_{>N}W_s$.
\end{proposition}
\begin{proof}
Every exactly resonant monomial with two high factors pairs those
factors at equal modulus and with opposite signs. The same is true
by definition for $Z_{\rm blk}$. Thus on the real slice
\begin{equation}\label{eq:high-matrix-full}
 Z_2=\sum_{n>N}\sum_{j,k\in\mathcal B_n}
                    \bar w_j A_{jk}(w^<)w_k,
 \qquad A_{jk}(w^<)=\overline{A_{kj}(w^<)}.
\end{equation}
The identity on the coefficients follows by conjugating the
polynomial and collecting the same monomial. It implies
\[
 \frac{d}{dt}\sum_{j\in\mathcal B_n}|v_j|^2
 =2\operatorname{Re}\sum_{j,k\in\mathcal B_n}
                  \bar v_j(-\ii A_{jk})v_k=0,
\]
for the homogeneous block equation. All indices in the block carry
the same weight, so this is also conservation of its weighted mass.

At $(w^<,0)$ with $\|w^<\|_s\leq r$, differentiation in a
high direction and Cauchy's estimate with margin $r/2$ give
\[
 \|A(w^<)\|
 =\|D_{w^>}\Pi_{>N}X_{Z_2}(w^<,0)\|
 \leq\frac2r\sup_{B_s(2r)}\|X_{Z_2}\|_s
 \leq\frac2r Cr^2\leq Cr.
\]
Here \eqref{eq:Zfinal} and Lemma~\ref{lem:cut} bound the field.
Along the continuous low-mode solution, $A(t)=A(w^<(t))$ is
continuous in operator norm. Its block support gives
$A(t)U(\tau)=U(\tau)A(t)$ because $\omega_j=\omega_k$ on each block.
In the interaction variable $v(t)=U(-t)w^>(t)$, the homogeneous
equation is $\dot v=-\ii A(t)v$. For $t\geq t_0$ its propagator is
\[
 \mathcal V(t,t_0)=\Id+\sum_{k\geq1}(-\ii)^k
  \int_{t_0<\tau_k<\cdots<\tau_1<t}
                   A(\tau_1)\cdots A(\tau_k)\,d\tau_k\cdots d\tau_1.
\]
For each integral in this series,
\[
 \left\|\int_{t_0<\tau_k<\cdots<\tau_1<t}
                   A(\tau_1)\cdots A(\tau_k)\,d\tau_k\cdots d\tau_1\right\|
 \leq\frac{(Cr|t-t_0|)^k}{k!}.
\]
The sum converges in operator norm on every compact time interval.
The Hermitian identity \eqref{eq:high-matrix-full} and the equal
weights within each block imply
\[
 (A(t)v,h)_s=(v,A(t)h)_s,\qquad
 \frac d{dt}\|v(t)\|_s^2
       =2\operatorname{Re}(-\ii A(t)v(t),v(t))_s=0.
\]
The backward equation yields
$\mathcal V(t_0,t)\mathcal V(t,t_0)=\Id$, so
\[
 \|\mathcal V(t,t_0)v\|_s=\|v\|_s,\qquad
 \mathcal W(t,t_0)=U(t)\mathcal V(t,t_0)U(-t_0),\qquad
 \|\mathcal W(t,t_0)v\|_s=\|v\|_s.
\]

Variation of constants in the full high equation now reads
\[
 w^>(t)=\mathcal W(t,t_0)w^>(t_0)
       +\int_{t_0}^t\mathcal W(t,\xi)\Pi_{>N}X_{\mathcal R_d}(w(\xi))\,d\xi.
\]
Using the isometry just proved, we obtain
\[
 \big|\|w^>(t)\|_s-\|w^>(t_0)\|_s\big|
 \leq\int_{\min(t,t_0)}^{\max(t,t_0)}
                      \|X_{\mathcal R_d}(w(\xi))\|_s\,d\xi
 \leq Cr^2\mathfrak e_{d,N}(r)|t-t_0|.
\]
This proves \eqref{eq:high-estimate-full}.
For $n>M_{d,N}$ the off-diagonal coefficient is zero:
if $j=-k\ne0$, its monomial momentum would require
$2n=|j-k|\leq(d-2)N$, contrary to the buffer choice. The
propagator argument covers the possible off-diagonal terms in the
finite intermediate region as well.
\end{proof}

Combining Propositions~\ref{prop:low-full} and~\ref{prop:high-full}
by the triangle inequality in
$\mathbb R^2$ gives
\[
 \|w(t)\|_s\leq\|w(t_0)\|_s
                     +Cr^2\mathfrak e_{d,N}(r)|t-t_0|.
\]
Indeed apply the low and high upper bounds to the two nonnegative
coordinates $(\|w^<\|_s,\|w^>\|_s)$ and use
$\|w\|_s^2=\|w^<\|_s^2+\|w^>\|_s^2$.
This is the integrated estimate of Lemma~\ref{lem:energy}, with
Propositions~\ref{prop:low-full} and~\ref{prop:high-full} giving
the bounds for each component.

\subsection{Fixed polynomial factors and equivalent radial weights}\label{app:radial-weights}
We extend the stability theorem to fixed polynomial factors and
bounded even multipliers in the Fourier weight.

\begin{proposition}[Fixed changes of the Fourier weight]\label{prop:radial-weights}
Fix one of the two weights $f$, $s>0$, $a\geq0$, and a positive
even sequence $(\chi_j)$ with
$0<c_\chi\leq\chi_j\leq C_\chi<\infty$. Set
\begin{equation}\label{eq:modified-radial-norm}
 \|u\|_{s,a,\chi}^2
 =\sum_{j\in\ZZ}\chi_j^2\br j^{2a}
                       \ee^{2sf(\br j)}|u_j|^2.
\end{equation}
Suppose the frequencies satisfy Assumption~\ref{ass:freq} and
$P$ satisfies Definition~\ref{def:coef} at some positive radius.
For every fixed $A>1$, the lower-bound conclusion of
Theorem~\ref{thm:joint} holds in \eqref{eq:modified-radial-norm},
with $L=\log(1/\|u_0\|_{s,a,\chi})$ and the same functions
$\Phi_{G,p}$ and $\Phi_{U,p,q}$. The constants may depend on
$a,c_\chi,C_\chi$ as well as the fixed data of that theorem.
\end{proposition}
\begin{proof}
Put $\widetilde f(x)=f(x)+(a/s)\log x$ and
$M_f=\sup_{x\geq c_0}\log x/f(x)<\infty$. For $x,y\geq2$,
$x+y\leq xy$; induction consequently gives
\[
 \log(x_1+\cdots+x_r)
 \leq\log x_*+\sum_{i\ne *}\log x_i.
\]
Combining this with \eqref{eq:strictsub}, we obtain
\begin{align*}
 \widetilde f(x_1+\cdots+x_r)-\widetilde f(x_*)
 &\leq\sum_{i\ne *}\left(C_f f(x_i)+(a/s)\log x_i\right)\\
 &\leq\widetilde C_f\sum_{i\ne *}\widetilde f(x_i),\qquad
 \widetilde C_f=1-\frac{1-C_f}{1+(a/s)M_f}<1.
\end{align*}
Moreover $\widetilde f\geq f$, so the summability condition holds
for every positive auxiliary width. Fix $\widetilde s_0=s/2$
and write $\widetilde\nu=(1-\widetilde C_f)s/2>0$.

For the bounded radial factor, let
$K_\chi=\max(1,C_\chi,c_\chi^{-1})$ and
$\widetilde w(j)=\chi_j\ee^{s\widetilde f(\br j)}$.
The momentum calculation of Appendix~\ref{app:cut-calculus}
now gives, for a Hamiltonian of degree $m$,
\begin{align*}
 \frac{\widetilde w(j)}{\prod_{i=1}^{m-1}\widetilde w(j_i)}
 &\leq K_\chi^m
    \exp\!\left(-(1-\widetilde C_f)s
                   \sum_{i\ne i_*}\widetilde f(\br{j_i})\right),\\
 \frac{\widetilde w(j)}{\prod_{i=1}^{m-1}\widetilde w(j_i)}
 &\leq K_\chi^m\ee^{-\widetilde\nu\widetilde f(N)}
    \exp\!\left(-(1-\widetilde C_f)\widetilde s_0
                   \sum_{i\ne i_*}\widetilde f(\br{j_i})\right)
\end{align*}
where the second inequality applies when one of the positions
$i\ne i_*$ is high. Set
\[
 \widetilde S_0^2=2\sum_{J\in\mathcal Z}
      \ee^{-2(1-\widetilde C_f)\widetilde s_0\widetilde f(\br j)},
 \qquad \widetilde b=\max\{b,K_\chi\max(2,2\widetilde S_0)\}.
\]
In the convolution calculation of Appendix~\ref{app:cut-calculus},
$K_\chi^m=K_\chi^2K_\chi^{m-2}$ is absorbed by the fixed
constant and $\widetilde b^{m-2}$, respectively. Writing
$\|F\|^{\#,\sim}_r=\sum_{m\geq3}m^3\widetilde b^{m-2}c_m(F)r^{m-2}$,
we obtain
\begin{align*}
 r^{-1}\sup_{\|u\|_{s,a,\chi}<r}\|X_F(u)\|_{s,a,\chi}
 &\leq C\|F\|^{\#,\sim}_r,\\
 r^{-1}\sup_{\|u\|_{s,a,\chi}<r}\|X_Q(u)\|_{s,a,\chi}
 &\leq C\ee^{-\widetilde\nu\widetilde f(N)}\|Q\|^{\#,\sim}_r
       \quad\text{if $Q$ has at least three high factors},\\
 r^{-1}\sup_{\|u\|_{s,a,\chi}<r}
                      \|\Pi_{\leq N}X_Q(u)\|_{s,a,\chi}
 &\leq C\ee^{-\widetilde\nu\widetilde f(N)}\|Q\|^{\#,\sim}_r
       \quad\text{if $Q$ has two high factors}.
\end{align*}
The coefficient radius $\widetilde R=bR/\widetilde b$ satisfies
\[
 \sum_{m\geq3}m^3\widetilde b^{m-2}c_m(P)
                    \widetilde R^{m-2}
 =\sum_{m\geq3}m^3b^{m-2}c_m(P)R^{m-2}\leq M.
\]
The bracket contraction is unchanged, and its radius-loss estimate
uses this new coefficient norm. Therefore the recurrence of
Lemma~\ref{lem:iteration-full}, including the transport identity,
and the restriction estimates of Theorem~\ref{thm:nf} give,
with $\Delta=\Delta_{d,N}$,
\begin{align*}
 \vf{R_T}_{2r,s,a,\chi}
 &\leq Cr(Cd^2\Delta r)^{d-2},\\
 \vf{E_d}_{2r,s,a,\chi}
 &\leq Cr\ee^{-\widetilde\nu\widetilde f(N)},\\
 r^{-1}\sup_{\|u\|_{s,a,\chi}<2r}
                  \|\Pi_{\leq N}X_{Z_{\rm blk}}(u)\|_{s,a,\chi}
 &\leq Cr\ee^{-\widetilde\nu\widetilde f(N)}.
\end{align*}
Here the vector-field seminorm is defined as in
Definition~\ref{def:coef}, using \eqref{eq:modified-radial-norm}.
All constants remain independent of $d,N,r$.
Since $\chi$ is even, paired indices have equal
$\widetilde w$ weights. The exact cancellations in
Lemma~\ref{lem:energy} consequently give
\[
 \left|\frac d{dt}\|w(t)\|_{s,a,\chi}\right|
 \leq Cr^2\left[(Cd^2\Delta r)^{d-2}
                  +\ee^{-\widetilde\nu\widetilde f(N)}\right].
\]
Finally, $\log x=o(f(x))$ implies
$\widetilde f(N)/f(N)\to1$. The choices in
\eqref{eq:Gparameters} and \eqref{eq:Uparameters}, with adjusted
fixed constants, therefore satisfy
\[
 d^p\log(dN)\leq cL,\qquad
 \widetilde f(N)\asymp dL\asymp\Phi(L).
\]
The graph-domain argument also holds in this norm. Writing
$\|u\|_*=\|u\|_{s,a,\chi}$ and
$\|u\|_{\mathscr D,*}=\|u\|_*+\|\mathscr Lu\|_*$, define
\[
 A_F^*(r)=C\sum_{m\geq3}m^{\beta+4}\widetilde b^{m-2}
                       c_m(F)r^{m-2}.
\]
This series is finite inside the coefficient radius, by the
geometric comparison preceding \eqref{eq:graph-tame-differential}.
For each output index,
\[
 \br {j_{\rm out}}^\beta
 \leq(m-1)^{\beta-1}\sum_{i=1}^{m-1}\br {j_i}^\beta.
\]
Applying the preceding weight-ratio estimate with this polynomial
factor on each input in turn gives
\begin{align*}
 \|X_F(u)\|_{\mathscr D,*}
 &\leq A_F^*(r)\|u\|_{\mathscr D,*},\\
 \|DX_F(u)h\|_{\mathscr D,*}
 &\leq A_F^*(r)\bigl(\|h\|_{\mathscr D,*}
       +r^{-1}\|u\|_{\mathscr D,*}\|h\|_*\bigr).
\end{align*}
Integration on the segment between two inputs gives the
corresponding difference estimate \eqref{eq:graph-difference}.
These are the estimates used for graph-domain conjugacy and its
extension to mild solutions in Subsection~\ref{subsec:flow}.
The near-identity and continuation argument of
Section~\ref{sec:stability}, applied to this exact norm, now gives
$\|u(t)\|_{s,a,\chi}<A\|u_0\|_{s,a,\chi}$ on the stated
interval.
\end{proof}

\section*{Data availability and competing interests}
Data sharing is not applicable to this article as no datasets were
generated or analysed during the current study.
The authors have no conflicts of interest to disclose.

\section*{Acknowledgement}
The second author (Y. Li) was supported by the National Natural
Science Foundation of China (12071175, 12471183 and 12531009).


\begin{thebibliography}{99}
\setlength{\itemsep}{3pt plus 1pt}

\bibitem[BGGH25]{BGGH25} P. Baldi, F. Giuliani, M. Guardia and E. Haus,
Effective chaos for the Kirchhoff equation on tori.
Ann. Inst. H. Poincar\'e Anal. Non Lin\'eaire 42 (2025), 281--330.
\url{https://doi.org/10.4171/AIHPC/110}

\bibitem[B99]{B99} D. Bambusi.
Nekhoroshev theorem for small amplitude solutions in nonlinear Schr\"odinger equations.
Math. Z. 230 (1999), 345--387.
\url{https://doi.org/10.1007/PL00004696}

\bibitem[BDGS07]{BDGS07}D. Bambusi,  J. M. Delort,   B. Grébert, J. Szeftel, 
Almost global existence for Hamiltonian semilinear Klein-Gordon equations with small Cauchy data on Zoll manifolds.
Comm. Pure Appl. Math. 60 (2007), no. 11, 1665--1690.
\url{https://doi.org/10.1002/cpa.20181}

\bibitem[BFLM25]{BFLM25} D. Bambusi, R. Feola, B. Langella and F. Monzani,
Almost global existence for some Hamiltonian PDEs on manifolds with globally integrable geodesic flow.
Nonlinearity 38 (2025), no. 5, 055013, 57 pp.
\url{https://doi.org/10.1088/1361-6544/adc967}

\bibitem[BFM24]{BFM24}D. Bambusi,  R. Feola,  R. Montalto,
Almost global existence for some Hamiltonian PDEs with small Cauchy data on general tori.
Comm. Math. Phys. 405 (2024), no. 1, Paper No. 15, 50 pp.
\url{https://doi.org/10.1007/s00220-023-04899-z}

\bibitem[BG24]{BG24} D. Bambusi and P. G\'erard.
A Nekhoroshev theorem for some perturbations of the Benjamin--Ono equation with initial data close to finite gap tori.
Math. Z. 307 (2024), Paper No. 54.
\url{https://doi.org/10.1007/s00209-024-03539-z}

\bibitem[BG06]{BG06}D. Bambusi, B. Grébert, 
Birkhoff normal form for partial differential equations with tame modulus.
Duke Math. J. 135 (2006), no. 3, 507--567.
\url{https://doi.org/10.1215/S0012-7094-06-13534-2}

\bibitem[B88]{B88} G. Benettin,  J. Fr\"ohlich, A. Giorgilli,
A Nekhoroshev-type theorem for Hamiltonian systems with infinitely many degrees of freedom.
Comm. Math. Phys. 119 (1988), no. 1, 95--108.
\url{https://doi.org/10.1007/BF01218262}

\bibitem[BC25]{BC25} J. Bernier and N. Camps,
Long time stability for cubic nonlinear Schr\"odinger equations on non-rectangular flat tori.
J. \'Ec. polytech. Math. 12 (2025), 713--800.
\url{https://doi.org/10.5802/jep.300}

\bibitem[BCGW24]{BCGW24}
J. Bernier, N. Camps, B. Grébert, Z. Wang,
Exponential stability of solutions to the Schr\"odinger-Poisson equation.
Discrete Contin. Dyn. Syst. 44 (2024), no. 11, 3398--3442.
\url{https://doi.org/10.3934/dcds.2024064}

\bibitem[BFG20]{BFG20}
J. Bernier, E. Faou, B. Grébert,
Long time behavior of the solutions of NLW on the d-dimensional torus.
Forum Math. Sigma 8 (2020), Paper No. e12, 26 pp.

\bibitem[BFG20R]{BFG20RN} J. Bernier, E. Faou and B. Gr\'ebert,
Rational normal forms and stability of small solutions to nonlinear Schr\"odinger equations.
Ann. PDE 6 (2020), no. 2, Paper No. 14.
\url{https://doi.org/10.1007/s40818-020-00089-5}

\bibitem[BG21]{BG21} J. Bernier and B. Gr\'ebert,
Long time dynamics for generalized Korteweg--de Vries and Benjamin--Ono equations.
Arch. Ration. Mech. Anal. 241 (2021), no. 3, 1139--1241.
\url{https://doi.org/10.1007/s00205-021-01666-z}

\bibitem[BG25]{BG25} J. Bernier and B. Gr\'ebert,
Almost global existence for some nonlinear Schr\"odinger equations on $\mathbb T^d$ in low regularity.
Ann. Inst. Fourier 75 (2025), no. 5, 2189--2232.
\url{https://doi.org/10.5802/aif.3706}

\bibitem[BMP19]{BMP19}
L. Biasco, J. E. Massetti, M. Procesi,
Exponential and sub-exponential stability times for the NLS on the circle.
Atti Accad. Naz. Lincei Rend. Lincei Mat. Appl. 30 (2019), no. 2, 351--364.
\url{https://doi.org/10.4171/RLM/850}

\bibitem[BMP20]{BMP20} L. Biasco, J. E. Massetti, M. Procesi,
An abstract Birkhoff normal form theorem and exponential type stability of the 1d NLS.
Comm. Math. Phys. 375 (2020), no. 3, 2089--2153.
\url{https://doi.org/10.1007/s00220-019-03618-x}

\bibitem[B10]{B10} A. Bounemoura,
Nekhoroshev estimates for finitely differentiable quasi-convex Hamiltonians.
J. Differential Equations 249 (2010), no. 11, 2905--2920.
\url{https://doi.org/10.1016/j.jde.2010.06.004}

\bibitem[Bou11]{Bou11} A. Bounemoura,
Effective stability for Gevrey and finitely differentiable prevalent Hamiltonians.
Comm. Math. Phys. 307 (2011), no. 1, 157--183.
\url{https://doi.org/10.1007/s00220-011-1306-x}

\bibitem[BF21]{BF21} A. Bounemoura and J. F\'ejoz,
Hamiltonian perturbation theory for ultra-differentiable functions.
Mem. Amer. Math. Soc. 270 (2021), no. 1319.
\url{https://doi.org/10.1090/memo/1319}

\bibitem[B00]{B00}J. Bourgain, 
On diffusion in high-dimensional Hamiltonian systems and PDE.
J. Anal. Math. 80 (2000), 1--35.
\url{https://doi.org/10.1007/BF02791532}

\bibitem[BProb00]{BProb00} J. Bourgain,
Problems in Hamiltonian PDE's.
In \emph{GAFA 2000 (Tel Aviv, 1999)},
Geom. Funct. Anal., Special Volume, Part I (2000), 32--56.
\url{https://doi.org/10.1007/978-3-0346-0422-2_2}

\bibitem[B04]{B04} J. Bourgain, 
Remarks on stability and diffusion in high-dimensional Hamiltonian systems and partial differential equations.
Ergodic Theory Dynam. Systems 24 (2004), no. 5, 1331--1357.
\url{https://doi.org/10.1017/S0143385703000750}

\bibitem[B05]{B05}J. Bourgain, 
On invariant tori of full dimension for 1D periodic NLS.
J. Funct. Anal. 229 (2005), no. 1, 62--94.
\url{https://doi.org/10.1016/j.jfa.2004.10.019}

\bibitem[CKSTT10]{CKSTT10} J. Colliander, M. Keel, G. Staffilani, H. Takaoka and T. Tao,
Transfer of energy to high frequencies in the cubic defocusing nonlinear Schr\"odinger equation.
Invent. Math. 181 (2010), 39--113.
\url{https://doi.org/10.1007/s00222-010-0242-2}


\bibitem[FGL13]{FGL13} E. Faou, L. Gauckler, C. Lubich,
Sobolev stability of plane wave solutions to the cubic nonlinear Schr\"odinger equation on a torus.
Comm. Partial Differential Equations 38 (2013), no. 7, 1123--1140. \url{https://doi.org/10.1080/03605302.2013.785562}

\bibitem[FG13]{FG13} E. Faou, B. Grébert,
A Nekhoroshev-type theorem for the nonlinear Schrödinger equation on the torus.
Anal. PDE 6 (2013), no. 6, 1243--1262.
\url{https://doi.org/10.2140/apde.2013.6.1243}

\bibitem[FF22]{FF22} G. Farr\'e and B. Fayad,
Instabilities of invariant quasi-periodic tori.
J. Eur. Math. Soc. 24 (2022), 4363--4383.
\url{https://doi.org/10.4171/JEMS/1206}

\bibitem[FM23]{FM23} R. Feola,  J. E. Massetti,
Sub-exponential stability for the beam equation.
J. Differential Equations 356 (2023), 188--242.
\url{https://doi.org/10.1016/j.jde.2023.01.038}

\bibitem[FM24]{FM24} R. Feola, J. E. Massetti,
On the lifespan of solutions and control of high Sobolev norms for the completely resonant NLS on tori.
J. Funct. Anal. 287 (2024), no. 12, Paper No. 110648, 39 pp.
\url{https://doi.org/10.1016/j.jfa.2024.110648}

\bibitem[GG24]{GG24} F. Giuliani and M. Guardia,
Arnold diffusion in Hamiltonian systems on infinite lattices.
Comm. Pure Appl. Math. 77 (2024), no. 8, 3333--3426.
\url{https://doi.org/10.1002/cpa.22191}

\bibitem[GGMP21]{GGM21} F. Giuliani, M. Guardia, P. Martin and S. Pasquali,
Chaotic-like transfers of energy in Hamiltonian PDEs.
Comm. Math. Phys. 384 (2021), 1227--1290.
\url{https://doi.org/10.1007/s00220-021-03956-9}

\bibitem[MS03]{MS03} J.-P. Marco and D. Sauzin,
Stability and instability for Gevrey quasi-convex near-integrable Hamiltonian systems.
Publ. Math. Inst. Hautes \'Etudes Sci. 96 (2003), 199--275.
\url{https://doi.org/10.1007/s10240-003-0011-5}

\bibitem[N77]{N77} N. N. Nekhoroshev,
An exponential estimate of the time of stability of nearly integrable Hamiltonian systems.
Uspehi Mat. Nauk 32 (1977), no. 6(198), 5--66, 287.
\url{https://doi.org/10.1070/RM1977v032n06ABEH003859}

\bibitem[P18]{P18} S. Pasquali,
A Nekhoroshev type theorem for the nonlinear Klein-Gordon equation with potential.
Discrete Contin. Dyn. Syst. Ser. B 23 (2018), no. 9, 3573--3594. \url{https://doi.org/10.3934/dcdsb.2017215}

\bibitem[P93]{P93} J. P\"oschel,
Nekhoroshev estimates for quasi-convex Hamiltonian systems.
Math. Z. 213 (1993), no. 2, 187--216.
\url{https://doi.org/10.1007/BF03025718}

\bibitem[SLW23]{SLW23} Y. Sun,  S. Li, X. Wu,
Exponential and sub-exponential stability times for the derivative wave equation.
Discrete Contin. Dyn. Syst. 43 (2023), no. 6, 2186--2212. \url{https://doi.org/10.3934/dcds.2023007}

\bibitem[W15]{W15}B. Wilson,
Sobolev stability of plane wave solutions to the nonlinear Schr\"odinger equation.
Comm. Partial Differential Equations 40 (2015), no. 8, 1521--1542.
\url{https://doi.org/10.1080/03605302.2015.1030759}


\end{thebibliography}
\end{document}